\documentclass[onefignum,onetabnum,final]{siamart190516}
\usepackage[T1]{fontenc}

\usepackage{amsmath,amssymb,amsfonts,latexsym,stmaryrd}
\usepackage{graphicx,float} 
\usepackage{mathrsfs} 
\usepackage{color}
\usepackage{setspace} 
\usepackage{lipsum}
\usepackage{graphicx,float}
\usepackage{color}
\usepackage{epstopdf}
\usepackage{multirow}
\usepackage{url}
\usepackage{diagbox}
\def\O{\Omega}

\renewcommand\sp{\mathop{\mathrm{Sp}}\nolimits}

\newcommand{\set}[1]{\lbrace #1 \rbrace}

\newcommand{\norm}[1]{\lVert#1\rVert}

\usepackage{booktabs}
\usepackage{float}

\newcommand\bu{\boldsymbol{u}}
\newcommand\bv{\boldsymbol{v}}
\newcommand\bw{\boldsymbol{w}}

\newcommand\br{\boldsymbol{r}}
\newcommand\bs{\boldsymbol{s}}

\def\hdel{\widehat{\delta}}

\newcommand\bT{\boldsymbol{T}}

\newcommand\bQ{\boldsymbol{Q}}

\def\CT{{\mathcal T}}

\newcommand{\dd}{\texttt{d}}

\newcommand\bsig{\boldsymbol{\sigma}}
\newcommand\btau{\boldsymbol{\tau}}

\newcommand\R{\mathbb{R}}

\renewcommand\H{\mathrm{H}}

\renewcommand\L{\mathrm{L}}

\renewcommand\O{\Omega}

\newcommand\bdiv{\mathop{\mathbf{div}}\nolimits}

\renewcommand\div{\mathop{\mathrm{div}}\nolimits}

\newcommand\tr{\mathop{\mathrm{tr}}\nolimits}

\renewcommand\sp{\mathop{\mathrm{sp}}\nolimits}

\newcommand{\vertiii}[1]{{\left\vert\kern-0.25ex\left\vert\kern-0.25ex\left\vert #1 
    \right\vert\kern-0.25ex\right\vert\kern-0.25ex\right\vert}}

\newsiamremark{remark}{Remark}
\newsiamremark{hypothesis}{Hypothesis}
\crefname{hypothesis}{Hypothesis}{Hypotheses}
\newsiamthm{problem}{Problem}
\newsiamthm{claim}{Claim}

\headers{Velocity-pseudostress formulation of the Oseen eigenproblem}{Felipe Lepe,  Gonzalo Rivera and Jesus Vellojin}

\title{A Mixed finite element  method for the velocity-pseudostress formulation of the Oseen eigenvalue problem\thanks{Submitted to the editors DATE.
\funding{The first author was partially supported by
	DIUBB through project 2120173 GI/C Universidad del B\'io-B\'io (Chile).
	The second author was supported by Universidad de Los Lagos Regular R02/21  and ANID-Chile through FONDECYT project 1231619 (Chile).  The third author was partially supported by the National Agency for Research and Development, ANID-Chile through FONDECYT Postdoctorado project 3230302, and by project Centro de Modelamiento Matemático (CMM), FB210005, BASAL funds for centers of excellence.
}}}

\author{Felipe Lepe\thanks{GIMNAP-Departamento de Matem\'atica, Universidad del B\'io - B\'io, Casilla 5-C, Concepci\'on, Chile. \email{flepe@ubiobio.cl}.}
\and Gonzalo Rivera\thanks{Departamento de Ciencias Exactas,
	Universidad de Los Lagos, Casilla 933, Osorno, Chile. \email{gonzalo.rivera@ulagos.cl}.}
\and Jesus Vellojin\thanks{GIMNAP-Departamento de Matem\'atica, Universidad del B\'io - B\'io, Casilla 5-C, Concepci\'on, Chile. \email{jvellojin@ubiobio.cl}.}}

\usepackage{amsopn}

\ifpdf
\hypersetup{
  pdftitle={Finite Element Analysis of the Oseen eigenvalue problem},
  pdfauthor={Felipe Lepe, Gonzalo Rivera and Jesus Vellojin}
}
\fi

\def\CT{{\mathcal T}}

\begin{document}

\maketitle

\begin{abstract}
In this paper, we introduce and analyze a mixed formulation for the Oseen eigenvalue problem by introducing the pseudostress tensor as a new unknown, allowing us to eliminate the fluid pressure. The well-posedness of the solution operator is established using a fixed-point argument. For the numerical analysis, we use the tensorial versions of Raviart-Thomas and Brezzi-Douglas-Marini elements to approximate the pseudostress, and piecewise polynomials for the velocity. Convergence and a priori error estimates are derived based on compact operator theory. We present a series of numerical tests in two and three dimensions to confirm the theoretical findings.
\end{abstract}



\section{Introduction}\label{sec:intro}
Mixed formulations are an attractive technique for approximating solutions to partial differential equations, as they inherently allow for the introduction of additional variables with physical significance. A classic reference discussing the advantages of mixed methods is \cite{MR3097958}.

Eigenvalue problems associated with partial differential equations are well-known topic and their applications are documented in \cite{MR2652780}, which presents a comprehensive discussion on the analysis. However, the study of eigenvalue problems is continuously evolving in various contexts, including those arising in fluid mechanics.

Recently, in \cite{LEPE2024116959}, a finite element method for the Oseen eigenvalue problem was introduced, where the classic velocity-pressure formulation is analyzed. This paper presents, to the best of the authors' knowledge, the first effort to numerically analyze this spectral problem using inf-sup stable families of finite elements where the  analysis includes a priori and a posteriori error estimates. The article also reports computed eigenvalues and eigenfunctions in different domains and for various convective velocities, allowing for comparison with results obtained by other methods and formulations.

Inspired by the work of \cite{MR4307023} in the context of the Navier-Stokes equations, we continue our research program, now focusing on the Oseen eigenvalue problem. This variable, introduced in \cite{cai2010} has led to a large number of works associated to numerical methods to solve problems in continuum mechanics, particularly mixed formulations,  as \cite{MR3535625, MR4434148, MR4593742, MR3453481, MR4627698, MR4789346,MR4480275}. From these references, we observe that in flow problems, the pseudostress is dependent on both the gradient of the velocity and the pressure. Through certain algebraic manipulations, the pressure can be eliminated from the formulation, leading to a mixed problem that is expressed solely in terms of the velocity and the pseudostress. The introduction of the pseudostress tensor was initially motivated by its applications in source problems. this idea has been extended for elasticity problems. The introduction of the pseudostress tensor was initially motivated by its applications in source problems. Nevertheless, for eigenvalue problems, the incorporation of this unknown has been also  a suitable alternative to approximate the spectrum of some differential operators as is shown on  \cite{MR4570534,MR4077220,MR4430561,LRVSISC,MR4471016,MR4666864,MR3335223}, where classic mixed finite elements, virtual elements, and discontinuous Galerkin methods have been considered. These formulations have demonstrated accuracy in computing the respective spectra. The work of \cite{MR4307023} motivates us to continue exploring pseudostress formulations, now for the non-symmetric Oseen problem.

Let us focus on the particularities of the pseudostress formulation for the Oseen eigenvalue problem. For the mathematical analysis of  this formulation, particularly the well-posedness of the solution operator, we need to adopt a fixed-point strategy due to the presence of the convective term. This analysis requires certain conditions on the data, such as viscosity and convective velocity, to prove the necessary criteria for the fixed-point argument. Furthermore, the proposed solution operator is compact, leading to a spectral analysis based on the theory of \cite{MR1115235}. Given that the Oseen eigenvalue problem is non-selfadjoint, it is essential to include the adjoint problem in the analysis, incorporating the corresponding adjoint solution operators in both continuous and discrete versions for the spectral analysis. However, one drawback of utilizing the fixed point strategy is that it is not straightforward to obtain a desirable estimate of the dependency on the data for the source problem. This limitation extends naturally to the corresponding eigenfunctions. As a result, when proving convergence in norm, it becomes difficult to establish bounds that reflect the dependence on the source data. To circumvent this issue, an additional assumption is required, namely that the solutions of both the source and eigenvalue problems exhibit continuous dependence on the data of the source problem.

Since the proposed formulation is mixed, we use classic finite element families for this type of formulation, such as the Raviart-Thomas and Brezzi-Douglas-Marini elements, for the discretization. With these families, we prove the well-posedness of the discrete source problem to define the discrete solution operator correctly. With the aid of \cite{MR1115235}, we derive convergence and error estimates for the mixed scheme and achieve the desired spurious-free feature of the method.

The paper is organized as follows. In Section \ref{sec:model_problem}, we derive a velocity-pseudostress eigenvalue problem. The well-posedness of the source problem is studied by means of a fixed point argument. Due to the non-symmetric nature of the proposed problem, we also introduce the adjoint model along with its properties. Section \ref{sec:fem} is devoted to state the discrete version of the problem. The discrete scheme is approximated by the Raviart-Thomas and Brezzi-Douglas-Marini elements. Some approximation properties for these spaces are presented. We also report the well-posedness of the discrete primal and adjoint formulations, together with their corresponding convergence analysis an error estimates. We conclude with Section \ref{sec:numerics}, where several numerical experiments are performed to test the accuracy and robustness of the scheme.

\subsection{Notations and preliminaries}
Throughout this work, we will use notations that will allow a smoother reading of the content. Let us set these notations. Given $n \in\{2,3\}$, we denote $\mathbb{C}^{n \times n}$ as  the space of square real matrices of order $n$, where $\mathbb{I}:=\left(\delta_{i j}\right) \in \mathbb{C}^{n \times n}$ denotes the identity matrix. Given $\boldsymbol{\mathbb{A}}:=\left(A_{i j}\right), \boldsymbol{\mathbb{B}}:=\left(B_{i j}\right) \in \mathbb{C}^{n \times n}$, we define the following operations
$$
 \boldsymbol{\mathbb{A}}: \boldsymbol{\mathbb{B}}:=\sum_{i, j=1}^n A_{i j} \overline{B}_{i j}, \quad\tr\left(\mathbb{A}\right):=\sum_{i=1}^nA_{ii},\quad \mathbb{A}^{\dd}:=\mathbb{A}-\frac{1}{n}\left(\tr\left(\mathbb{A}\right)\right)\mathbb{I}.
$$
The entry $\overline{B}_{i j}$ represent the complex conjugate of $B_{i j}$. Similarly, given two vectors $\bs:=(s_i), \br:=(r_i) \in  \mathbb{C}^n$, we define the products
$$
\bs\cdot\br:= \sum_{i=1}^n s_i \overline{r}_i, \qquad \bs \otimes \br : = \bs \overline{\br}^\texttt{t}  = \sum_{i=1}^d\sum_{j=1}^n s_i \overline{r}_j,
$$
as the dot and dyadic product in $\mathbb{C}$, respectively, where $(\cdot)^\texttt{t}$ denotes the transpose operator. For simplicity, a generic null vector or matrix will be simply denoted by $\boldsymbol{0}$.

Let $\mathcal{O}$ be a subset of $\mathbb{R}^n$ with a Lipschitz boundary $\partial \mathcal{O}$. For $s \geq 0$ and $p \in [1, \infty]$, we denote by $\mathrm{L}^p(\mathcal{O})$ and $\mathrm{W}^{s, p}(\mathcal{O})$ the Lebesgue and Sobolev spaces of maps from $\mathcal{O}$ to $\mathbb{C}$, endowed with the norms $\|\cdot\|_{\mathrm{L}^p(\mathcal{O})}$ and $\|\cdot\|_{\mathrm{W}^{s, p}(\mathcal{O})}$, respectively, where $\mathrm{W}^{0, p}(\mathcal{O}) = \mathrm{L}^p(\mathcal{O})$.  If $p=2$, we write $\mathrm{H}^s(\mathcal{O})$ instead of $\mathrm{W}^{s, 2}(\mathcal{O})$, and denote the corresponding Lebesgue and Sobolev norms by $\|\cdot\|_{0, \mathcal{O}}$ and $\|\cdot\|_{s, \mathcal{O}}$, respectively. For $p=\infty$, we will denote $\Vert\cdot\Vert_{\infty,\mathcal{O}}$ as the induced norm over the space $\mathrm{W}^{1,\infty}(\mathcal{O})$. Bold letters will denote vector function spaces. For example, $\boldsymbol{\mathrm{L}}^p(\mathcal{O})$, $\boldsymbol{\mathrm{W}}^{s, p}(\mathcal{O})$ and $\boldsymbol{\mathrm{H}}^s(\mathcal{O})$ denote the vectorial version of $\mathrm{L}^p(\mathcal{O})$, $\mathrm{W}^{s, p}(\mathcal{O})$ and $\mathrm{H}^s(\mathcal{O})$, respectively. In the case of tensor spaces,  blackboard bold letters will be used. For example, $\mathbb{H}^s(\mathcal{O})$ is the tensorial space associated with $\mathrm{H}^s(\mathcal{O})$.
In this line, we define the Hilbert space 
$\mathbb{H}^{s}(\bdiv;\mathcal{O}):=\set{\btau\in\mathbb{H}^s(\mathcal{O}):\ \bdiv\btau\in\boldsymbol{\H}^s(\mathcal{O})}$, whose norm
is given by $\norm{\btau}^2_{\mathbb{H}^s(\bdiv;\mathcal{O})}
:=\norm{\btau}_{s,\mathcal{O}}^2+\norm{\bdiv\btau}^2_{s,\mathcal{O}}$. If $s=0$, we simply have $\mathbb{H}(\bdiv,\mathcal{O}):=\mathbb{H}^{0}(\bdiv;\mathcal{O})=\set{\btau\in\mathbb{L}^2(\mathcal{O}):\ \bdiv\btau\in\boldsymbol{\L}^2(\mathcal{O})}$, with norm $\norm{\btau}^2_{\bdiv,\mathcal{O}}
:=\norm{\btau}_{0,\mathcal{O}}^2+\norm{\bdiv\btau}^2_{0,\mathcal{O}}$.

\section{The model problem}
\label{sec:model_problem}

Let  $\O\subset\mathbb{R}^n$,  with $n\in\{2,3\}$,  be an open and bounded polygonal/polyhedral domain with Lipschitz boundary $\partial\O$. 
The problem of interest in this study is based on the Oseen equations, whose main characteristic is to model incompressible  viscous fluids at small Reynolds number. The corresponding strong form of the equations are given as:
\begin{equation}\label{def:oseen-eigenvalue}
\left\{
\begin{array}{rcll}
-\nu\Delta \bu + (\boldsymbol{\beta}\cdot\nabla)\bu + \nabla p&=&\lambda\bu,&\text{in}\,\O,\\
\div \bu&=&0,&\text{in}\,\O,\\
\displaystyle\int_{\O} p &=&0, &\text{in}\,\O,\\
\bu &=&\boldsymbol{0},&\text{in}\,\partial\O,
\end{array}
\right.
\end{equation}
where $\bu$ is the displacement, $p$ is the pressure and $\boldsymbol{\beta}$ is a given vector field, representing a \textit{steady flow velocity} such that $\boldsymbol{\beta}\in \boldsymbol{\mathrm{W}}^{1,\infty}(\O)$ is divergence free and $\nu>0$ is the kinematic viscosity. Over this last  parameter, we assume that there exists two positive  numbers $\nu^+$ and $\nu^{-}$ such that $\nu^{-}< \nu< \nu^{+}$. 

The standard assumptions on the coefficients are the following (see \cite{John2016})
\begin{itemize}
\item $\|\boldsymbol{\beta}\|_{\infty,\O}\sim 1$ if $\nu\leq \|\boldsymbol{\beta}\|_{\infty,\O}$,
\item $\nu\sim 1$ if $\|\boldsymbol{\beta}\|_{\infty,\O}<\nu$,
\end{itemize}
where the first point is the case more close to the real applications. 

Now we introduce the pseudostress tensor 
\begin{equation}\label{eq:psudostress}
\bsig=\nu\nabla\bu-( \bu\otimes\boldsymbol{\beta})-p\mathbb{I} \text{ in }\O.
\end{equation}
It is easy to check that $p=-\frac{1}{n}\left(\tr(\bsig)+\tr{(\bu\otimes\boldsymbol{\beta})}\right)$ in $\O$, which, along with \eqref{eq:psudostress}, leads us to the equation
\begin{equation}\label{eq:psudostressD}
\bsig^{\dd}=\nu\nabla\bu-(\bu\otimes\boldsymbol{\beta})^{\dd} \text{ in }\O.
\end{equation}

Note that given that $\div \boldsymbol{\beta}=0$ in $\O$, then it can be shown that $\bdiv(\bu\otimes\boldsymbol{\beta})=(\boldsymbol{\beta}\cdot\nabla)\bu.$  in $\O$

On the other hand, from the first equation of \eqref{def:oseen-eigenvalue} together with \eqref{eq:psudostress} we have $-\bdiv(\bsig)=\lambda\bu$ in $\O$. Finally, from \eqref{def:oseen-eigenvalue} we observe that the condition $(p, 1)_{0,\O} = 0$, is equivalent to $(\tr(\bsig),1)_{0,\O}=-(\tr{(\bu\otimes\boldsymbol{\beta})},1)_{0,\O}$.

Therefore, in accordance with the above, problem \eqref{def:oseen-eigenvalue} can be rewritten as follows 

\begin{equation}\label{def:oseen-eigenvalue_v-p}
\left\{
\begin{array}{rcll}
\bsig^{\dd}&=&\nu\nabla\bu-(\bu\otimes\boldsymbol{\beta})^{\dd},&\text{in}\,\O,\\
-\bdiv(\bsig)&=&\lambda\bu,&\text{in}\,\O,\\
(\tr(\bsig),1)_{0,\O}&=&-(\tr{(\bu\otimes\boldsymbol{\beta})},1)_{0,\O}, &\text{in}\,\O,\\
\bu &=&\boldsymbol{0},&\text{in}\,\partial\O.
\end{array}
\right.
\end{equation}
Let us define the spaces $\mathbb{H}:=\mathbb{H}(\bdiv,\O)$ and $\bQ= \boldsymbol{\L}^2(\O)$. 
We also define the norm $\vertiii{(\btau,\bv)}^2:=\|\btau\|_{\bdiv,\O}^2+\|\bv\|_{0,\O}^2$. A variational formulation for \eqref{def:oseen-eigenvalue_v-p} consists of: Find $\lambda\in\mathbb{C}$ and $(\boldsymbol{0},\boldsymbol{0})\neq(\bsig, \bu)\in \mathbb{H}\times\bQ$ such that  $(\tr(\bsig),1)_{0,\O}=-(\tr{(\bu\otimes\boldsymbol{\beta})},1)_{0,\O}$ and
\begin{equation}\label{def:oseen_system_weak_1}
	\left\{
	\begin{array}{rcll}
a(\bsig,\btau)+ b(\btau,\bu)+c(\bu,\btau)&=&0&\forall \btau\in \mathbb{H},\\

 b(\bsig, \bv)&=&-\lambda(\bu,\bv)_{0,\O}&\forall \bv\in \bQ,
\end{array}
	\right.
\end{equation}
where where the bilinear forms $a(\cdot,\cdot)$, $b(\cdot,\cdot)$ and $c(\cdot,\cdot)$, are defined by
\begin{align*}
a(\boldsymbol{\rho},\btau)&:=\dfrac{1}{\nu}\int_\O\boldsymbol{\rho}^{\dd}:\btau^{\dd},\qquad &\boldsymbol{\rho},\btau\in\mathbb{H},\\
b(\btau,\bv)&:=\int_\O\bdiv(\btau)\cdot\bv,\qquad &(\btau,\bv)\in \mathbb{H}\times\bQ,\\
c(\bv,\btau)&:=\dfrac{1}{\nu}\int_\O(\bv\otimes\boldsymbol{\beta})^{\dd}:\btau,\qquad &(\bv,\btau)\in \bQ\times\mathbb{H}.
\end{align*}

Now for our analysis it is necessary the following  decomposition $\mathbb{H}=\mathbb{H}_0\oplus \R\mathbb{I}$, where 
$$\mathbb{H}_0:=\left\{\btau\in \mathbb{H} \;\:\;\ \int_{\O}\tr(\btau)=0 \right\}.$$
More precisely, each $\btau\in\mathbb{H}$ can be decomposed uniquely as 
$$\btau=\btau_0+c\mathbb{I},\quad\text{with}\quad \btau_0\in \mathbb{H}_0\quad \text{and}\quad c:=\dfrac{1}{n|\O|}\int_{\O}\tr(\btau)\in\mathbb{R}.$$

Thus, if we define the tensor $\bsig_0:=\bsig+\displaystyle\dfrac{1}{n|\O|}\int_\O\tr(\bu\otimes\boldsymbol{\beta})\mathbb{I}$, it is clear that $\bsig_0\in\mathbb{H}_0$ and  it is not difficult to prove that Problem \eqref{def:oseen_system_weak_1}  can be equivalently rewritten in terms of $\bsig_0$ as: Find $\lambda\in\mathbb{C}$ and $(\boldsymbol{0},\boldsymbol{0})\neq(\bsig_0, \bu)\in \mathbb{H}_0\times\bQ$ such that 
$$
	\left\{
	\begin{array}{rcll}
a(\bsig_0,\btau)+ b(\btau,\bu)+c(\bu,\btau)&=&0&\forall \btau\in \mathbb{H}_0,\\

 b(\bsig_0, \bv)&=&-\lambda(\bu,\bv)_{0,\O}&\forall \bv\in \bQ.
\end{array}
	\right.
$$
Consequently, to simplify the analysis, instead of addressing Problem \eqref{def:oseen_system_weak_1} , we propose to analyze and discretize the following problem: Find $\lambda\in\mathbb{C}$ and $(\boldsymbol{0},\boldsymbol{0})\neq(\bsig, \bu)\in \mathbb{H}_0\times\bQ$ such that 
\begin{equation}\label{def:oseen_system_weak_1_H0}
	\left\{
	\begin{array}{rcll}
a(\bsig,\btau)+ b(\btau,\bu)+c(\bu,\btau)&=&0&\forall \btau\in \mathbb{H}_0,\\

 b(\bsig, \bv)&=&-\lambda(\bu,\bv)_{0,\O}&\forall \bv\in \bQ.
\end{array}
	\right.
\end{equation}
\begin{remark}
It is easy to check that $(\lambda,\bsig, \bu)\in\mathbb{C}\times\mathbb{H}_0\times\bQ$ is a solution of Problem \eqref{def:oseen_system_weak_1}  if and only if $(\lambda,\bsig, \bu)\in\mathbb{C}\times\mathbb{H}_0\times\bQ$ is a solution of Problem \eqref{def:oseen_system_weak_1_H0}. Also, note that the post-processing formula for the pressure $p$ now reduces  
$$p=-\dfrac{1}{n}\left(\tr(\bsig)+\tr{(\bu\otimes\boldsymbol{\beta})}-\dfrac{1}{|\O|}\int_\O\tr(\bu\otimes\boldsymbol{\beta})\right).$$
\end{remark}

To continue with our analysis, we invoke the following result (see \cite[Chapter 9, Proposition 9.1.1]{MR3097958}) 
\begin{lemma}\label{lmm:cota}
There exists $c_1>0$, such that
\begin{equation}
c_1\|\btau\|_{\bdiv,\O}^2\leq \|\btau^{\dd}\|_{0,\O}^2+\|\bdiv\btau\|_{0,\O}^2,\qquad\forall \btau\in \mathbb{H}_0.
\end{equation}
\end{lemma}

Let us define the kernel $\mathcal{K}$ of $b(\cdot,\cdot)$ as follows $
\mathcal{K}:=\{\bv\in\mathbb{H}_0\,:\,  b(\bv, \boldsymbol{q})=0\,\,\,\,\forall \boldsymbol{q}\in\bQ\}$.
Thanks to the above lemma  we have the coercivity of  $a(\cdot,\cdot)$ in $\mathcal{K}$. i.e
\begin{equation}\label{eq:K-coercive}
\dfrac{c_1}{\nu}\|\btau\|_{\bdiv,\O}^2\leq a(\btau,\btau),\qquad \forall\btau\in\mathcal{K}.
\end{equation}

It is easy to check that  $b(\cdot,\cdot)$ satisfies the following inf-sup condition
\begin{equation}
\label{eq:inf-sup_cont}
\displaystyle\sup_{\btau\in\mathbb{H}_0}\frac{b(\btau,\boldsymbol{q})}{\|\btau\|_{\bdiv,\O}}\geq\gamma\|\boldsymbol{q}\|_{0,\O}\quad\forall \boldsymbol{q}\in\bQ.
\end{equation}
Hence, we introduce the so-called solution operator, which we denote by $\bT$ and is defined as follows
\begin{equation}\label{eq:operador_solucion_u}
	\bT:\bQ\rightarrow \bQ,\qquad \boldsymbol{f}\mapsto \bT\boldsymbol{f}:=\widehat{\bu}, 
\end{equation}
where the pair  $(\widehat{\bsig}, \widehat{\bu})\in\mathbb{H}_0\times\bQ$ is the solution of the following  source problem
\begin{equation}\label{def:oseen_system_source}
	\left\{
	\begin{array}{rcll}
a(\widehat{\bsig},\btau)+ b(\btau,\widehat{\bu})+c(\widehat{\bu},\btau)&=&0&\forall \btau\in \mathbb{H}_0,\\

 b(\widehat{\bsig}, \bv)&=&-(\boldsymbol{f},\bv)_{0,\O}&\forall \bv\in \bQ.
\end{array}
	\right.
\end{equation}
Now our  aim is to determine that $\bT$ is well defined. To do this task, we employ a fixed point strategy.
\subsection{The fixed point argument}\label{sec:fixed-point-continuous} In this section we will define the solution operator associated to our spectral problem and analyze, through the fixed point theory, the good approach of the solution operator. Let us consider the following source  problem
\begin{equation}\label{def:oseen_system_source_00}
	\left\{
	\begin{array}{rcll}
a(\widehat{\bsig},\btau)+ b(\btau,\widehat{\bu})&=&-\displaystyle\frac{1}{\nu}\int_{\O}(\boldsymbol{w}\otimes\boldsymbol{\beta})^{\texttt{d}}:\btau&\forall \btau\in \mathbb{H}_0,\\
 b(\widehat{\bsig}, \bv)&=&-(\boldsymbol{f},\bv)_{0,\O}&\forall \bv\in \bQ.
\end{array}
	\right.
\end{equation}
Let $\mathcal{G}_{\boldsymbol{w},\boldsymbol{\beta}}$ and $\mathcal{F}$ be two functionals defined by
$$
\mathcal{G}_{\boldsymbol{w},\boldsymbol{\beta}}(\btau):=-\displaystyle\frac{1}{\nu}\int_{\O}(\boldsymbol{w}\otimes\boldsymbol{\beta})^{\texttt{d}}:\btau^{\texttt{d}}\,\,\,\forall\btau\in\mathbb{H}_0,\,\,\,\,\,\,\,\,\mathcal{F}(\bv):=-\int_{\O}\boldsymbol{f}\cdot\bv\,\,\,\forall\bv\in\bQ.
$$ and are such that
$$
\frac{|\mathcal{G}_{\boldsymbol{w},\boldsymbol{\beta}}(\btau)|}{\|\btau\|_{\bdiv,\O}}\leq\frac{1}{\nu}\|\boldsymbol{w}\|_{0,\O}\|\boldsymbol{\beta}\|_{\infty,\O}\,\,\,\,\text{and}\,\,\,\,\,\frac{|\mathcal{F}(\bv)|}{\|\bv\|_{0,\O}}\leq \|\boldsymbol{f}\|_{0,\O}.
$$
$$
|a(\bsig,\btau)|\leq\dfrac{1}{\nu}\|\bsig\|_{\bdiv,\O}\|\btau\|_{\bdiv,\O}\,\,\,\,\text{and}\,\,\,\,\,|b(\bsig,\bv)|\leq\|\bsig\|_{\bdiv,\O}\|\bv\|_{0,\O}.
$$
Then, thanks to the above estimates, to the coercivity of $a(\cdot,\cdot)$ in $\mathcal{K}$, to the inf-sup property, we have from the  Bab\v{u}ska-Brezzi theory (see, for example, \cite[Theorem 2.34]{MR2050138}) that there is guaranteed to exist a unique pair $(\widehat{\bsig} ,\widehat{\bu})$ solution of \eqref{def:oseen_system_source_00}, which also satisfies the continuous dependence on the data. i.e
\begin{align}\label{eq:dependece_date_1}
\|\widehat{\bsig}\|_{\bdiv,\O}\leq \dfrac{\nu}{c_1}\|\mathcal{G}_{\boldsymbol{w},\boldsymbol{\beta}}\|_{\mathbb{H}(\bdiv,\O)'}+\dfrac{1}{\gamma}\left(1+\dfrac{\nu}{c_1}\|a\|\right)\|\mathcal{F}\|_{0,\O},\\\label{eq:dependece_date_2}
\|\widehat{\bu}\|_{0,\O}\leq\dfrac{1}{\gamma}\left(1+\dfrac{\nu}{c_1}\|a\|\right)\|\mathcal{G}_{\boldsymbol{w},\boldsymbol{\beta}}\|_{\mathbb{H}(\bdiv,\O)'}+\dfrac{\|a\|}{\gamma^2}\left(1+\dfrac{\nu}{c_1}\|a\|\right)\|\mathcal{F}\|_{0,\O},
\end{align}
where $\|a\|:=\min\left\{1,\dfrac{1}{\nu}\right\}$.

 To continue with the analysis, let us introduce the following space $M_{R_0}:=\{\bv\in\bQ\,:\,\|\bv\|_{0,\O}\leq R_0\}$. Also, let us define the following operator
 \begin{align*}
 \mathcal{J}:&\bQ\rightarrow\bQ\times \mathbb{H}_0,\\
 &\bw\mapsto\mathcal{J}(\bw):=(\mathcal{J}_1(\bw),\mathcal{J}_2(\bw))=(\bu,\boldsymbol{\sigma}).
 \end{align*}
 Observe that fixing $R_0> 0$ and let us consider $\bw\in M_{R_0}$. We have, thanks to the definition
of $J_1$, the triangular inequality, and \eqref{eq:dependece_date_2}, that
 \begin{multline*}
 \|\mathcal{J}_1(\bw)\|_{0,\O}=\|\bu\|_{0,\O} \leq \dfrac{1}{\gamma}\left(1+\dfrac{\nu}{c_1}\|a\|\right)\|\mathcal{G}_{\bw,\boldsymbol{\beta}}\|_{\mathbb{H}(\bdiv,\O)'}+ \dfrac{\|a\|}{\gamma^2}(1+\dfrac{\nu}{c_1}\|a\|)\|\mathcal{F}\|_{0,\O}\\
 \leq\dfrac{1}{\nu\gamma}\left(1+\dfrac{\nu}{c_1}\|a\|\right)(\|\bw\|_{0,\O}\|\boldsymbol{\beta}\|_{\infty,\O})+\dfrac{\|a\|}{\gamma^2}(1+\dfrac{\nu}{c_1}\|a\|)\|\boldsymbol{f}\|_{0,\O}.
 \end{multline*}
Now, we define $H:=\left(1+\dfrac{\nu}{c_1}\|a\|\right)$, $C_1:=\dfrac{1}{2\nu\gamma}$ and $C_2:=\dfrac{C_1}{2\nu\gamma}\|\boldsymbol{\beta}\|_{\infty,\O}^2+\dfrac{\|a\|}{\gamma^2}\|\boldsymbol{f}\|_{0,\O}$.
 After elementary algebraic manipulations, we can assert that the right-hand side above is smaller or
equal than $R_0$ if
 \begin{equation}\label{eq:R_0}
 \displaystyle \dfrac{HC_1}{2}R_0^2-R_0+HC_2\leq 0,
 \end{equation}
 where $R_0$ and $C_2$ are such that $$0<R_0<\dfrac{1-\sqrt{1-2H^2C_1C_2}}{HC_1},\qquad C_2< \dfrac{1}{2H^2C_1}.$$
Then we have that $\mathcal{J}_1(M_{R_0})\subset M_{R_0}$.

The following results proves that $\mathcal{J}_1$ is a contraction.
\begin{lemma}\label{lmm:cotaj}
There exists a positive constant $L$ , depending only on data, such that 
$$\|\mathcal{J}_1(\bw_1)-\mathcal{J}_1(\bw_2)\|_{0,\O}\leq L\|\bw_1-\bw_2\|_{0,\O},\qquad \forall\bw_1,\bw_2\in M_{R_0}.$$ 
\end{lemma}
\begin{proof}
Let $\bw_1,\bw_2\in M_{R_0}$, such that $\mathcal{J}(\bw_1):=(\bu_1,\bsig_1)$ , $\mathcal{J}(\bw_2)=(\bu_2,\bsig_2)$ solutions to the following problems:
$$
	\left\{
	\begin{array}{rcll}
a(\bsig_1,\btau)+ b(\btau,\bu_1)&=&-\displaystyle\frac{1}{\nu}\int_{\O}(\boldsymbol{w}_1\otimes\boldsymbol{\beta})^{\texttt{d}}:\btau&\forall \btau\in \mathbb{H}_0,\\
 b(\bsig_1, \bv)&=&-(\boldsymbol{f},\bv)_{0,\O}&\forall \bv\in \bQ.
\end{array}
	\right.
$$
$$
	\left\{
	\begin{array}{rcll}
a(\bsig_2,\btau)+ b(\btau,\bu_2)&=&-\displaystyle\frac{1}{\nu}\int_{\O}(\boldsymbol{w}_2\otimes\boldsymbol{\beta})^{\texttt{d}}:\btau&\forall \btau\in \mathbb{H}_0,\\
 b(\bsig_2, \bv)&=&-(\boldsymbol{f},\bv)_{0,\O}&\forall \bv\in \bQ.
\end{array}
	\right.
$$
Therefore, subtracting the second problem from the first, it follows that 
\begin{equation}\label{eq:intermedia}
	\left\{
	\begin{array}{rcll}
a(\bsig_1-\bsig_2,\btau)+ b(\btau,\bu_1-\bu_2)&=&-\displaystyle\frac{1}{\nu}\int_{\O}(\bw_1\otimes\boldsymbol{\beta})^{\dd}:\btau+\displaystyle\frac{1}{\nu}\int_{\O}(\bw_2\otimes\boldsymbol{\beta})^{\dd}:\btau,\\
 b(\bsig_1-\bsig_2, \bv)&=&0.
\end{array}
	\right.
\end{equation}
We note that from the above equation, if we take $\btau=\bsig_1-\bsig_2$ and $\bv=\bu_1-\bu_2$ we have that 
$$a(\bsig_1-\bsig_2,\bsig_1-\bsig_2)=\displaystyle\frac{1}{\nu}\int_{\O}(\bw_1\otimes\boldsymbol{\beta})^{\dd}:(\bsig_1-\bsig_2)-\displaystyle\frac{1}{\nu}\int_{\O}(\bw_2\otimes\boldsymbol{\beta})^{\dd}:(\bsig_1-\bsig_2).$$
Now using \eqref{eq:K-coercive}, we obtain 
$$
\dfrac{c_1}{\nu}\|\bsig_1-\bsig_2\|_{\bdiv,\O}^2\leq a(\bsig_1-\bsig_2,\bsig_1-\bsig_2)\leq\dfrac{1}{\nu}||\bw_1-\bw_2\|_{0,\O}\|\bsig_1-\bsig_2\|_{\bdiv,\O}\|\boldsymbol{\beta}\|_{\infty,\O}. 
$$
Then
\begin{equation}
\|\bsig_1-\bsig_2\|_{\bdiv,\O}\leq \dfrac{1}{c_1}||\bw_1-\bw_2\|_{0,\O}\|\boldsymbol{\beta}\|_{\infty,\O}. 
\end{equation}
On the other hand, using the inf-sup condition \eqref{eq:inf-sup_cont} we have
\begin{multline*}
\gamma\|\mathcal{J}_1(\bw_1)-\mathcal{J}_1(\bw_2)\|_{0,\O}=\gamma\|\bu_1-\bu_2\|_{0,\O}\leq \sup_{\btau\in\mathbb{H}_0}\frac{b(\btau,\bu_1-\bu_2)}{\|\btau_h\|_{\bdiv,\O}}\\
\leq \sup_{\btau\in\mathbb{H}_0}\frac{\left|\mathcal{G}_{\boldsymbol{\bw_1},\boldsymbol{\beta}}(\btau)-\mathcal{G}_{\boldsymbol{\bw_2},\boldsymbol{\beta}}(\btau)-a(\bsig_1-\bsig_2,\btau)\right|}{\|\btau_h\|_{\bdiv,\O}}\\
	\leq \sup_{\btau\in\mathbb{H}_0}\frac{\left|\mathcal{G}_{\boldsymbol{\bw_1},\boldsymbol{\beta}}(\btau)-\mathcal{G}_{\boldsymbol{\bw_2},\boldsymbol{\beta}}(\btau)\right|}{\|\btau_h\|_{\bdiv,\O}}+\sup_{\btau\in\mathbb{H}_0}\frac{\left|a(\bsig_1-\bsig_2,\btau)\right|}{\|\btau_h\|_{\bdiv,\O}}\\ \leq \dfrac{1}{\nu}\|\bw_1-\bw_2\|_{0,\O}\|\boldsymbol{\beta}\|_{\infty,\O}+\|\bsig_1-\bsig_2\|_{\bdiv,\O}\\
\leq  \dfrac{1}{\nu}\|\bw_1-\bw_2\|_{0,\O}\|\boldsymbol{\beta}\|_{\infty,\O}+ \dfrac{1}{c_1}\|\bw_1-\bw_2\|_{0,\O}\|\boldsymbol{\beta}\|_{\infty,\O} \\
\leq \|\boldsymbol{\beta}\|_{\infty,\O}\max\left\{\dfrac{1}{\nu},\dfrac{1}{c_1}\right\}||\bw_1-\bw_2\|_{0,\O}.
\end{multline*}
Therefore
$$\|\mathcal{J}_1(\bw_1)-\mathcal{J}_1(\bw_2)\|_{0,\O}\leq \dfrac{\|\boldsymbol{\beta}\|_{\infty,\O}}{\gamma}\max\left\{\dfrac{ 1}{\nu},\dfrac{1}{c_1}\right\}\|\bw_1-\bw_2\|_{0,\O}.$$
Then the desired result follows  choosing
$L:=\dfrac{\|\boldsymbol{\beta}\|_{\infty,\O}}{\gamma}\max\left\{\dfrac{ 1}{\nu},\dfrac{1}{c_1}\right\}.$
\end{proof}
\begin{theorem}
Given $\boldsymbol{f}\in\bQ$, assume  that the data is sufficiently small as in \eqref{eq:R_0}, and , further assume that 
\begin{equation}\label{eq_contraction}\|\boldsymbol{\beta}\|_{\infty,\O}\max\left\{\dfrac{ 1}{\nu},\dfrac{1}{c_1}\right\}<\gamma.
\end{equation}
Then there exists a unique solution of problem \eqref{def:oseen_system_source}.
\end{theorem}
\begin{proof}
The result is a direct consequence of the well definition of $\mathcal{J}$  together with  $\mathcal{J}_1(M_{R_0})\subset M_{R_0}$, Lemma \ref{lmm:cotaj}, and the fact that \eqref{eq_contraction} gives that $\mathcal{J}$ is a contraction mapping.
\end{proof}

We observe that $(\lambda , (\bsig, \bu )) \in  \mathbb{C}  \times \mathcal{X}$  is solution of \eqref{def:oseen_system_weak_1} if and only if $(\kappa , \bu )$ is an eigenpair of $\bT , i.e., \bT \bu  = \kappa \bu$  with $\kappa  := 1/\lambda$.

\begin{remark}\label{remark-del-uhat-sigmahat}
It is known that the additional regularity associated with the solutions of problem \ref{def:oseen-eigenvalue} is $\bu\in \boldsymbol{\H}^{1+s}(\Omega)$ and $p\in \H^s(\Omega)$ with $s>0$ (see for instance  \cite{LEPE2024116959}). therefore, observe this regularity, together with the first and second equations of \eqref{def:oseen-eigenvalue_v-p} allow us to conclude that  $\bdiv\bsig\in \boldsymbol{\H}^{1+s}(\O)$ and $\bsig\in \mathbb{H}^{s}(\O)$, respectively. This additional regularity for the pseudotress tensor is a key ingredient for the numerical approximation, letting us to conclude that operator $\bT$ is compact. Moreover, the spectrum of $\bT$ satisfies $\sp(\bT)=\{0\}\cup\{\mu_k\}_{k\in\mathbb{N}}$, where $\{\mu_k\}_{k\in\mathbb{N}}\in (0,1)$ is a sequence of complex  eigenvalues which converges to zero, repeated according their respective multiplicities.
\end{remark}
From now and on, the following assumption is made:
 Given $\boldsymbol{f}\in\bQ$, if the data are sufficiently small, as specified in  \eqref{eq:R_0}, and condition \eqref{eq_contraction} is satisfied, then it follows that 
\begin{equation}
	\label{eq:reg_u_if}
	\|\widehat{\bsig}\|_{s,\O}+\|\widehat{\bu}\|_{1+s,\Omega}\leq C\|\boldsymbol{f}\|_{0,\Omega},
\end{equation}
where $C>0$.
%

%
%

\subsection{The adjoint eigenvalue problem}
With the aim of proving convergence and error estimates of the non-selfadjoint problem \eqref{def:oseen_system_weak_1}, it is necessary to invoke its associated adjoint problem. 
The strong form of the dual equations written in terms of the velocity and pressure  are given by
\begin{equation}\label{def:oseen-eigenvalue-dual}
	\left\{
	\begin{array}{rcll}
		-\nu\Delta \bu^* - \div(\bu^*\otimes\boldsymbol{\beta}) - \nabla p^*&=&\lambda^*\bu^*,&\text{in}\,\O,\\
		-\div \bu^*&=&0,&\text{in}\,\O,\\
		\displaystyle\int_{\O} p^* &=&0, &\text{in}\,\O,\\
		\bu^* &=&\boldsymbol{0},&\text{in}\,\partial\O.
	\end{array}
	\right.
\end{equation}
Now, the pseudostress formulation of \eqref{def:oseen-eigenvalue-dual} is
\begin{equation}\label{def:oseen-eigenvalue-dual-pseudo} 
	\left\{
	\begin{array}{rcll}
		\boldsymbol{\sigma}^{*,\texttt{d}}-(\bu^*\otimes\boldsymbol{\beta})^{\texttt{d}}&=&\nu\nabla\bu^*,&\text{in}\,\O,\\
		-\div \boldsymbol{\sigma}^*&=&\lambda^*\bu^*,&\text{in}\,\O,\\
		(\tr(\boldsymbol{\sigma}^*),1)_{0,\O} &=&(\tr(\bu^*\otimes\boldsymbol{\beta}),1)_{0,\O}, &\text{in}\,\O,\\
		\bu^* &=&\boldsymbol{0},&\text{in}\,\partial\O.
	\end{array}
	\right.
\end{equation}

Now, proceeding as in the primal case, we introduce the decomposition of 
$\mathbb{H}=\mathbb{H}_0\oplus \R\mathbb{I}$. The dual weak variational formulation, which is equivalent to the one obtained from Problem  \eqref{def:oseen-eigenvalue-dual-pseudo} is given as follows:
Find $\lambda^*\in\mathbb{C}$ and the pair $(\boldsymbol{0},\boldsymbol{0})\neq(\boldsymbol{\sigma}^*,\bu^*)\in\mathbb{H}_0\times\bQ$ such that  
\begin{equation}\label{def:oseen_system_weak_dual_eigen}
	\left\{
	\begin{array}{rcll}
\displaystyle\int_{\O}\boldsymbol{\sigma}^{*,\texttt{d}}:\boldsymbol{\tau}^{\texttt{d}}-\int_{\O}(\bu^*\otimes\boldsymbol{\beta})^{\texttt{d}}:\boldsymbol{\tau}+\int_{\O}\bdiv\boldsymbol{\tau}\cdot\bu^*&=&0&\forall \btau\in \mathbb{H}_{0},\\
\displaystyle\int_{\O}\bdiv\boldsymbol{\sigma}^*\cdot\bv&=&\displaystyle-\lambda^*\int_{\O}\bv\cdot\bu^*&\forall \bv\in \bQ,
\end{array}
	\right.
\end{equation}
which we can write in an abstract manner as
\begin{equation}\label{def:oseen_system_weak_dual_source1}
	\left\{
	\begin{array}{rcll}
a(\btau,\bsig^*)+b(\btau,\bu^*)-c(\btau,\bu^*)&=&0&\forall \btau\in \mathbb{H}_0,\\
b(\bsig^*,\bv)&=&-\lambda^*(\bv,\bu^*)&\forall \bv\in \bQ.
\end{array}
	\right.
\end{equation}
Now we introduce the adjoint of \eqref{eq:operador_solucion_u} defined  by 
\begin{equation}\label{eq:operador_adjunto_solucion_u}
	\bT^*:\bQ\rightarrow \bQ,\qquad \boldsymbol{f}\mapsto \bT\boldsymbol{f}:=\widehat{\bu}^*, 
\end{equation} 
where $\widehat{\bu}^*\in\L^2(\O,\mathbb{C})$ is the adjoint velocity of $\widehat{\bu}$ and solves the following adjoint source  problem: Find  
$(\widehat{\bsig}^*, \widehat{\bu}^*)\in\mathbb{H}_0\times\bQ$ such that 
\begin{equation}\label{def:oseen_system_weak_dual_source11}
	\left\{
	\begin{array}{rcll}
a(\btau,\widehat{\bsig}^*)+b(\btau,\widehat{\bu}^*)-c(\btau,\widehat{\bu}^*)&=&0&\forall \btau\in \mathbb{H}_0,\\
b(\widehat{\bsig}^*,\bv)&=&-(\bv,\boldsymbol{f})&\forall \bv\in \bQ.
\end{array}
	\right.
\end{equation}
To prove that $\bT^*$ is well defined, we can follow the same arguments utilized for the well posedness of $\bT$. So, in order to simplify the presentation of our results, we skip the details.

Similar to Remark \ref{remark-del-uhat-sigmahat}, we have that the dual source and eigenvalue problems are such that the following regularity  holds: $\bsig^*\in \mathbb{H}^{s^*}(\O)$ and $\bu^*\in  \boldsymbol{\H}^{1+s^*}(\Omega)$ with $s^*>0$.

 As with the primal problem, we will assume the following:
 Given $\boldsymbol{f}\in\bQ$ sufficiently small, then there exists $C>0$, such that 
\begin{equation}
	\label{eq:reg_u_if*}
	\|\widehat{\bsig}^*\|_{s^{*},\O}+\|\widehat{\bu}^*\|_{1+s^{*},\Omega}\leq C\|\boldsymbol{f}\|_{0,\Omega}.
\end{equation}

The spectral characterization result for the $\bT^*$ operator is presented below.
\begin{lemma}(Spectral Characterization of $\bT^*$).
The spectrum of $\bT^*$ is such that $\sp(\bT^*)=\{0\}\cup\{\kappa_{k}^*\}_{k\in{N}}$ where $\{\kappa_{k}^*\}_{k\in\mathbf{N}}$ is a sequence of complex eigenvalues that converge to zero, according to their respective multiplicities. 
\end{lemma}
It is easy to prove that if $\kappa$ is an eigenvalue of $\bT$ with multiplicity $m$, $\overline{\kappa^*}$ is an eigenvalue of $\bT^*$ with the same multiplicity $m$.
\section{The finite element method}
\label{sec:fem}

Now we present the finite element spaces  involved in our numerical methods. We remark that our methods differ when the pseuodstress tensor is approximated, since the velocity will be approximated  with  basic piecewise polynomial spaces. To make matter precise,  from now and on we 
only refer to $\mathbb{H}_h$ to the finite element space related to the approximation of $\bsig$ whereas  for the velocity field,
we consider the space $\bQ_h$. Let us define the aforementioned spaces.  For $k\geq 0$, we define the vector field of piecewise polynomials of degree at most $k$
$$
	\mathbf{P}_k(\CT_h):=\{v\in\mathbf{L}^2(\O)\,:\, v|_T\in\text{P}_k(T)\,\,\forall T\in\CT_h\},
$$

We define the local Raviart-Thomas space of order $k$ as follows  (see \cite{MR3097958}) $\mathbf{RT}_k(T)=\mathbf{P}_k(T)\oplus \text{P}_k(T)\boldsymbol{x}$,
where if $\texttt{t}$ denotes the transpose operator,  $\boldsymbol{x}^{\texttt{t}}$ represents a generic vector of $\mathbb{R}^n$. Hence, the global Raviart-Thomas (RT) is defined by
$$
	\mathbb{RT}_k(\CT_h):=\{\btau\in\mathbb{H} \,:\, \btau|_T^{\texttt{t}}\in\mathbf{RT}_k(T),\,\,\forall T\in\CT_h\}.
$$

In the definition above $\btau|_T^{\texttt{t}}$ must be understood as $(\tau_{i1},\tau_{i2})^{\texttt{t}}\in\mathbf{RT}_k(T)$ for all $i\in\{1,2\}$ when $n=2$, and $(\tau_{j1},\tau_{j2}, \tau_{j3})^{\texttt{t}}\in\mathbf{RT}_k(T)$ for all $j\in\{1,2,3\}$ when $n=3$.

together with the  Brezzi-Douglas-Marini (BDM) finite element space \cite{MR799685},
$$
	\mathbb{BDM}_{k}:=\mathbf{P}_k(\CT_h)\cap\mathbb{H} \text{ with } k\geq 1.
$$

Now we introduce the following space
$$
	\mathbb{H}_{0,h}:=\left\{\btau\in\mathbb{H}_h\,:\,\,\int_{\Omega}\tr(\btau)=0 \right\},
$$
where $\mathbb{H}_h\in\{\mathbb{RT}_k,\mathbb{BDM}_{k+1}\}$. Also, we define  $\bQ_h:=\mathbf{P}_k(\CT_h)$.

\subsection{The discrete eigenvalue problem}\label{subsec:the-discrete-eigenvalue-problem}
With the finite elements described above, now we introduce the discrete versions of \eqref{def:oseen_system_weak_1_H0} and 
\eqref{def:oseen_system_weak_dual_source1}.

The discretization of \eqref{def:oseen_system_weak_1_H0} reads as follows:  Find $\lambda_h\in\mathbb{C}$ and $(\boldsymbol{0},0)\neq(\bsig_h, \bu_h)\in \mathbb{H}_{0,h}\times\bQ_h$ such that 
\begin{equation}\label{def:oseen_system_weak_1_H0_disc}
	\left\{
	\begin{array}{rcll}
a(\bsig_h,\btau_h)+ b(\btau_h,\bu_h)+c(\bu_h,\btau_h)&=&0&\forall \btau_h\in \mathbb{H}_{0,h},\\

 b(\bsig_h, \bv_h)&=&-\lambda_h(\bu_h,\bv_h)_{0,\O}&\forall \bv_h\in \bQ_h.
\end{array}
	\right.
\end{equation}
Now we introduce the discrete version of the  solution operator $\bT$ which we denote by $\bT_h$
\begin{equation}\label{eq:operador_solucion_uh}
	\bT_h:\bQ\rightarrow \bQ_h,\qquad \boldsymbol{f}\mapsto \bT_h\boldsymbol{f}:=\widehat{\bu}_h, 
\end{equation}
where the pair  $(\widehat{\bsig}_h, \widehat{\bu}_h)\in\mathbb{H}_{0,h}\times\bQ_h$ solves  the following  discrete  source problem
\begin{equation}\label{def:oseen_system_source_disc}
	\left\{
	\begin{array}{rcll}
a(\widehat{\bsig}_h,\btau_h)+ b(\btau_h,\widehat{\bu}_h)+c(\widehat{\bu}_h,\btau_h)&=&0&\forall \btau_h\in \mathbb{H}_{0,h},\\

 b(\widehat{\bsig}_h, \bv_h)&=&-(\boldsymbol{f},\bv_h)_{0,\O}&\forall \bv_h\in \bQ_h.
\end{array}
	\right.
\end{equation}

Let $\mathcal{K}_h$ be the discrete kernel of bilinear form $b(\cdot,\cdot)$, we have that as in the continuous case, the bilinear form $a(\cdot,\cdot)$ is coercive, i.e.
\begin{equation}
\dfrac{c_1}{\nu}\|\btau_h\|_{\bdiv,\O}^2\leq a(\btau_h,\btau_h),\qquad\forall\btau\in \mathcal{K}_h.
\end{equation}
On the other hand, the following discrete inf–sup condition holds.
\begin{equation}
\label{eq:inf-sup_discret}
\displaystyle\sup_{\btau_h\in\mathbb{H}_{0,h}}\frac{b(\btau_h,\boldsymbol{q}_h)}{\|\btau_h\|_{\bdiv,\O}}\geq\widehat{\gamma}\|\boldsymbol{q}_h\|_{0,\O}\quad\forall \boldsymbol{q}_h\in\bQ_h.
\end{equation}
The following bilinear form, which does not contain the convective term, is defined as follows: 
\begin{equation}
\label{eq:formB}
B((\bsig_h,\bu_h);(\btau_h,\bv_h)):=a(\bsig_h,\btau_h)+ b(\btau_h,\bu_h)+b(\bsig_h,\bv_h).
\end{equation}
In the following result, we establish an inf-sup condition for $B((\cdot,\cdot);(\cdot,\cdot))$, which will be instrumental to ensure the uniqueness and convergence of the discrete solution.
\begin{lemma}\label{inf-sup_B}
For each $(\bsig_h,\bu_h)\in \mathbb{H}_{0,h}\times \bQ_h$, there exists $(\btau_h,\bv_h)\in \mathbb{H}_{0,h}\times \bQ_h$ with 
$\vertiii{(\btau_h,\bv_h)} \leq C\vertiii{(\bsig_h,\bu_h)},$ and a constant $C_J>0$, depending on $\nu$,  such that
\begin{equation}
\vertiii{(\bsig_h,\bu_h)}^2\leq C_J B((\bsig_h,\bu_h);(\btau_h,\bv_h)).
\end{equation}
\end{lemma}
\begin{proof}
The arguments for the proof  we exploit the discrete inf-sup condition of $b(\cdot,\cdot)$ and the ellipticity of $a(\cdot,\cdot)$, together with appropriate test functions.

Let us start by taking $(\btau_h,\bv_h)=(\bsig_h,-\bu_h)$ to obtain
\begin{equation}\label{eq:lema-inf-sup-000}
	B((\bsig_h,\bu_h),(\bsig_h,-\bu_h))=a(\bsig_h,\bsig_h)\geq \frac{c_1}{\nu}\Vert\bsig_h\Vert_{\bdiv,\Omega}^2.
\end{equation}
From the inf-sup condition \eqref{eq:inf-sup_discret}, given $\bu_h\in\bQ_h$, there exists $\widetilde{\bsig}_h\in\mathbb{H}_{0,h}$ such that 
\begin{equation}\label{eq:lema-inf-sup-001}
\bdiv \widetilde{\bsig}_h=-\bu_h,\quad\mbox{and}\quad \Vert \widetilde{\bsig}_h\Vert_{\bdiv,\Omega}\leq \widehat{\gamma}^{-1}\Vert \bu_h\Vert_{0,\Omega}.
\end{equation}
By taking $(\btau_h,\bv_h)=(-\widetilde{\bsig}_h,\boldsymbol{0})$, using Cauchy–Schwarz inequality, the continuity of $a(\cdot,\cdot)$ and \eqref{eq:lema-inf-sup-001} we obtain
\begin{equation}
\label{eq:lema-inf-sup-002}
\begin{aligned}
	B((\bsig_h,\bu_h),(-\widetilde{\bsig}_h,\boldsymbol{0}))&= -a(\bsig_h,\widetilde{\bsig}_h) - b(\widetilde{\bsig}_h,\bu_h)\\
	&\geq -\frac{1}{\nu}\left(\Vert\bsig_h\Vert_{\bdiv,\Omega}\Vert\widetilde{\bsig}_h\Vert_{\bdiv,\Omega}\right) + \Vert \bu_h\Vert_{0,\Omega}^2\\
	&\geq -\frac{1}{\nu}\left(\frac{\varepsilon}{2}\Vert\bsig_h\Vert_{\bdiv,\Omega}^2 + \frac{\widehat{\gamma}^{-2}}{2\varepsilon}\Vert\bu_h\Vert_{0,\Omega}^2\right) + \Vert \bu_h\Vert_{0,\Omega}^2\\
	&= -\frac{\varepsilon}{2\nu}\Vert\bsig_h\Vert_{\bdiv,\Omega}^2 + \left(1- \frac{\widehat{\gamma}^{-2}}{2\nu\varepsilon}\right)\Vert \bu_h\Vert_{0,\Omega}^2,
\end{aligned}
\end{equation}
for all $\varepsilon>0$.
Let $\delta>0$ and consider the test functions $(\btau_h,\bv_h)=(\bsig_h-\delta\widetilde{\bsig}_h,-\bu_h)$. Then, from \eqref{eq:lema-inf-sup-000}-\eqref{eq:lema-inf-sup-001} and \eqref{eq:lema-inf-sup-002} we have
$$
\begin{aligned}
B((\bsig_h,\bu_h),(\bsig_h-\delta\widetilde{\bsig}_h,-\bu_h&)) = a(\bsig_h,\bsig_h-\delta\widetilde{\bsig}_h) + b(\bsig_h-\delta\widetilde{\bsig}_h,\bu_h) + b(\bsig_h,-\bu_h)\\
&\geq \frac{c_1}{\nu}\Vert \bsig_h\Vert_{\bdiv,\Omega}^2 + \delta\left[-a(\bsig_h,\widetilde{\bsig}_h) - b(\widetilde{\bsig}_h,\bu_h)\right]\\
&\geq \frac{c_1}{\nu}\Vert \bsig_h\Vert_{\bdiv,\Omega}^2 - \frac{\delta\varepsilon}{2\nu}\Vert\bsig_h\Vert_{\bdiv,\Omega}^2  + \delta\left(1- \frac{\widehat{\gamma}^{-2}}{2\nu\varepsilon}\right)\Vert \bu_h\Vert_{0,\Omega}^2\\
& =  \frac{1}{\nu}\left(c_1 - \frac{\delta\varepsilon}{2}\right)\Vert \bsig_h\Vert_{\bdiv,\Omega}^2  + \delta\left(1- \frac{\widehat{\gamma}^{-2}}{2\nu\varepsilon}\right)\Vert \bu_h\Vert_{0,\Omega}^2.
\end{aligned}
$$
Choosing $\varepsilon=\widehat{\gamma}^{-2}/\nu$ and $\delta= c_1\widehat{\gamma}^2\nu$ it follows that
$$
B((\bsig_h,\bu_h),(\bsig_h-\delta\widetilde{\bsig}_h,-\bu_h))\geq \frac{c_1}{2\nu}\Vert\bsig_h\Vert_{\bdiv,\Omega}^2 + \frac{c_1\widehat{\gamma}^2\nu}{2}\Vert\bu_h\Vert_{0,\Omega}^2\geq C_J\vertiii{(\bsig_h,\bu_h)}^2,
$$
with $C_J:=\min\left\{\dfrac{c_1}{2\nu},\dfrac{c_1\widehat{\gamma}^2\nu}{2}\right\}$.  Moreover, we have
$$
\vertiii{(\bsig_h-\delta\widetilde{\bsig}_h,-\bu_h)}^2= \Vert \bsig_h - \delta\widetilde{\bsig}_h\Vert_{\bdiv,\Omega}^2 + \Vert \bu_h\Vert_{0,\Omega}^2\leq \max\{1,c_1^2\widehat{\gamma}^2\nu^2 \}\vertiii{(\bsig_h,\bu_h)}^2.
$$
The proof is completed.
\end{proof}

It now corresponds to discuss the existence and uniqueness of solutions for the source problem \eqref{def:oseen_system_source_disc}. 
The proof of the existence of a solution follows analogously as in the continuous case, but taking into consideration the discrete inf-sup constant. 
On the other hand, uniqueness follows is stated in  following result.
\begin{lemma}
\label{lmm:gonzalo1}
Let $(\widehat{\bsig}_h,\widehat{\bu}_h)\in \mathbb{H}_{0,h}\times \bQ_h$ be a solution of problem \eqref{def:oseen_system_source_disc}. If \linebreak $C_J\dfrac{\|\boldsymbol{\beta}\|_{\infty}}{\nu}<1$, where the constant $C_J$ is the one provided by Lemma \ref{inf-sup_B}, then $(\widehat{\bsig}_h,\widehat{\bu}_h)\in \mathbb{H}_{0,h}\times \bQ_h$ is unique.
\end{lemma}
\begin{proof}
Let $(\widehat{\bu}_h^1,\widehat{\bsig}_h^1)$ and $(\widehat{\bu}_h^2,\widehat{\bsig}_h^2)$ be two solutions of \eqref{def:oseen_system_source_disc}. Let us prove that they are equal. From Lemma \ref{inf-sup_B} we have that there exist $(\btau_h,\bv_h)\in\mathbb{H}_{0,h}\times \bQ_h$ and a positive constant $C_J$ such that $\vertiii{(\widehat{\bsig}_h,\widehat{\bu}_h)}^2\leq C_J B((\widehat{\bsig}_h,\widehat{\bu}_h),(\btau_h,\bv_h))$. Then,
$$\vertiii{(\btau_h,\bv_h)}\leq \vertiii{(\widehat{\bsig}_h^1-\widehat{\bsig}_h^2,\widehat{\bu}_h^1-\widehat{\bu}_h^2)}.$$
Hence
\begin{align*}
\vertiii{(\widehat{\bsig}_h^1-\widehat{\bsig}_h^2,\widehat{\bu}_h^1-\widehat{\bu}_h^2)}^2\leq& C_J[B((\widehat{\bsig}_h^1,\widehat{\bu}_h^1),(\btau_h,\bv_h))-B((\widehat{\bsig}_h^2,\widehat{\bu}_h^2),(\btau_h,\bv_h))]\\
\leq & C_J[-c(\widehat{\bu}_h^1,\btau_h)-(\boldsymbol{f},\bv_h)_{0,\O}+c(\widehat{\bu}_h^2,\btau_h)+(\boldsymbol{f},\bv_h)]\\
\leq& C_Jc(\widehat{\bu}_h^2-\widehat{\bu}_h^1,\btau_h)\leq C_J\frac{\|\boldsymbol{\beta}\|_{\infty,\O}}{\nu}\|\widehat{\bu}_h^2-\widehat{\bu}_h^2\|_{0,\O}\|\btau_h\|_{0,\O}.
\end{align*}
Then
\begin{align*}
\|\widehat{\bsig}_h^1-\widehat{\bsig}_h^2\|_{\bdiv,\O}^2+\|\bu_h^1-\bu_h^2\|_{0,\O}^2-C_J\frac{\|\boldsymbol{\beta}\|_{\infty,\O}}{\nu}\|\widehat{\bu}_h^1-\widehat{\bu}_h^2\|_{0,\O}\|\btau_h\|_{0,\O}&\leq 0\\
\Longrightarrow \left(1-C_J\frac{\|\boldsymbol{\beta}\|_{\infty,\O}}{\nu}\right)\left(\|\widehat{\bsig}_h^1-\widehat{\bsig}_h^2\|_{\bdiv,\O}^2+\|\widehat{\bu}_h^1-\widehat{\bu}_h^2\|_{0,\O}^2\right)&\leq 0
\end{align*}
Finally, if $\displaystyle C_J\frac{\|\boldsymbol{\beta}\|_{\infty,\O}}{\nu}<1$, we conclude the proof.
\end{proof}

As in the continuous case, it is easy to check that $(\lambda_h,(\bu_h,p_h))$ solves problem \eqref{def:oseen_system_weak_1_H0_disc} if and only if $(\kappa_h,\bu_h)$ is an eigenpair of  $\bT_h$.
\subsection{The discrete dual problem}
The dual discrete eigenvalue problem reads as follows: Find $\lambda^*_h\in\mathbb{C}$ and the pair $(\boldsymbol{0},\boldsymbol{0})\neq(\bsig_h^*, \bu_h^*)\in\mathbb{H}_{0,h}\times\bQ_h$ such that  
\begin{equation}\label{def:oseen_system_weak_dual_source11_disc}
	\left\{
	\begin{array}{rcll}
a(\btau_h,\bsig_h^*)+b(\btau_h,\bu_h^*)-c(\btau_h,\bu_h^*)&=&0&\forall \btau_h\in \mathbb{H}_{0,h},\\
b(\bsig_h^*,\bv_h)&=&-\lambda_h^*(\bv_h,\boldsymbol{u}^*_h)&\forall \bv_h\in \bQ_h.
\end{array}
	\right.
\end{equation}
and  let us introduce the discrete version of \eqref{eq:operador_adjunto_solucion_u} which we define by $$\label{eq:operador_solucion_u_h_adjunto}
	\bT^*_h:\bQ\rightarrow \bQ_h,\qquad \boldsymbol{f}\mapsto \bT_h^*\boldsymbol{f}:=\widehat{\bu}_h^*, 
$$
where the pair $(\widehat{\bu}_h^*,\widehat{p}_h^*)\in\mathcal{X}_h$ is the solution of the following adjoint discrete source problem
\begin{equation}\label{def:oseen_system_weak_dual_source_disc}
	\left\{
	\begin{array}{rcll}
a(\btau_h,\widehat{\bsig}_h^*)+b(\btau_h,\widehat{\bu}_h^*)-c(\btau_h,\widehat{\bu}_h^*)&=&0&\forall \btau_h\in \mathbb{H}_{0,h},\\
b(\widehat{\bsig}_h^*,\bv_h)&=&-(\bv_h,\boldsymbol{f})&\forall \bv_h\in \bQ_h.
\end{array}
	\right.
\end{equation}

As in the continuous case, to prove that $\bT^*_h$ is well defined, we can follow the same arguments used for the goodness of $\bT_h$. So, we skip the details.
\subsection{Convergence for the solution operators}
Now, due to the compactness of $\bT$, we are able to prove that $\bT_h$ converge to $\bT$ as $h$ goes to zero in norm and under some assumptions on the data. This is contained in the following result.
\begin{lemma}
\label{lmm:conv1}
Let $\boldsymbol{f}\in \bQ$ be such that $\widehat{\bu}:=\boldsymbol{T}\boldsymbol{f}$ and $\widehat{\bu}_h:=\bT_h\boldsymbol{f}$, solutions of problems \eqref{def:oseen_system_source_00} and \eqref{def:oseen_system_source_disc} respectively, assume that $\widehat{\bu}\in \boldsymbol{\H}^{1+s}(\O)$ and $\widehat{\bsig}\in \mathbb{H}^s(\O)$ with $s>0$ and assume that the data is sufficiently small and that \eqref{eq:reg_u_if} holds. Then,  if $C_J\dfrac{\|\boldsymbol{\beta}\|_{\infty}}{\nu}<1$, there exists a positive constant $\widehat{C}_{\nu,\beta}$, independent of $h$, such that 
$$
\vertiii{(\widehat{\bsig}-\widehat{\bsig}_h,(\bT-\bT_h)\boldsymbol{f})}\leq  \widehat{C}_{\nu,\beta}h^s\left(\|\widehat{\bsig}\|_{s,\O}+\|\widehat{\bu}\|_{1+s,\O}\right)\leq  \widehat{C}_{\nu,\beta}h^s\|\boldsymbol{f}\|_{0,\O}.
$$
\end{lemma}

\begin{proof}
Let $(\widehat{\bsig},\widehat{\bu})$ and $(\widehat{\bsig}_h,\widehat{\bu}_h)$ be the solutions 
of the continuous and discrete source problems, respectively,  and 
let $(\btau_h,\bv_h)$ be such that
$$\vertiii{(\btau_h,\bv_h)}\leq C_J\vertiii{(\widehat{\bsig}_I-\widehat{\bsig},\widehat{\bu}_I-\widehat{\bu})},$$ where 
$\widehat{\bsig}_I$ and $\widehat{\bu}_I$ are the best approximations of $\widehat{\bsig}$ and $\widehat{\bu}$, respectively.
Then, from triangle inequality we have
$$\vertiii{(\widehat{\bsig}-\widehat{\bsig}_h,\widehat{\bu}-\widehat{\bu}_h)}
\leq \vertiii{(\widehat{\bsig}-\widehat{\bsig}_I,\widehat{\bu}-\widehat{\bu}_I)}+
\vertiii{(\widehat{\bsig}_I-\widehat{\bsig}_h,\widehat{\bu}_I-\widehat{\bu}_h)}.$$

On the other hand
\begin{multline*}
\vertiii{(\widehat{\bsig}_I-\widehat{\bsig}_h,\widehat{\bu}_I-\widehat{\bu}_h)}^2
\leq C_JB((\widehat{\bsig}_I-\widehat{\bsig},\widehat{\bu}_I-\widehat{\bu})),(\btau_h,\bv_h))\\
+C_JB((\widehat{\bsig}-\widehat{\bsig}_h,\widehat{\bu}-\widehat{\bu}_h)),(\btau_h,\bv_h))\\
=B((\widehat{\bsig}_I-\widehat{\bsig},\widehat{\bu}_I-\widehat{\bu})),(\btau_h,\bv_h))
+C_J[-(\boldsymbol{f},\bv_h)-c(\widehat{\bu},\btau_h)-[-(\boldsymbol{f},\bv_h)-c(\widehat{\bu}_h,\btau_h)]]\\
=C_JB((\widehat{\bsig}_I-\widehat{\bsig},\widehat{\bu}_I-\widehat{\bu}),(\btau_h,\bv_h))+C_Jc(\widehat{\bu}_h-\widehat{\bu},\btau_h)\\
\leq C_JB((\widehat{\bsig}_I-\widehat{\bsig},\widehat{\bu}_I-\widehat{\bu}),(\btau_h,\bv_h))+C_J\frac{\|\boldsymbol{\beta}\|_{\infty,\O}}{\nu}\|\widehat{\bu}_h-\widehat{\bu}\|_{0,\O}\|\btau_h\|_{0,\O}\\
\leq C_JB((\widehat{\bsig}_I-\widehat{\bsig},\widehat{\bu}_I-\widehat{\bu}),(\btau_h,\bv_h))+C_J\frac{\|\boldsymbol{\beta}\|_{\infty,\O}}{\nu}(\|\widehat{\bu}_h-\widehat{\bu}_I\|_{0,\O}+\|\widehat{\bu}_I-\widehat{\bu}\|_{0,\O})\|\btau_h\|_{0,\O}.
\end{multline*}
From this inequality and proceeding as in the proof of Lemma \ref{lmm:gonzalo1} we have
\begin{multline*}
\left(\dfrac{1}{C_J}-\frac{\|\boldsymbol{\beta}\|_{\infty,\O}}{\nu}\right)(\|\widehat{\bsig}_I-\widehat{\bsig}_h\|_{\bdiv,\O}^2+\|\widehat{\bu}_I-\widehat{\bu}_h\|_{0,\O}^2)\leq B((\widehat{\bsig}_I-\widehat{\bsig},\widehat{\bu}_I-\widehat{\bu}),(\btau_h,\bv_h))\\
+\frac{\|\boldsymbol{\beta}\|_{\infty,\O}}{\nu}\|\widehat{\bu}_I-\widehat{\bu}\|_{0,\O}\|\btau_h\|_{0,\O}.
\end{multline*}
Let us focus our attention on the right-hand side of the inequality above. According to the definition of $B(\cdot,\cdot)$ (cf. \eqref{eq:formB}) we have
\begin{multline*}
B((\widehat{\bsig}_I-\widehat{\bsig},\widehat{\bu}_I-\widehat{\bu}),(\btau_h,\bv_h))+\frac{\|\boldsymbol{\beta}\|_{\infty,\O}}{\nu}\|\widehat{\bu}_I-\widehat{\bu}\|_{0,\O}\|\btau_h\|_{0,\O}\\
=a(\widehat{\bsig}_I-\widehat{\bsig},\btau_h)+b(\btau_h,\widehat{\bu}_I-\widehat{\bu})
+b(\widehat{\bsig}_I-\widehat{\bsig},\bv_h)+\frac{\|\boldsymbol{\beta}\|_{\infty,\O}}{\nu}\|\widehat{\bu}_I-\widehat{\bu}\|_{0,\O}\|\btau_h\|_{0,\O}\\
\leq\frac{1}{\nu}\|\widehat{\bsig}_I-\widehat{\bsig}\|_{\bdiv,\O}\|\btau_h\|_{\bdiv,\O}+
\|\btau_h\|_{\bdiv,\O}\|\widehat{\bu}_I-\widehat{\bu}\|_{0,\O}\\
+\|\widehat{\bsig}_I-\widehat{\bsig}\|_{\bdiv,\O}\|\bv_h\|_{0,\O}+\frac{\|\boldsymbol{\beta}\|_{\infty,\O}}{\nu}\|\widehat{\bu}_I-\widehat{\bu}\|_{0,\O}\|\btau_h\|_{0,\O}\\
\leq\underbrace{\max\left\{\frac{1}{\nu},\frac{\|\boldsymbol{\beta}\|_{\infty,\O}}{\nu}+1\right\}}_{C_{\nu,\boldsymbol{\beta}}}
(\|\widehat{\bsig}_I-\widehat{\bsig}\|_{\bdiv,\O}+\|\widehat{\bu}_I-\widehat{\bu}\|_{0,\O})\|\btau_h\|_{\bdiv,\O}\\
+\|\widehat{\bsig}_I-\widehat{\bsig}\|_{\bdiv,\O}\|\bv_h\|_{0,\O}\\
\leq C_{\nu,\boldsymbol{\beta}}(\|\widehat{\bsig}_I-\widehat{\bsig}\|_{\bdiv,\O}^2+\|\widehat{\bu}_I-\widehat{\bu}\|_{0,\O}^2)+\frac{C_{\nu,\boldsymbol{\beta}}}{2}\|\btau_h\|_{0,\O}^2+\frac{1}{2}(\|\widehat{\bsig}_I-\widehat{\bsig}\|_{\bdiv,\O}^2+\|\bv_h\|_{0,\O}^2)\\
\leq C_{\nu,\boldsymbol{\beta}}(\|\widehat{\bsig}_I-\widehat{\bsig}\|_{\bdiv,\O}^2+\|\widehat{\bu}_I-\widehat{\bu}\|_{0,\O}^2)+\frac{\max\{C_{\nu,\boldsymbol{\beta}},1\}}{2}(\|\btau_h\|_{\bdiv,\O}^2+\|\bv_h\|_{0,\O}^2)\\
+\frac{\|\widehat{\bsig}_I-\widehat{\bsig}\|_{\bdiv,\O}^2}{2}+\frac{\|\widehat{\bu}_I-\widehat{\bu}\|_{0,\O}^2}{2}\\
\leq\left(C_{\nu,\boldsymbol{\beta}}+\frac{\max\{C_{\nu,\boldsymbol{\beta}},1\}}{2}+\frac{1}{2}\right)\vertiii{(\widehat{\bsig}_I-\widehat{\bsig},\widehat{\bu}_I-\widehat{\bu})}^2.
\end{multline*}
In summary, we have
$$\vertiii{(\widehat{\bsig}-\widehat{\bsig}_h,\widehat{\bu}-\widehat{\bu}_h)}
\leq \widehat{C}_{\nu,\boldsymbol{\beta}}\vertiii{(\widehat{\bsig}-\widehat{\bsig}_I,\widehat{\bu}-\widehat{\bu}_I)},$$
where $\widehat{C}_{\nu,\boldsymbol{\beta}}=\max\left\{1,\sqrt{\dfrac{\left(C_{\nu,\boldsymbol{\beta}}+\frac{\max\{C_{\nu,\boldsymbol{\beta}},1\}}{2}+\frac{1}{2}\right)}{\left(\dfrac{1}{C_J}-\frac{\|\boldsymbol{\beta}\|_{\infty,\O}}{\nu}\right)}}\right\}$. 
Now using the approximation properties presented in \cite[Section 3.2]{MR4430561} we have
$$\vertiii{(\widehat{\bsig}-\widehat{\bsig}_h,\widehat{\bu}-\widehat{\bu}_h)}
\leq \widehat{C}_{\nu,\boldsymbol{\beta}}h^s\left(\|\widehat{\bsig}\|_{s,\O}+\|\widehat{\bu}\|_{1+s,\O}\right).$$
We conclude the proof invoking \eqref{eq:reg_u_if}.
\end{proof}

For the adjoint counterparts of $\bT$ and $\bT_h$, namely $\bT^*$ and $\bT_h^*$, 
we can prove convergence of   $\bT_h^*$  to $\bT^*$ as $h$ goes to zero. Also, this convergence result depends on some suitable choice on the data of the problem. Since the proof is essentially identical to Lemma \ref{lmm:conv1} we skip the steps of the proof.
\begin{lemma}
\label{eq:adjoint_diff}
Assume that \eqref{eq:reg_u_if*} holds. Then, there exists a constant $C>0$, independent of $h$, such that
$$
\vertiii{(\widehat{\bsig}^*-\widehat{\bsig}_h^*,(\bT^*-\bT_h^*)\boldsymbol{f})}\leq  \widehat{C}_{\nu,\beta}^*h^{s^*}\left(\|\widehat{\bsig}^*\|_{s^*,\O}+\|\widehat{\bu}^*\|_{1+s^*,\O}\right)\leq  \widehat{C}_{\nu,\beta}^*h^{s^*}\|\boldsymbol{f}\|_{0,\O}.
$$

\end{lemma}

With all these results at hand, we are in position to apply the theory of  \cite[Chapter IV]{MR0203473} and \cite[Theorem 9.1]{MR2652780} to conclude that  our numerical methods does not introduce spurious eigenvalues.
\begin{theorem}
	\label{thm:spurious_free}
	Let $V\subset\mathbb{C}$ be an open set containing $\sp(\bT)$. Then, there exists $h_0>0$ such that $\sp(\bT_h)\subset V$ for all $h<h_0$.
\end{theorem}
\subsection{Error estimates}
\label{sec:conv}
Now our task is to prove a priori error estimates for the eigenfunctions and eigenvalues. With this goal in mind, let us recall the definition of spectral projectors. Let $\mu$ be a nonzero isolated eigenvalue of $\bT$ with algebraic multiplicity $m$ and let $\Gamma$
be a disk of the complex plane centered in $\mu$, such that $\mu$ is the only eigenvalue of $\bT$ lying in $\Gamma$ and $\partial\Gamma\cap\sp(\bT)=\emptyset$. The spectral projections of $\boldsymbol{E}$ and $\boldsymbol{E}^*$, associated to $\bT$ and $\bT^*$ are defined, respectively, in the following way
\begin{enumerate}
\item[a)] The spectral projector of $\bT$ associated to $\mu$ is $\displaystyle \boldsymbol{E}:=\frac{1}{2\pi i}\int_{\partial\Gamma} (z\boldsymbol{I}-\bT)^{-1}\,dz;$
\item[b)] The spectral projector of $\bT^*$ associated to $\bar{\mu}$ is $\displaystyle \boldsymbol{E}^*:=\frac{1}{2\pi i}\int_{\partial\Gamma} (z\boldsymbol{I}-\bT^*)^{-1}\,dz,$
\end{enumerate}
where $\boldsymbol{I}$ represents the identity operator and  $\boldsymbol{E}$ and $\boldsymbol{E}^*$ are the projections onto the generalized eigenspaces $R(\boldsymbol{E})$ and $R(\boldsymbol{E}^*)$, respectively. 

The convergence in norm stated in Lemma \ref{lmm:conv1} gives as consequence  the  existence of  $m$ eigenvalues lying  in $\Gamma$, which we denote by  $\mu_h^{(1)},\ldots,\mu_h^{(m)}$, repeated according their respective multiplicities, that converge to $\mu$ as $h$ goes to zero. This motivates the  definition of the following discrete  spectral projection
$$
\boldsymbol{E}_h:=\frac{1}{2\pi i}\int_{\partial\Gamma} (z\boldsymbol{I}-\bT_h)^{-1}\,dz.
$$
This operator is precisely a projection onto the discrete invariant subspace $R(\boldsymbol{E}_h)$ of $\bT$, spanned by the generalized eigenvector of $\bT_h$ corresponding to 
 $\mu_h^{(1)},\ldots,\mu_h^{(m)}$.
Another necessary ingredient  for the error analysis is the \textit{gap} $\hdel(\cdot,\cdot)$ between two closed
subspaces $\mathfrak{X}$ and $\mathfrak{Y}$ of $ \boldsymbol{\L}^2(\O)$, which is defined by 
$$
\hdel(\mathfrak{X},\mathfrak{Y})
:=\max\big\{\delta(\mathfrak{X},\mathfrak{Y}),\delta(\mathfrak{Y},\mathfrak{X})\big\}, \text{ where } \delta(\mathfrak{X},\mathfrak{Y})
:=\sup_{\underset{\left\|\boldsymbol{x}\right\|_{0,\O}=1}{\boldsymbol{x}\in\mathfrak{X}}}
\left(\inf_{\boldsymbol{y}\in\mathfrak{Y}}\left\|\boldsymbol{x}-\boldsymbol{y}\right\|_{0,\O}\right).
$$
Now we present the main result of this section.
\begin{theorem}
\label{thm:errors1}
For $k\geq 0$ and  assume that $\widehat{C}_{\nu,\boldsymbol{\beta}}$, $\widehat{C}_{\nu,\boldsymbol{\beta}}^*$, and $\widetilde{C}_{\nu,\boldsymbol{\beta}}$ are sufficiently small as in Lemma \ref{lmm:conv1}  and Lemma \ref{eq:adjoint_diff}. the following estimates hold
$$
\hdel(R(\boldsymbol{E}),R(\boldsymbol{E}_h))\leq \widehat{C}_{\nu,\boldsymbol{\beta}} h^{\min\{k+1,s\}}\quad\text{and}\quad
|\mu-\mu_h|\leq  \widetilde{C}_{\nu,\boldsymbol{\beta}}h^{\min\{2(k+1),s+s^*\}},
$$
where  $$\widetilde{C}_{\nu,\boldsymbol{\beta}}:=\widehat{C}_{\nu,\boldsymbol{\beta}}\widehat{C}_{\nu,\boldsymbol{\beta}}^*\left(\dfrac{1+\|\boldsymbol{\beta}\|_{\infty,\O}}{\nu}+3\right).$$

\end{theorem}
\begin{proof}
The proof of the gap between the eigenspaces is a direct consequence of the convergence in norm between $\bT$ and $\bT_h$ as $h$ goes to zero.
We focus on the double order of convergence for the eigenvalues. Let $\{\bu_k\}_{k=1}^m$ be such that $\bT \bu_k=\mu \bu_k$, for $k=1,\ldots,m$. A dual basis for $R(\boldsymbol{E}^*)$ is $\{\bu_k^*\}_{k=1}^m$. This basis satisfies $A((\bsig,\bu_k);(\bsig^*,\bu_l^*))=\delta_{k.l},$
where $\delta_{k.l}$ represents the Kronecker delta and $A((\bsig,\bu_k);(\bsig^*,\bu_l^*))=a(\bsig,\bsig^*)+b(\bsig^*,\bu_k)+c(\bu_k,\bsig^*)+b(\bsig,\bu_l^*))$.
On the other hand, the following identity holds
$$
|\mu-\widehat{\mu}_h|\leq  \frac{1}{m}\sum_{k=1}^m|\langle(\bT-\bT_h)\bu_k,\bu_k^* \rangle|+\|(\bT-\bT_h)|_{R(\boldsymbol{E})} \|_{1,\Omega} \|(\bT^*-\bT_h^*)|_{R(\boldsymbol{E}^*)}\|_{1,\Omega},
$$
where $\langle\cdot,\cdot\rangle$ denotes the corresponding duality pairing. For the first term on the right-hand side we note that
\begin{multline*}
\langle(\bT-\bT_h)\bu_k,\bu_k^* \rangle=A((\bsig,(\bT-\bT_h)\bu_k);(\bsig^*,\bu_k^*))\\=
A((\bsig-\bsig_h,(\bT-\bT_h)\bu_k);(\bsig^*,\bu_k^*))+A((\bsig_h,(\bT-\bT_h)\bu_k);(\bsig^*,\bu_k^*))\\
=A((\bsig-\bsig_h,(\bT-\bT_h)\bu_k);(\bsig^*-\bsig_h^*,\bu_k^*-\bu_{k,h}^*))\\
+A((\bsig-\bsig_h,(\bT-\bT_h)\bu_k);(\bsig_h^*,\bu_{k,h}^*)))+A((\bsig_h,(\bT-\bT_h)\bu_k);(\bsig^*,\bu_k^*))\\
=A((\bsig-\bsig_h,(\bT-\bT_h)\bu_k);(\bsig^*-\bsig_h^*,\bu_k^*-\bu_{k,h}^*)),
\end{multline*}
where for the last equality, we have used the Galerkin orthogonality. Then
\begin{multline*}
\langle(\bT-\bT_h)\bu_k,\bu_k^* \rangle=A((\bsig-\bsig_h,(\bT-\bT_h)\bu_k);(\bsig^*-\bsig_h^*,\bu_k^*-\bu_{k,h}^*))\\
=a(\bsig-\bsig_h,\bsig^*-\bsig_h^*)+b(\bsig^*-\bsig_h^*,(\bT-\bT_h)\bu_k)+b(\bsig-\bsig_h,\bu_k^*-\bu_{k,h}^*)\\
+c((\bT-\bT_h)\bu_k,\bsig^*-\bsig_h^*)\\
\leq \dfrac{1}{\nu}\|\bsig-\bsig_h\|_{\bdiv,\O}\|\bsig^*-\bsig_h^*\|_{\bdiv,\O}+\|\bsig^*-\bsig_h^*\|_{\bdiv,\O}\|(\bT-\bT_h)\bu_k\|_{0,\O}\\
+\|\bsig-\bsig_h\|_{\bdiv,\O}\|\bu_k^*-\bu_{k,h}^*\|_{0,\O}+\dfrac{1}{\nu}\|(\bT-\bT_h)\bu_k\|_{0,\O}\|\bsig^*-\bsig_h^*\|_{0,\O}\|\boldsymbol{\beta}\|_{\infty,\O}.
\end{multline*}
Then, Theorem \ref{thm:errors1} follows from the above estimates and the approximation properties of discrete spaces, in addition to Lemmas \ref{lmm:conv1} and \ref{eq:adjoint_diff}.
\end{proof}

\section{Numerical experiments}
\label{sec:numerics}
This section is devoted to perform several numerical experiments to analyze the behavior of the scheme in different geometries and physical configurations. We resort to the DOLFINx software \cite{barrata2023dolfinx,scroggs2022basix} for the implementations of the codes, where the SLEPc eigensolver \cite{hernandez2005slepc} and the linear solver MUMPS are used to solve the resulting general eigenvalue problem. The meshes are constructed with the help of GMSH \cite{geuzaine2009gmsh} and the DOLFINx built-in generic meshes.

The rates of convergence for each eigenvalue  are computed by using a least-square fitting. More precisely, if $\lambda_h$ is a discrete complex eigenvalue, then the rate of convergence $\alpha$ is calculated by extrapolation and the least square  fitting
$$
\lambda_{h}\approx \lambda_{\text{extr}} + Ch^{\alpha},
$$
where $\lambda_{\text{extr}}$ is the extrapolated eigenvalue given by the fitting. 

In what follows, we denote the mesh resolution by $N$, which is connected to the mesh-size $h$ through the relation $h\sim N^{-1}$. We also denote the number of degrees of freedom by $\texttt{dof}$. The relation between $\texttt{dof}$ and the mesh size is given by $h\sim\texttt{dof}^{-1/n}$, with $n\in\{2,3\}$. 

According to the assumptions on $\boldsymbol{\beta}$, we will assume $\nu=1/2$ in all the experiments, unless stated otherwise. Also, for each convective velocity, we assume that it is normalized and, for the sake of uniformity, is denoted by $\boldsymbol{\beta}=\boldsymbol{\beta}/\Vert\boldsymbol{\beta}\Vert_{\infty,\Omega}$.

\subsection{Results in 2D geometries}\label{subsec:numeric-2D}
In this section we study the behavior of the scheme when choosing different geometries and convective velocities. The idea is to observe the convergence and the complex spectrum when we have different configurations of the model. To do this task, the values for the solenoidal convective velocity are considered among the following
$$
\begin{aligned}
	&\boldsymbol{\beta}_1(x,y)=(1,0)^{\texttt{t}},\quad  \boldsymbol{\beta}_2(x,y)=(\cos(\pi x)\sin(\pi y),-\sin(\pi x)\cos(\pi y))^{\texttt{t}},\\
	&\boldsymbol{\beta}_3(x,y)=(y,-x)^{\texttt{t}},\quad  \boldsymbol{\beta}_4(x,y)=\left(\frac{\partial\phi}{\partial y},-\frac{\partial\phi}{\partial x}\right)^{\texttt{t}}, \\
\end{aligned}
$$
with $\phi(x,y)=1000(1-x^2)^2(1-y^2)^2$.
The geometries under study are the square and lshaped domain. Each case is described below.

\subsubsection{Square domain}\label{subsec:square-domain2D}
Let us start the experiments by considering the square domain $\Omega:=(0,1)^2$, which serves as a benchmark to compare the computed eigenvalues with existing references. 

Table \ref{table-square2D} describe the convergence behavior of the scheme with $\mathbb{RT}_k$ and $\mathbb{BDM}_{k+1}$ with $\boldsymbol{\beta}(x,y)=\boldsymbol{\beta}_1(x,y)$. Both cases shows an optimal rate of convergence, with errors behaving like $\mathcal{O}(h^{2(k+1)})$. However, when comparing the rates between both finite element families, we note that the $\mathbb{BDM}$ family yields to a more stable scheme overall. Moreover, when the convective velocity results to be variable, the order of convergence with the lowest order Brezzi-Douglas-Marini family is also attained on the square, as is shown in Table  \ref{table-square2D-BDM_variablebeta}.  Also, for these values of  $\boldsymbol{\beta}$, complex eigenvalues are observed when the spectrum is computed. This is clearly expected due the nature of the Oseen eigenvalue problem. 
		
	Finally, the first fifty computed eigenvalues for the given choices of convective velocities are portrayed in Figure \ref{fig:square-domain-complejos}. For this test, a mesh level $N=100$ was selected. Here, we observe that the amount of complex eigenvalues for the simplest case $\boldsymbol{\beta}_2$ is considerable lower than the rest. From simple inspection on the graphs, we note that the complex eigenvalues appearing for $\boldsymbol{\beta}_i$, $i=2,3,4$, are almost equal, except for their magnitude. 

\begin{table}[hbt!]
	\centering 
	{\footnotesize
		\begin{center}
			\caption{Example \ref{subsec:square-domain2D}. Convergence behavior of the first four lowest computed eigenvalues on the square domain with $\mathbb{RT}_k$ and $\mathbb{BDM}_{k+1}$ elements. We consider the field $\boldsymbol{\beta}=(1,0)^{\texttt{t}}$. }
			\begin{tabular}{|c c c c |c| c|c|}
				\hline
				\hline
				$N=20$             &  $N=30$         &   $N=40$         & $N=50$ & Order & $\lambda_{\text{extr}}$ & Ref. \cite{LEPE2024116959} \\ 
				\hline
				\multicolumn{7}{|c|}{$\mathbb{RT}_0$}\\
				\hline
				13.6044  &    13.6066  &    13.6077  &    13.6082  & 1.92 &    13.6096 &    13.6096  \\
				23.0754  &    23.0992  &    23.1102  &    23.1162  & 2.02 &    23.1296 &    23.1298  \\
				23.3761  &    23.3966  &    23.4061  &    23.4113  & 2.02 &    23.4229 &    23.4230  \\
				32.2352  &    32.2626  &    32.2754  &    32.2823  & 2.00 &    32.2981 &    32.2981  \\
				
				\hline
				\multicolumn{7}{|c|}{$\mathbb{RT}_1$}\\
				\hline	
				
				13.6095  &    13.6096  &    13.6096  &    13.6096  & 3.89 &    13.6096 &    13.6096  \\
				23.1299  &    23.1298  &    23.1298  &    23.1298  & 4.02 &    23.1297 &    23.1298  \\
				23.4231  &    23.4230  &    23.4230  &    23.4230  & 4.72 &    23.4230 &    23.4230  \\
				32.2983  &    32.2981  &    32.2981  &    32.2981  & 4.13 &    32.2981 &    32.2981  \\

				\hline
				\multicolumn{7}{|c|}{$\mathbb{RT}_2$}\\
				\hline
				
				13.6096  &    13.6096  &    13.6096  &    13.6096  & 5.83 &    13.6096 &    13.6096  \\
				23.1297  &    23.1297  &    23.1297  &    23.1297  & 5.86 &    23.1297 &    23.1298  \\
				23.4230  &    23.4230  &    23.4230  &    23.4230  & 5.78 &    23.4230 &    23.4230  \\
				32.2981  &    32.2981  &    32.2981  &    32.2981  & 5.76 &    32.2981 &    32.2981  \\
				
				\hline
				\multicolumn{7}{|c|}{$\mathbb{BDM}_1$}\\
				\hline
				13.6929  &    13.6467  &    13.6305  &    13.6229  & 2.02 &    13.6097 &    13.6096  \\
				23.4112  &    23.2551  &    23.2003  &    23.1749  & 2.02 &    23.1302 &    23.1298  \\
				23.7040  &    23.5481  &    23.4934  &    23.4681  & 2.02 &    23.4233 &    23.4230  \\
				32.8460  &    32.5430  &    32.4361  &    32.3865  & 2.01 &    32.2985 &    32.2981  \\
				
				\hline
				\multicolumn{7}{|c|}{$\mathbb{BDM}_2$}\\
				\hline	
				
				13.6098  &    13.6096  &    13.6096  &    13.6096  & 4.00 &    13.6096 &    13.6096  \\
				23.1309  &    23.1300  &    23.1298  &    23.1298  & 4.00 &    23.1297 &    23.1298  \\
				23.4242  &    23.4232  &    23.4231  &    23.4230  & 4.00 &    23.4230 &    23.4230  \\
				32.3017  &    32.2989  &    32.2984  &    32.2982  & 3.99 &    32.2981 &    32.2981  \\

				\hline
				\multicolumn{7}{|c|}{$\mathbb{BDM}_3$}\\
				\hline
				
				13.6096  &    13.6096  &    13.6096  &    13.6096  & 6.17 &    13.6096 &    13.6096  \\
				23.1298  &    23.1297  &    23.1297  &    23.1297  & 5.98 &    23.1297 &    23.1298  \\
				23.4230  &    23.4230  &    23.4230  &    23.4230  & 5.98 &    23.4230 &    23.4230  \\
				32.2981  &    32.2981  &    32.2981  &    32.2981  & 5.94 &    32.2981 &    32.2981  \\
				
				\hline
				
				\hline             
			\end{tabular}
	\end{center}}
	\smallskip
	
	\label{table-square2D}
\end{table}

\begin{figure}[!hbt]
	\centering
	\begin{minipage}{0.24\linewidth}\centering
		{$\bu_{h,1}$}\\
		\includegraphics[scale=0.09,trim=23cm 5cm 23cm 5cm,clip]{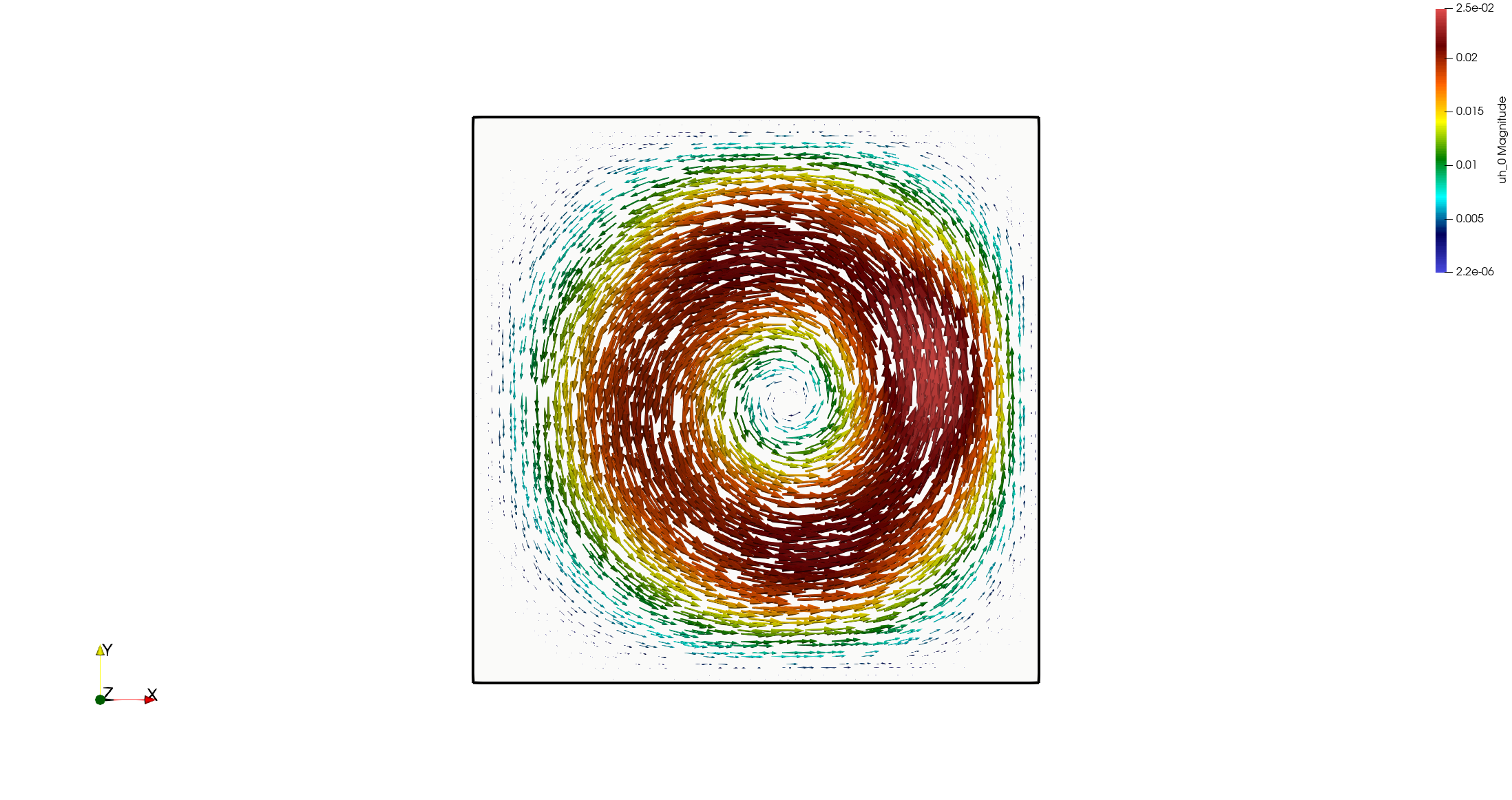}
		\end{minipage}
		\begin{minipage}{0.24\linewidth}\centering
			{$\bu_{h,4}$}\\
			\includegraphics[scale=0.09,trim=23cm 5cm 23cm 5cm,clip]{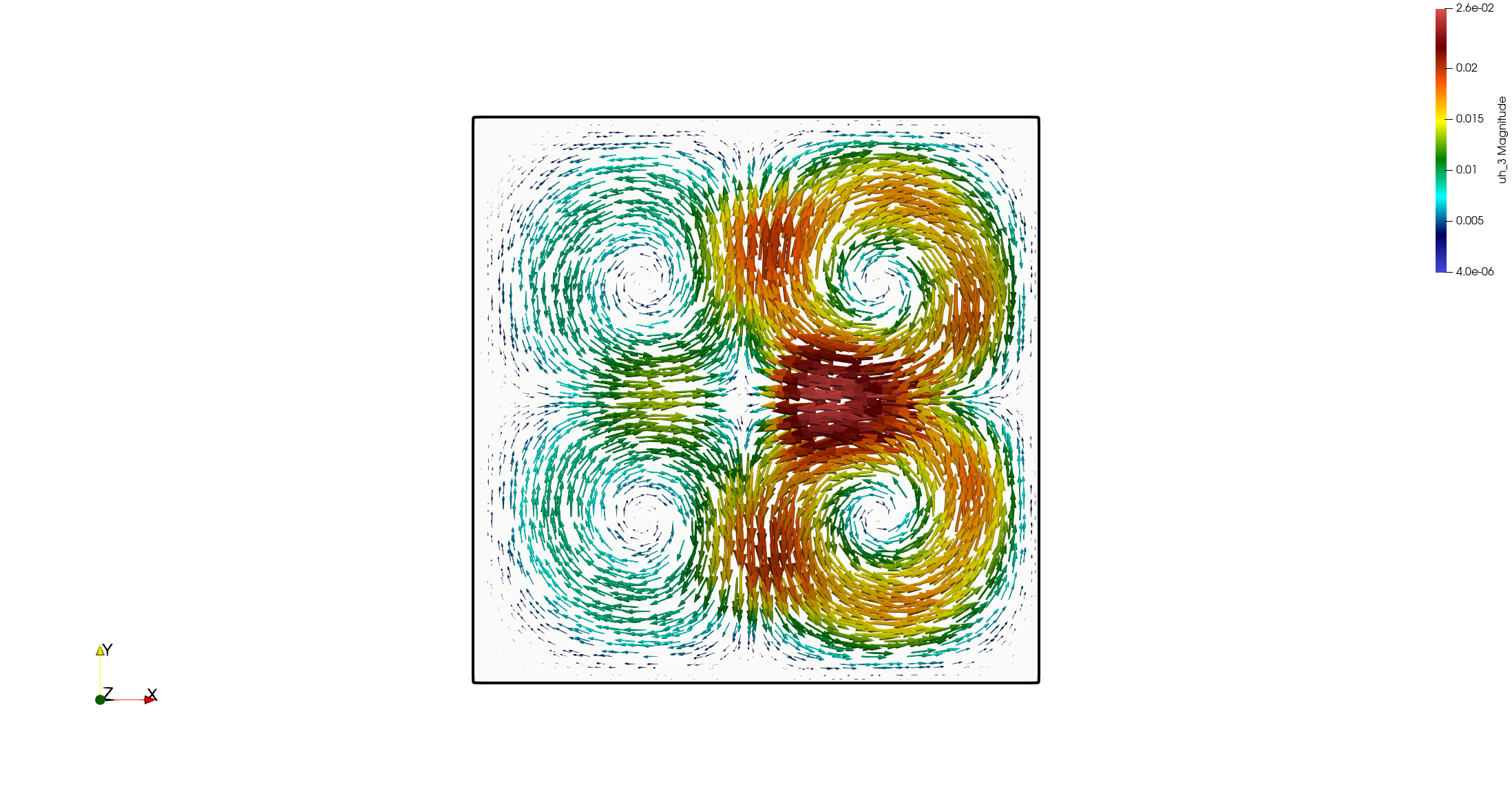}
		\end{minipage}
		\begin{minipage}{0.24\linewidth}\centering
			{$p_{h,1}$}\\
			\includegraphics[scale=0.09,trim=23cm 5cm 23cm 5cm,clip]{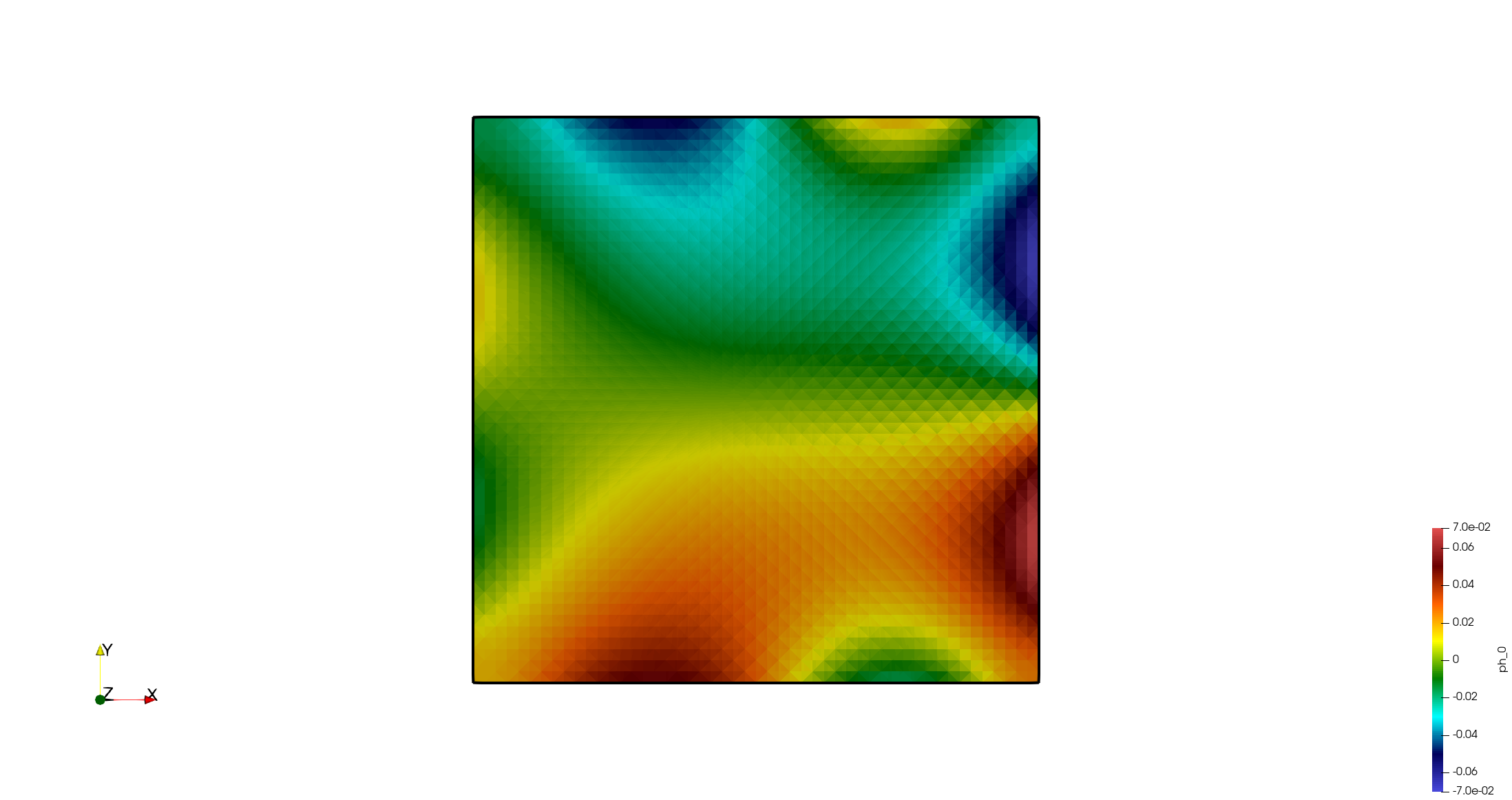}
		\end{minipage}
		\begin{minipage}{0.24\linewidth}\centering
			{$p_{h,4}$}\\
			\includegraphics[scale=0.09,trim=23cm 5cm 23cm 5cm,clip]{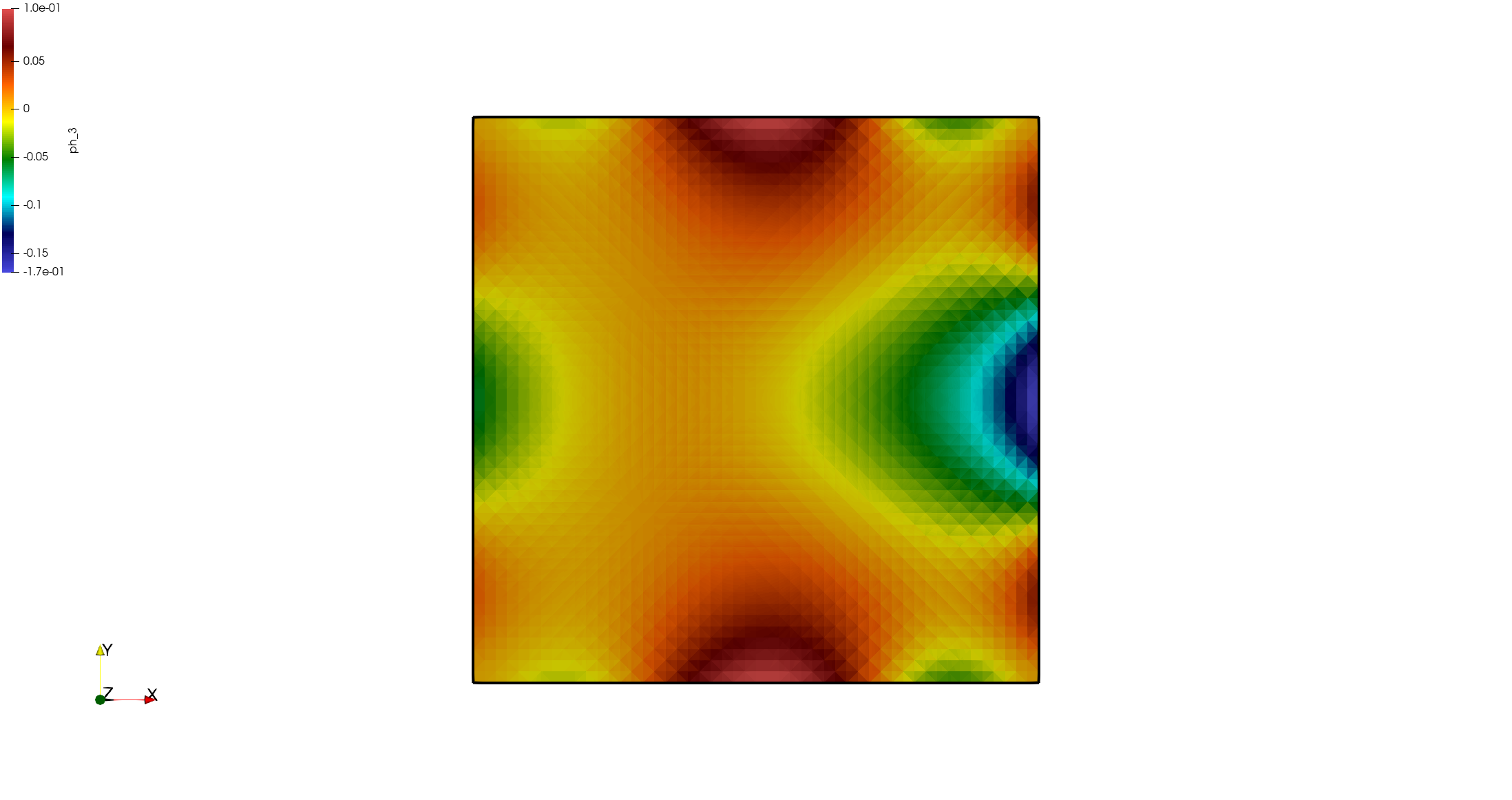}
		\end{minipage}\\
		\caption{Test \ref{subsec:square-domain2D}. Velocity streamlines and pressures surface plot for the first and fourth computed eigenvalue with $\boldsymbol{\beta}_1$.}
		\label{fig:square-2D-uhs}
\end{figure}

\begin{table}[hbt!]
	\centering 
	{\footnotesize
		\begin{center}
			\caption{Example \ref{subsec:square-domain2D}. Convergence behavior of the first four lowest computed eigenvalues on the square domain with $\mathbb{BDM}_1$ elements for the velocity field and different choices of $\boldsymbol{\beta}$.  }
			\begin{tabular}{|c| c c c c|}
				\hline
				\hline
				\texttt{dof}   &    $\lambda_{h,1}$  &   $\lambda_{h,2}$   &   $\lambda_{h,3}$  &   $\lambda_{h,4}$  \\
				\hline
				\multicolumn{5}{|c|}{$\boldsymbol{\beta}_2(x,y)$}\\
				\hline
				  6560 & $   13.1926  $   & $   23.3609  +0.8022i$  & $   23.3609  -0.8022i$  & $   33.2146  $ \\
				 14640 & $   13.1400  $   & $   23.1980  +0.7837i$  & $   23.1980  -0.7837i$  & $   32.8884  $ \\
				 25920 & $   13.1218  $   & $   23.1414  +0.7838i$  & $   23.1414  -0.7838i$  & $   32.7838  $ \\
				 40400 & $   13.1132  $   & $   23.1151  +0.7826i$  & $   23.1151  -0.7826i$  & $   32.7334  $ \\
				Order &                 2.04    &                 2.05    &                 2.05    &                 2.19 \\
				Extrap &   $   13.0985  $   & $   23.0702  +0.7771i$  & $   23.0702  -0.7771i$  & $   32.6585  $ \\
				\hline
				\multicolumn{5}{|c|}{$\boldsymbol{\beta}_3(x,y)$}\\
				\hline
				  6560 & $   13.1826  $   & $   23.3350  +0.9878i$  & $   23.3350  -0.9878i$  & $   33.3089  $ \\
				 14640 & $   13.1301  $   & $   23.1726  +0.9767i$  & $   23.1726  -0.9767i$  & $   32.9983  $ \\
				 25920 & $   13.1117  $   & $   23.1155  +0.9712i$  & $   23.1155  -0.9712i$  & $   32.8861  $ \\
				 40400 & $   13.1031  $   & $   23.0890  +0.9680i$  & $   23.0890  -0.9680i$  & $   32.8329  $ \\
				Order &                 2.02    &                 2.01    &                 2.01    &                 1.93 \\
				Extrap &   $   13.0881  $   & $   23.0417  +0.9653i$  & $   23.0417  -0.9653i$  & $   32.7332  $ \\
				\hline
				\multicolumn{5}{|c|}{$\boldsymbol{\beta}_4(x,y)$}\\
				\hline
				   6560 & $   13.1838  $   & $   23.3357  +0.9529i$  & $   23.3357  -0.9529i$  & $   33.2570  $ \\
				  14640 & $   13.1312  $   & $   23.1732  +0.9440i$  & $   23.1732  -0.9440i$  & $   32.9483  $ \\
				  25920 & $   13.1127  $   & $   23.1159  +0.9356i$  & $   23.1159  -0.9356i$  & $   32.8313  $ \\
				  40400 & $   13.1042  $   & $   23.0894  +0.9333i$  & $   23.0894  -0.9333i$  & $   32.7793  $ \\
				 Order &                 2.01    &                 2.00    &                 2.00    &                 1.88 \\
				 Extrap &   $   13.0889  $   & $   23.0415  +0.9307i$  & $   23.0415  -0.9307i$  & $   32.6724  $ \\
				\hline
				\hline             
			\end{tabular}
	\end{center}}

	\smallskip
	
	\label{table-square2D-BDM_variablebeta}
\end{table}

\begin{figure}[!hbt]\centering
	\begin{minipage}{0.49\linewidth}\centering
		\includegraphics[scale=0.35, trim=0cm 0cm 1.8cm 1.2cm, clip]{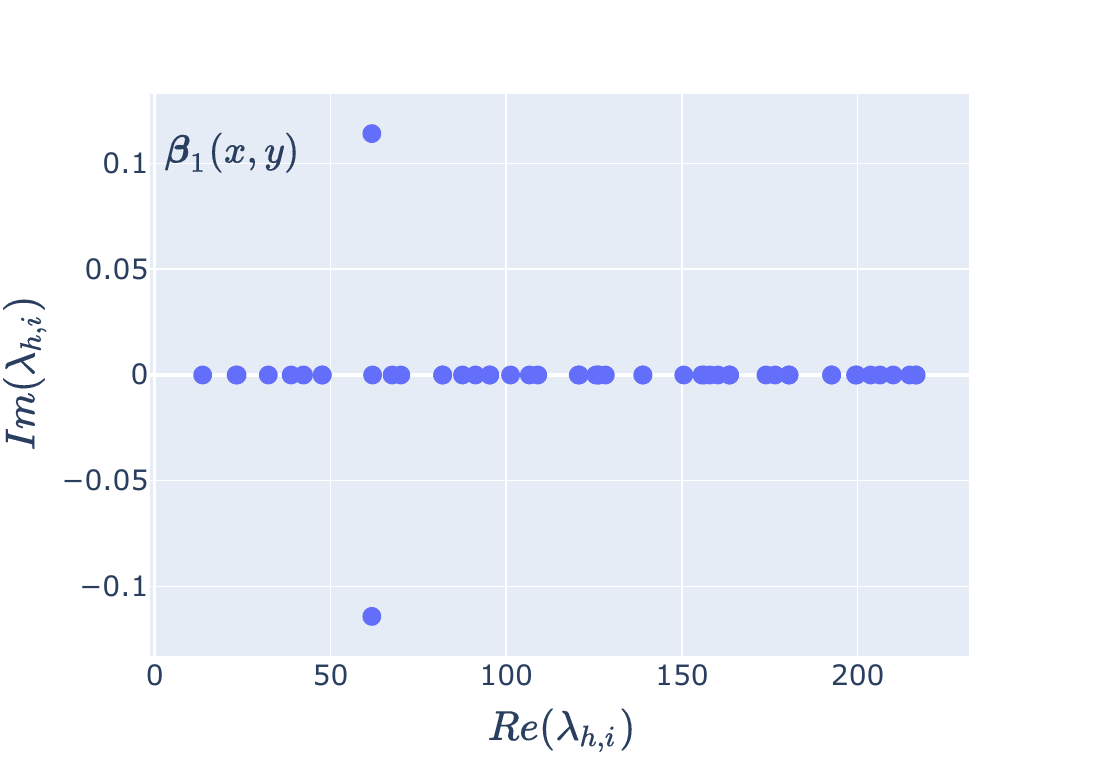}
	\end{minipage}
	\begin{minipage}{0.49\linewidth}\centering
		\includegraphics[scale=0.35, trim=0cm 0cm 1.8cm 1.2cm,clip]{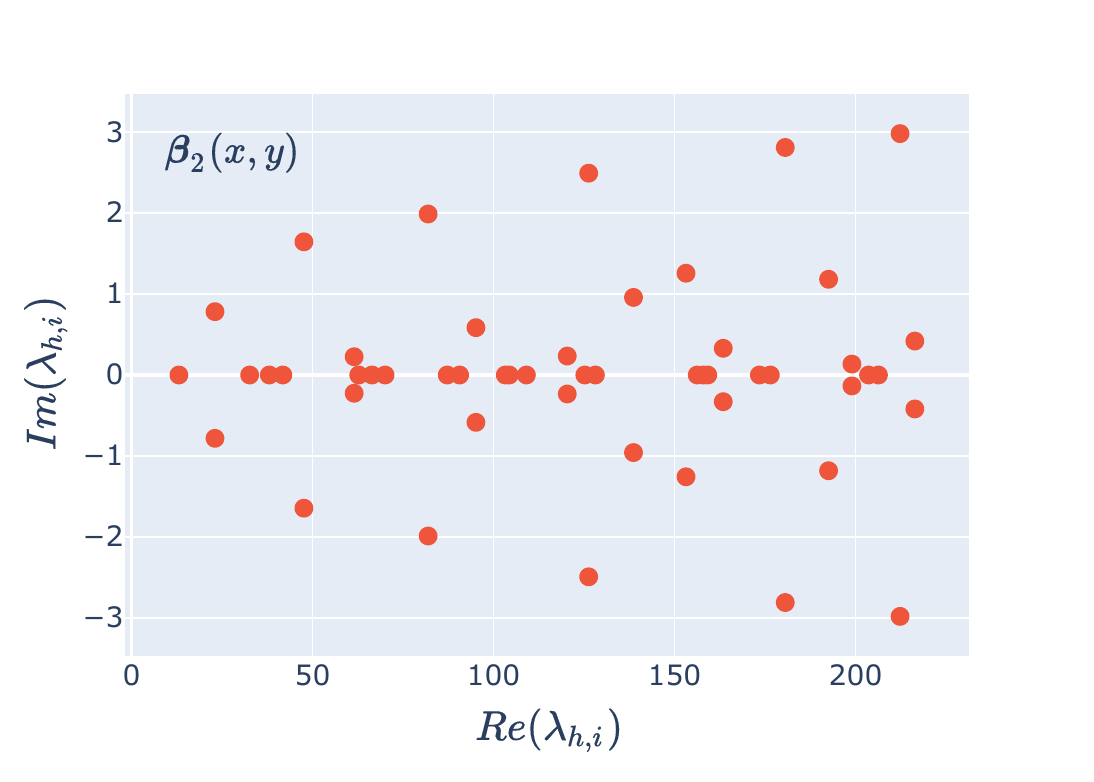}
	\end{minipage}\\
	\begin{minipage}{0.49\linewidth}\centering
		\includegraphics[scale=0.35, trim=0cm 0cm 1.8cm 1.2cm,clip]{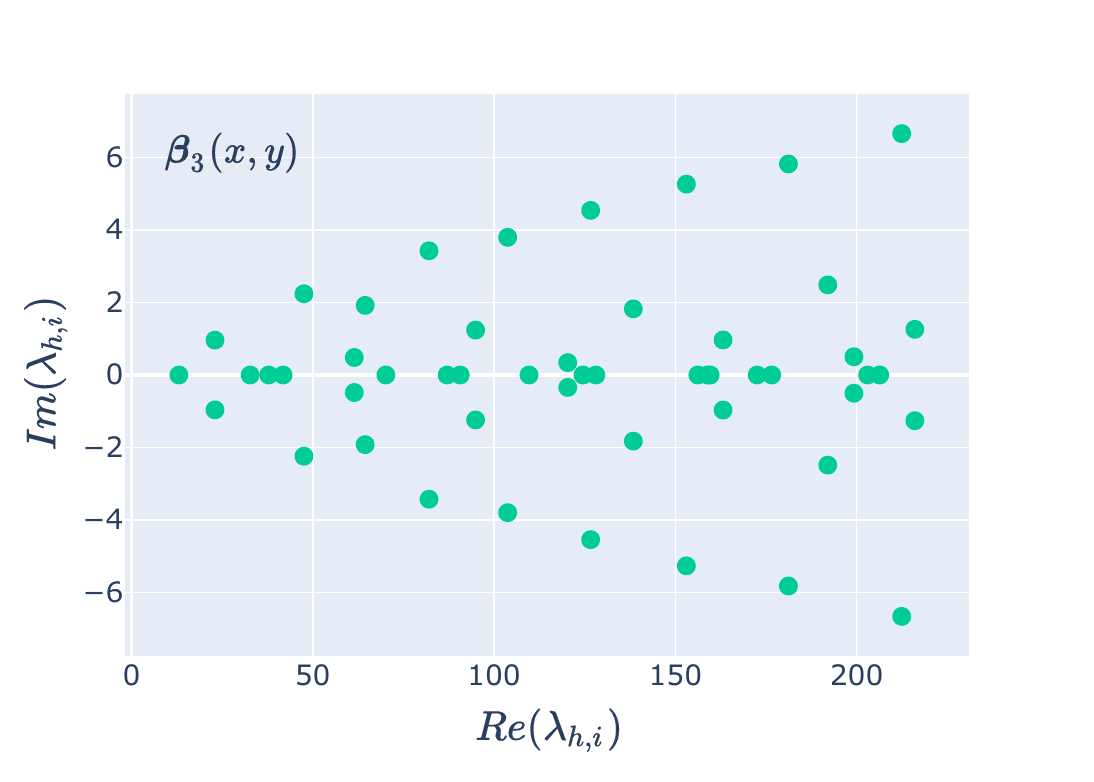}
	\end{minipage}
	\begin{minipage}{0.49\linewidth}\centering
		\includegraphics[scale=0.35, trim=0cm 0cm 1.8cm 1.2cm,clip]{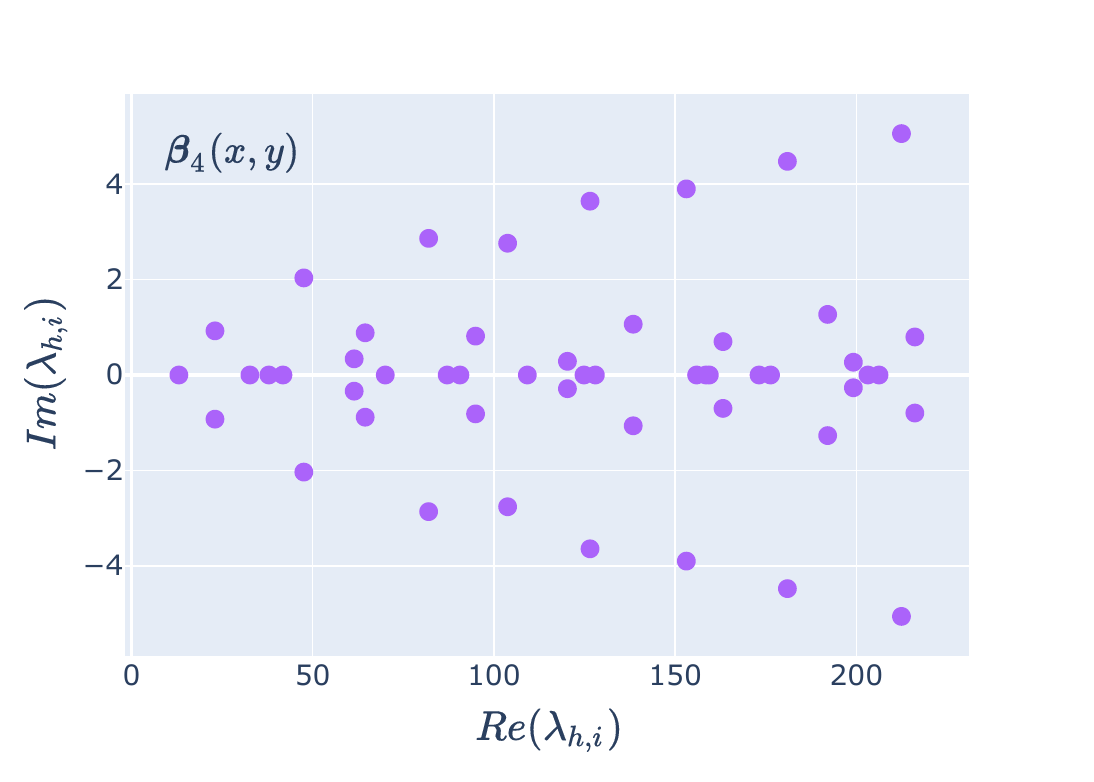}
	\end{minipage}\\
	\caption{Test \ref{subsec:square-domain2D}. Computed eigenvalues distribution on the square domain with different choices of $\boldsymbol{\beta}$ and $N=100$.}
	\label{fig:square-domain-complejos}
\end{figure}

\subsubsection{L-shaped domain}\label{subsec:lshape-2D}
We now focus on the L-shaped domain, defined by $\Omega:=(-1,1)^1\backslash(-1,0)^2$. Here, we have a singularity at $(0,0)$ which yield to a suboptimal rate of convergence for the eigenvalues.  We study the convergence and the spectrum behavior when changing the convective velocity.

Convergence results for the simplest case $\boldsymbol{\beta}_1$ are presented in Table \ref{table-lshape2D-beta10}. It notes that, because of the singularity, the first eigenfunction is singular, hence we expect at least $\mathcal{O}(h^{2\min\{r,k+1\}})$, with $2r\geq 1.08$. The observed experimental rate for $\lambda_{h,1}$ is approximately $\mathcal{O}(h^{1.7})$. The rest of eigenvalues show an optimal rate of convergence. On Table \ref{table-lshape2D-BDM_variablebeta} we have the first fourth computed eigenvalues for different values of $\boldsymbol{\beta}_1$ using Raviart-Thomas elements. The rates of convergence are similar to that of Table \ref{table-lshape2D-beta10}. We complement the test by plotting the discrete velocity and pressure (Figure \ref{fig:lshape-2D-uhs}). 
A clear right-shift is observed, with high gradients for the pressure and pseudostress near the re-entrant corner.

We now pass to the study of the complex spectrum. The distribution of the first 50 real and complex eigenvalues are plotted in Figure \ref{fig:lshape-domain-complejos}. We note that for $\boldsymbol{\beta}_1$ there are only two complex conjugate eigenvalues. This is similar to the results of the square domain. In contrast, the complex spectrum for the rest of convective velocities is considerable lower. For instance, only two complex eigenvalues for $\boldsymbol{\beta}_2$ were computed, and six for $\boldsymbol{\beta}_4$. Because of the non-symmetric nature of the problem, the uniformity of the mesh and number of elements, we could expect different complex distribution on other cases.
\begin{table}[hbt!]
	\centering 
	{\footnotesize
		\begin{center}
			\caption{Example \ref{subsec:lshape-2D}. Convergence behavior of the first four lowest computed eigenvalues on the lshaped domain with convective velocity $\boldsymbol{\beta}=\boldsymbol{\beta}_1$.}
			\begin{tabular}{|c| c c c c|}
				\hline
				\hline
				\texttt{dof}   &    $\lambda_{h,1}$  &   $\lambda_{h,2}$   &   $\lambda_{h,3}$  &   $\lambda_{h,4}$  \\
				\hline
				\multicolumn{5}{|c|}{$\mathbb{RT}_0$}\\
				\hline
				 32080 & $   32.6886  $   & $   37.0604  $  & $   42.3640  $  & $   49.2095  $ \\
				 55890 & $   32.7644  $   & $   37.0853  $  & $   42.3757  $  & $   49.2280  $ \\
				 87680 & $   32.8070  $   & $   37.0966  $  & $   42.3812  $  & $   49.2368  $ \\
				126870 & $   32.8302  $   & $   37.1035  $  & $   42.3844  $  & $   49.2416  $ \\
				Order &                 1.61    &                 2.06    &                 2.02    &                 2.08 \\
				Extrap&   $   32.9007  $   & $   37.1170  $  & $   42.3911  $  & $   49.2517  $ \\
				\hline
				\multicolumn{5}{|c|}{$\mathbb{BDM}_1$}\\
				\hline
				 51424 & $   32.9375  $   & $   37.1790  $  & $   42.4697  $  & $   49.3601  $ \\
				 89552 & $   32.9377  $   & $   37.1533  $  & $   42.4381  $  & $   49.3145  $ \\
				140448 & $   32.9384  $   & $   37.1405  $  & $   42.4226  $  & $   49.2923  $ \\
				203184 & $   32.9389  $   & $   37.1337  $  & $   42.4144  $  & $   49.2804  $ \\
				Order &                 1.76    &                 1.98    &                 2.02    &                 2.03 \\
				Extrap&   $   32.9395$   & $   37.1180  $  & $   42.3960  $  & $   49.2541  $ \\
				\hline
				\hline             
			\end{tabular}
	\end{center}}
	\smallskip
	
	\label{table-lshape2D-beta10}
\end{table}

\begin{table}[hbt!]
	\centering 
	{\footnotesize
		\begin{center}
			\caption{Example \ref{subsec:square-domain2D}. Convergence behavior of the first four lowest computed eigenvalues on the square domain with $\mathbb{RT}_0$ elements for the velocity field and different choices of $\boldsymbol{\beta}$.  }
			\begin{tabular}{|c| c c c c|}
				\hline
				\hline
				\texttt{dof}   &    $\lambda_{h,1}$  &   $\lambda_{h,2}$   &   $\lambda_{h,3}$  &   $\lambda_{h,4}$  \\
				\hline
				\multicolumn{5}{|c|}{$\boldsymbol{\beta}_2(x,y)$}\\
				\hline
				32080 & $   32.1199  $   & $   36.9210  $  & $   42.1028  $  & $   49.0400  $ \\
				55890 & $   32.1938  $   & $   36.9468  $  & $   42.1145  $  & $   49.0584  $ \\
				87680 & $   32.2352  $   & $   36.9586  $  & $   42.1199  $  & $   49.0672  $ \\
				126870 & $   32.2579  $   & $   36.9656  $  & $   42.1231  $  & $   49.0716  $ \\
				Order &                 1.60    &                 2.06    &                 2.05    &                 2.16 \\
				Extrap&   $   32.3273  $   & $   36.9798  $  & $   42.1296  $  & $   49.0809  $ \\
				\hline
				\multicolumn{5}{|c|}{$\boldsymbol{\beta}_3(x,y)$}\\
				\hline
				32080 & $   32.1984  $   & $   36.7792  $  & $   42.1642  $  & $   49.1085  $ \\
				55890 & $   32.2729  $   & $   36.8039  $  & $   42.1745  $  & $   49.1256  $ \\
				87680 & $   32.3147  $   & $   36.8153  $  & $   42.1793  $  & $   49.1339  $ \\
				126870 & $   32.3373  $   & $   36.8222  $  & $   42.1820  $  & $   49.1381  $ \\
				Order &                 1.62    &                 2.04    &                 2.12    &                 2.13 \\
				Extrap&   $   32.4058  $   & $   36.8359  $  & $   42.1873  $  & $   49.1470  $ \\
				\hline
				\multicolumn{5}{|c|}{$\boldsymbol{\beta}_4(x,y)$}\\
				\hline
				32080 & $   32.1601  $   & $   36.8090  $  & $   42.1161  $  & $   49.0450  $ \\
				55890 & $   32.2342  $   & $   36.8342  $  & $   42.1268  $  & $   49.0627  $ \\
				87680 & $   32.2754  $   & $   36.8459  $  & $   42.1316  $  & $   49.0711  $ \\
				126870 & $   32.2979  $   & $   36.8530  $  & $   42.1343  $  & $   49.0756  $ \\
				Order &                 1.63    &                 2.01    &                 2.19    &                 2.11 \\
				Extrap&   $   32.3651  $   & $   36.8675  $  & $   42.1395  $  & $   49.0849  $ \\
				\hline
				\hline             
			\end{tabular}
	\end{center}}

	\smallskip
	
	\label{table-lshape2D-BDM_variablebeta}
\end{table}

\begin{figure}[!hbt]
	\centering
	\begin{minipage}{0.24\linewidth}\centering
		{$\bu_{h,1}$}\\
		\includegraphics[scale=0.09,trim=23cm 5cm 23cm 5cm,clip]{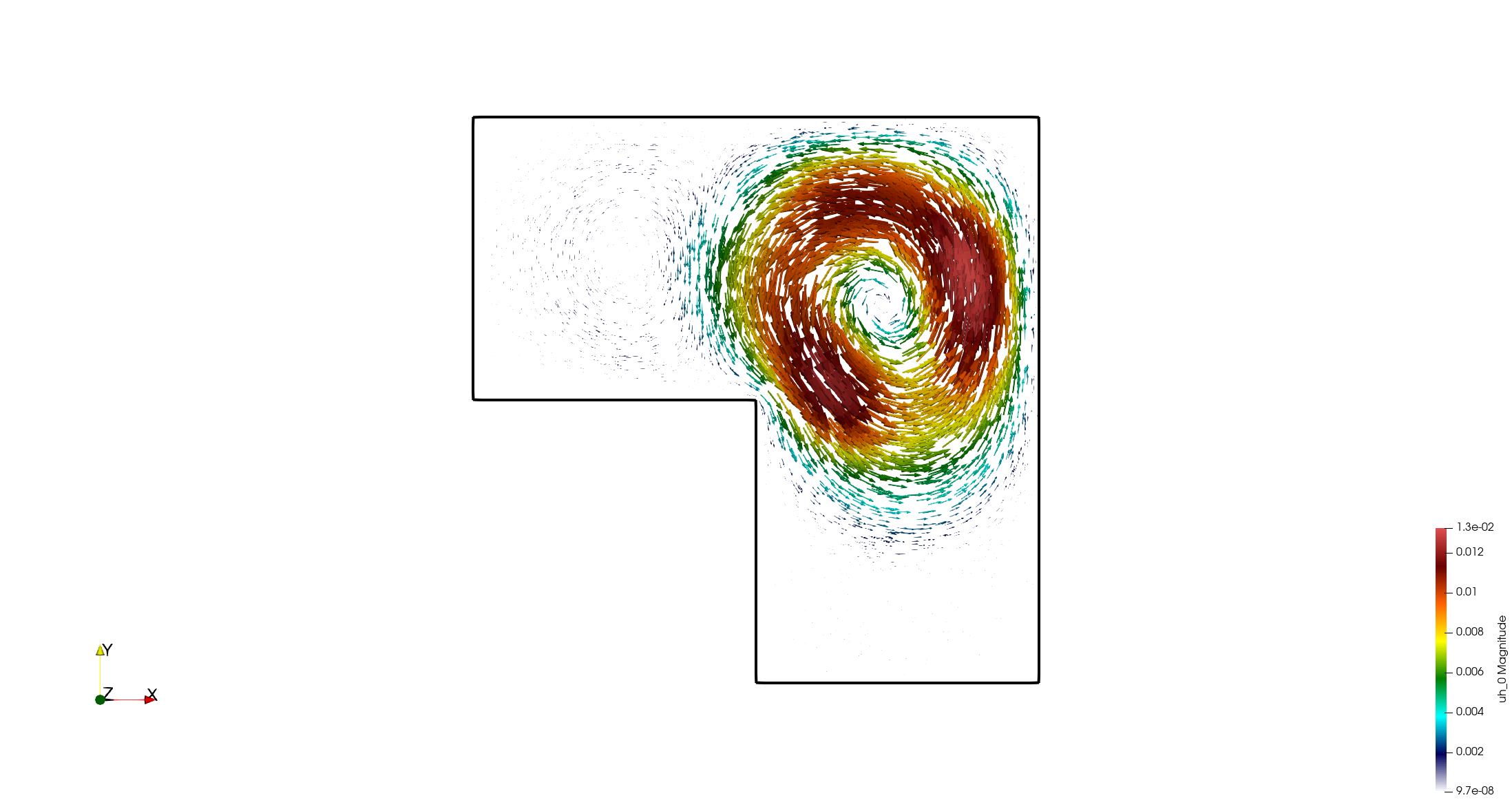}
	\end{minipage}
	\begin{minipage}{0.24\linewidth}\centering
		{$\bu_{h,4}$}\\
		\includegraphics[scale=0.09,trim=23cm 5cm 23cm 5cm,clip]{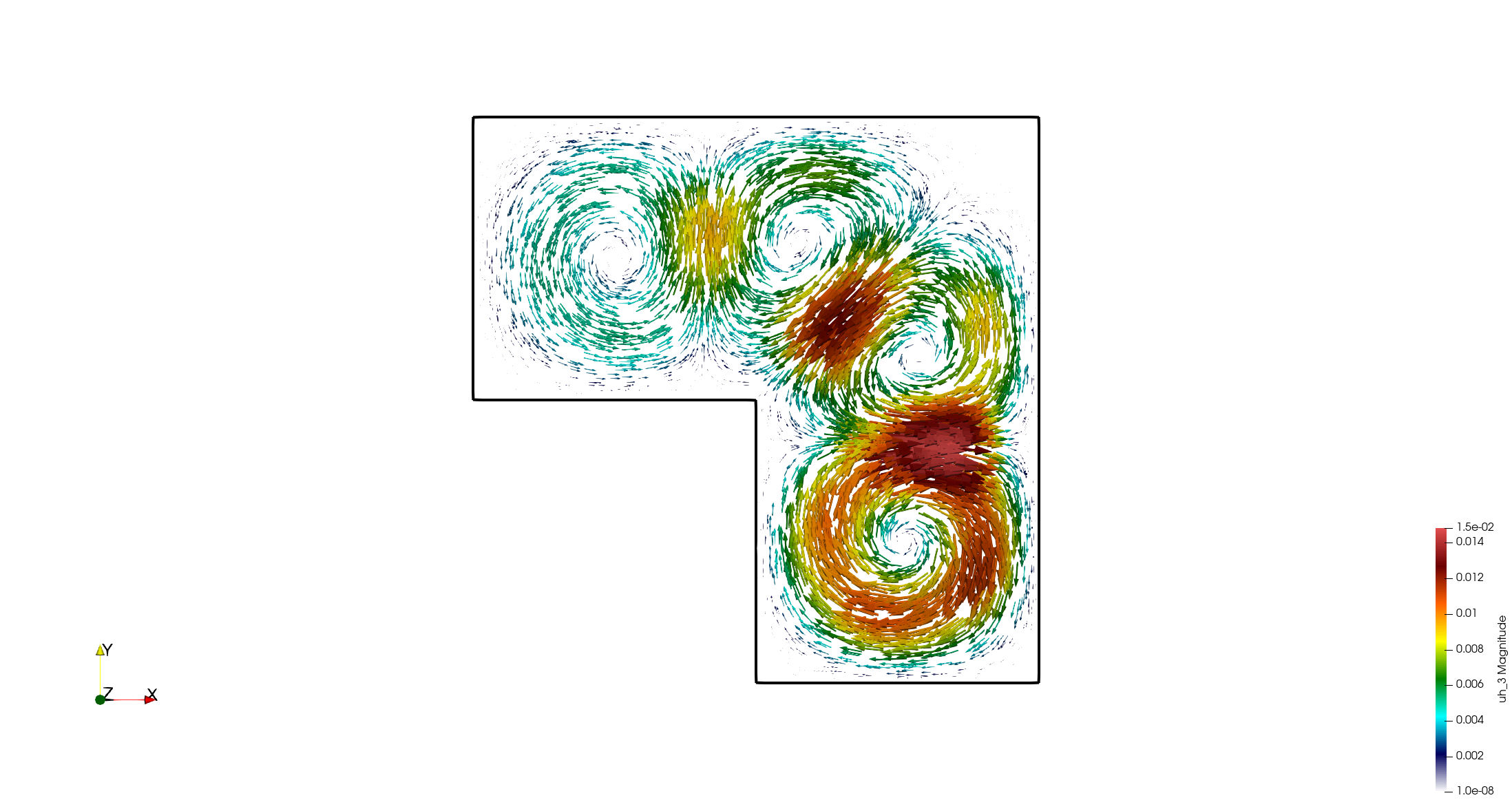}
	\end{minipage}
	\begin{minipage}{0.24\linewidth}\centering
		{$p_{h,1}$}\\
		\includegraphics[scale=0.09,trim=23cm 5cm 23cm 5cm,clip]{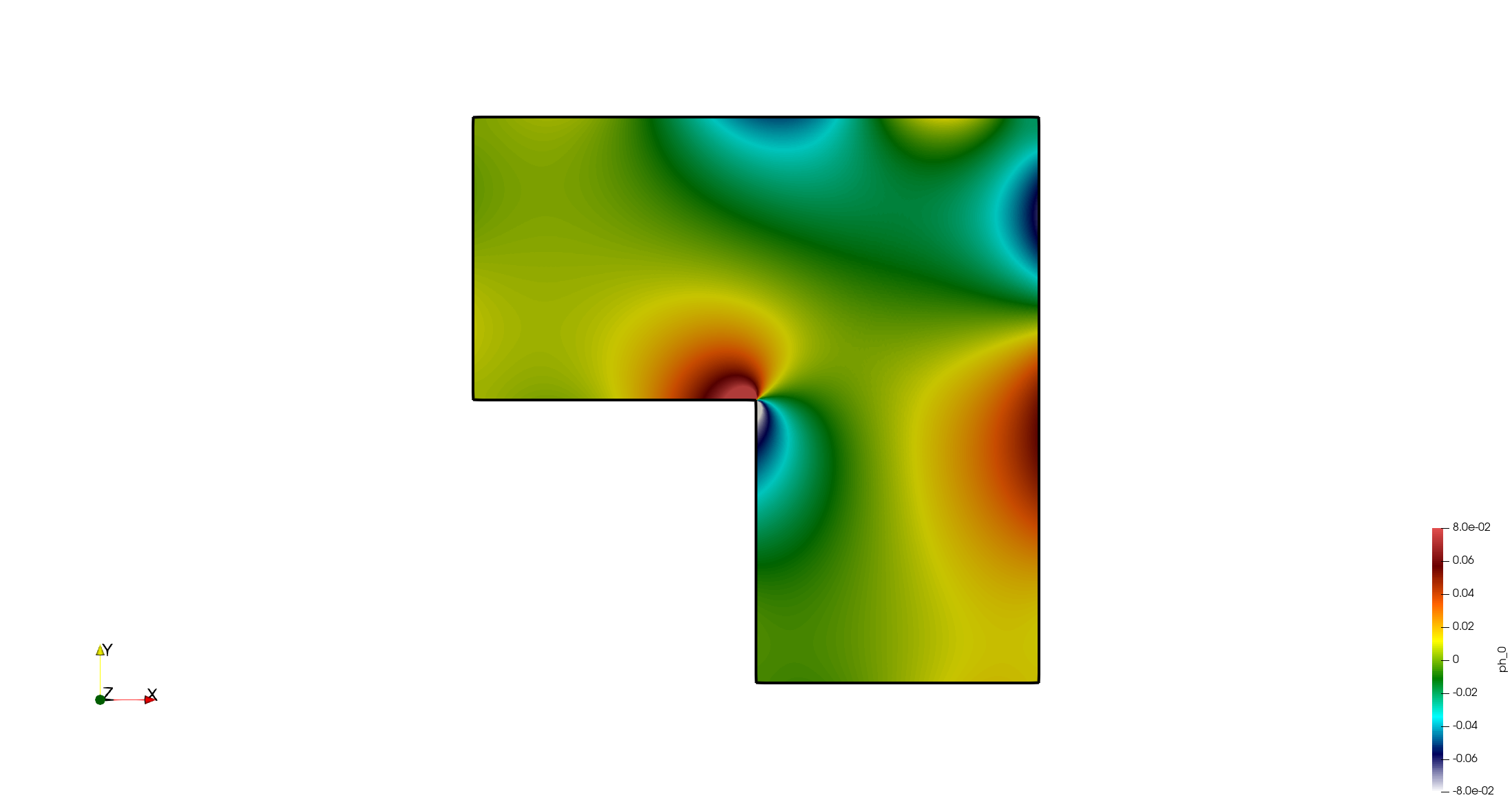}
	\end{minipage}
	\begin{minipage}{0.24\linewidth}\centering
		{$p_{h,4}$}\\
		\includegraphics[scale=0.09,trim=23cm 5cm 23cm 5cm,clip]{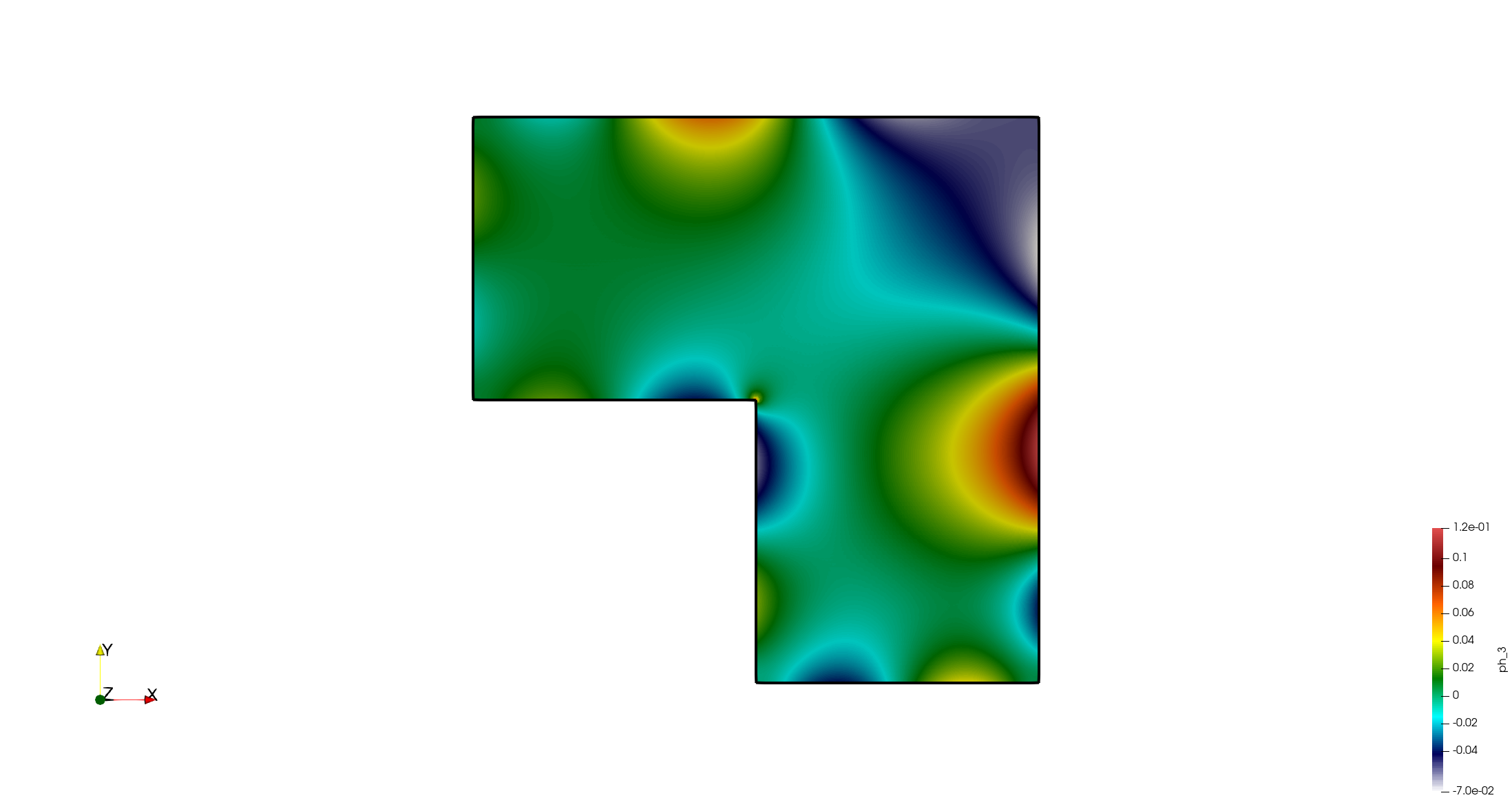}
	\end{minipage}\\
	\caption{Test \ref{subsec:lshape-2D}. Velocity streamlines and pressures surface plot for the first and fourth computed eigenvalue with $\boldsymbol{\beta}_1$.}
	\label{fig:lshape-2D-uhs}
\end{figure}

\begin{figure}[!hbt]\centering
	\begin{minipage}{0.49\linewidth}\centering
		\includegraphics[scale=0.35, trim=0cm 0cm 1.8cm 1.2cm, clip]{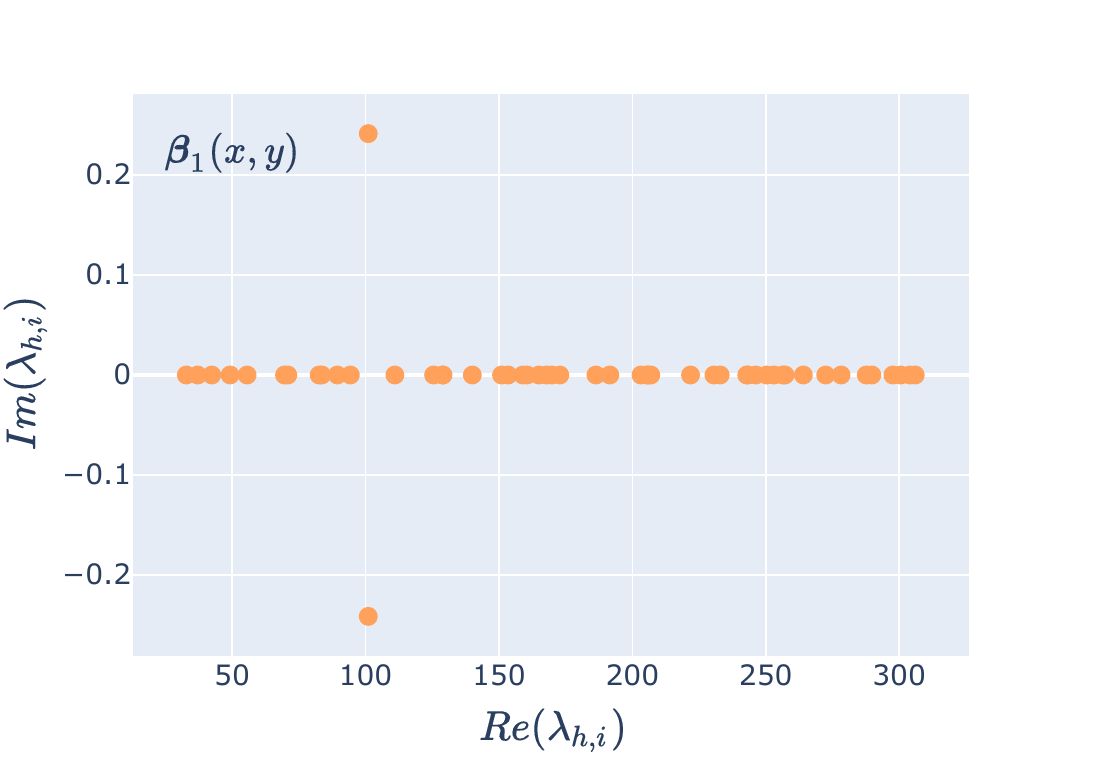}
	\end{minipage}
	\begin{minipage}{0.49\linewidth}\centering
		\includegraphics[scale=0.35, trim=0cm 0cm 1.8cm 1.2cm,clip]{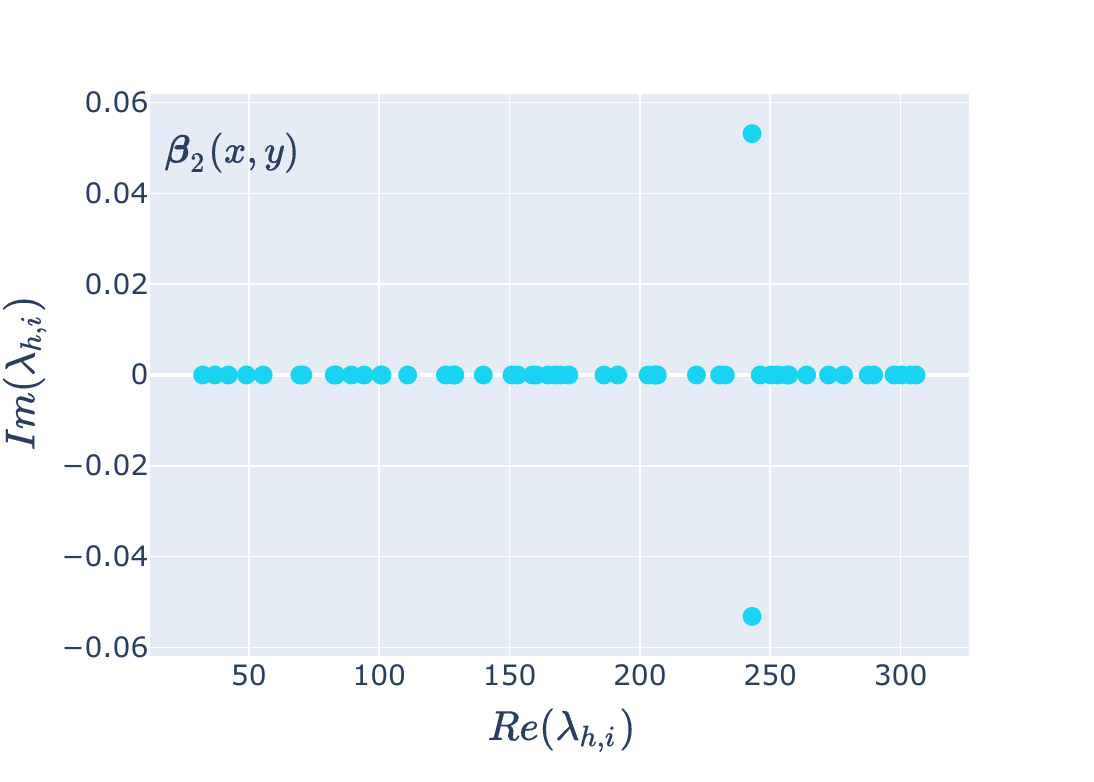}
	\end{minipage}\\
	\begin{minipage}{0.49\linewidth}\centering
		\includegraphics[scale=0.35, trim=0cm 0cm 1.8cm 1.2cm,clip]{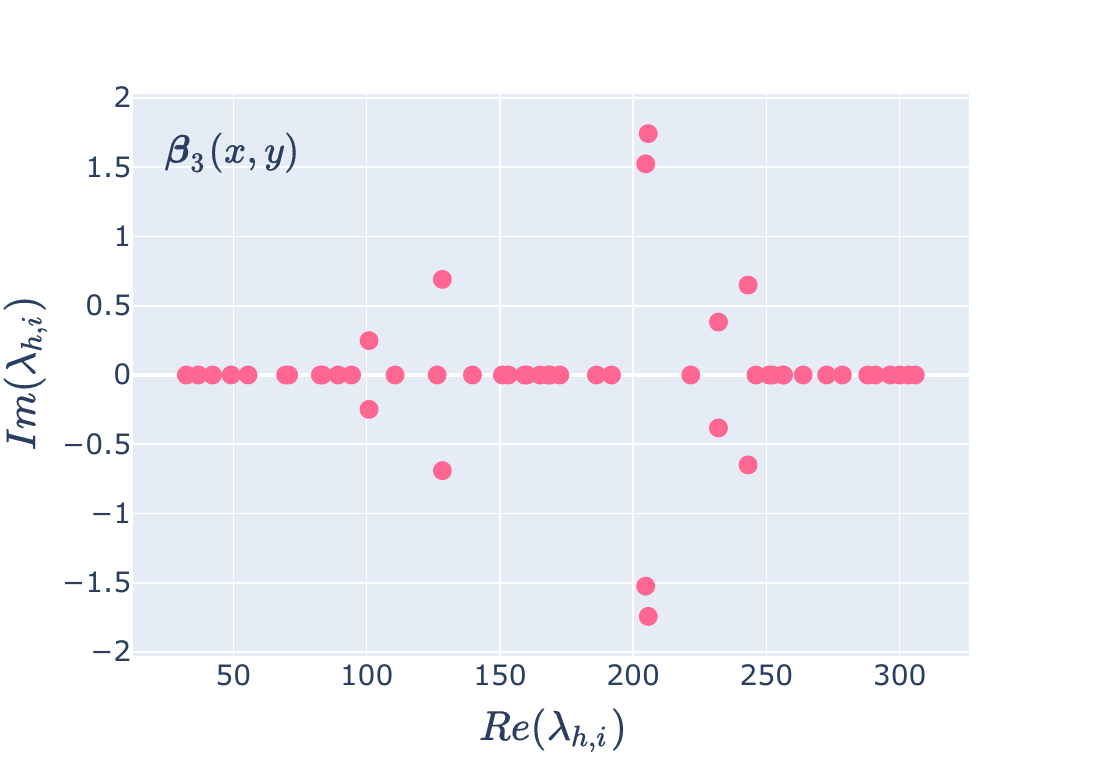}
	\end{minipage}
	\begin{minipage}{0.49\linewidth}\centering
		\includegraphics[scale=0.35, trim=0cm 0cm 1.8cm 1.2cm,clip]{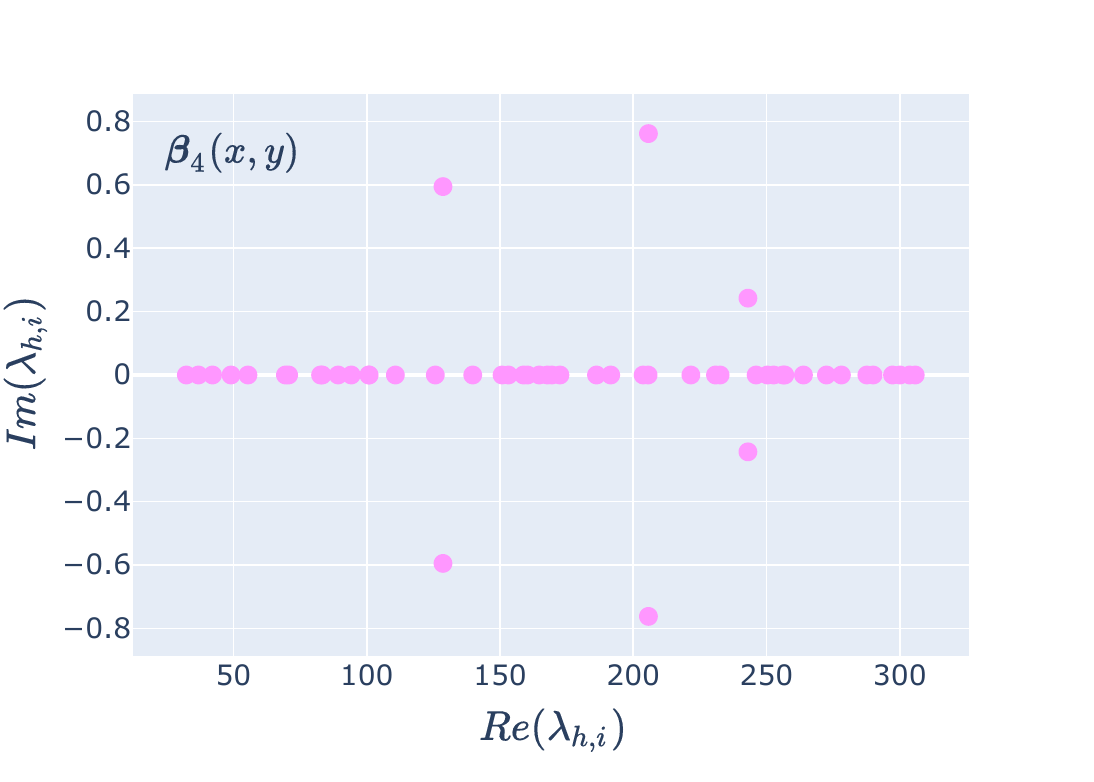}
	\end{minipage}\\
	\caption{Test \ref{subsec:lshape-2D}. Computed eigenvalues distribution on the lshape domain with different choices of $\boldsymbol{\beta}$ and $N=64$.}
	\label{fig:lshape-domain-complejos}
\end{figure}

\subsection{Results in 3D geometries}\label{subsec:3D-examples}
This section is devoted to test the scheme on three dimensional geometries. We consider polygonal and non-polygonal domains. For simplicity, the Cartesian canonical base vectors are considered as divergence-free convective velocities.

\subsubsection{The unit cube domain}\label{subsec:cube-domain3D}
Let us consider the domain $\Omega_c:=(0,1)^3$ together with homogeneous boundary conditions.  For this case, the mesh level $N$ scales as the number of cells such that the number of tetrahedrons is $6(N +1)^3$ and $N \sim  \texttt{dof}^{1/3}$.

By observing the results presented in Table \ref{table-cube3D-BDM} we note that the computational convergence rate is roughly $\mathcal{O}(h^{2k})\simeq\mathcal{O}(\texttt{dof}^{-2k/3})$. A good agreement with the reference values is also observed. Similar results were obtained with Raviart-Thomas elements, so we omit presenting the results. In Figure \ref{fig:uhs-cube3D} we show the $\boldsymbol{\beta}$-shifted velocity field for the first and fourth eigenmode, together with their respective pressure modes. 

\begin{table}[hbt!]
	\centering 
	{\footnotesize
		\begin{center}
			\caption{Example \ref{subsec:cube-domain3D}. Convergence behavior of the first four lowest computed eigenvalues on the unit cube domain with $\mathbb{BDM}_k$ elements and homogeneous boundary conditions.  We consider the field $\boldsymbol{\beta}=(1,0,0)^{\texttt{t}}$. }
			\begin{tabular}{|c c c c |c| c|c|}
				\hline
				\hline
				$N=20$             &  $N=24$         &   $N=28$         & $N=32$ & Order & $\lambda_{\text{extr}}$ & Ref. \cite{LEPE2024116959} \\ 
				\hline
				\multicolumn{7}{|c|}{$k=1$}\\
				\hline
				   62.8177  &    62.6992  &    62.6271  &    62.5801  & 1.97 &    62.4229 &    62.7468  \\
				63.0814  &    62.9678  &    62.8994  &    62.8552  & 2.03 &    62.7122 &    62.8363  \\
				63.1252  &    62.9991  &    62.9228  &    62.8733  & 2.00 &    62.7103 &    62.9327  \\
				92.7328  &    92.4753  &    92.3188  &    92.2168  & 1.98 &    91.8780 &    92.4487  \\
				\hline
				\multicolumn{7}{|c|}{$k=2$}\\
				\hline	
				
				62.6637  &    62.4426  &    62.4289  &    62.4266  & 3.83 &    62.4253 &    62.7468  \\
				62.9262  &    62.7259  &    62.7139  &    62.7118  & 3.88 &    62.7107 &    62.8363  \\
				62.9713  &    62.7293  &    62.7146  &    62.7120  & 3.85 &    62.7107 &    62.9327  \\
				92.5972  &    91.9344  &    91.8918  &    91.8843  & 3.77 &    91.8801 &    92.4487  \\					
				\hline
				\hline             
			\end{tabular}
	\end{center}}
	\smallskip
	
	\label{table-cube3D-BDM}
\end{table}

\begin{figure}[!hbt]\centering
\begin{minipage}{0.49\linewidth}\centering
	{$\bu_{h,1}$}\\
	\includegraphics[scale=0.1,trim=20cm 0cm 20cm 2cm,clip]{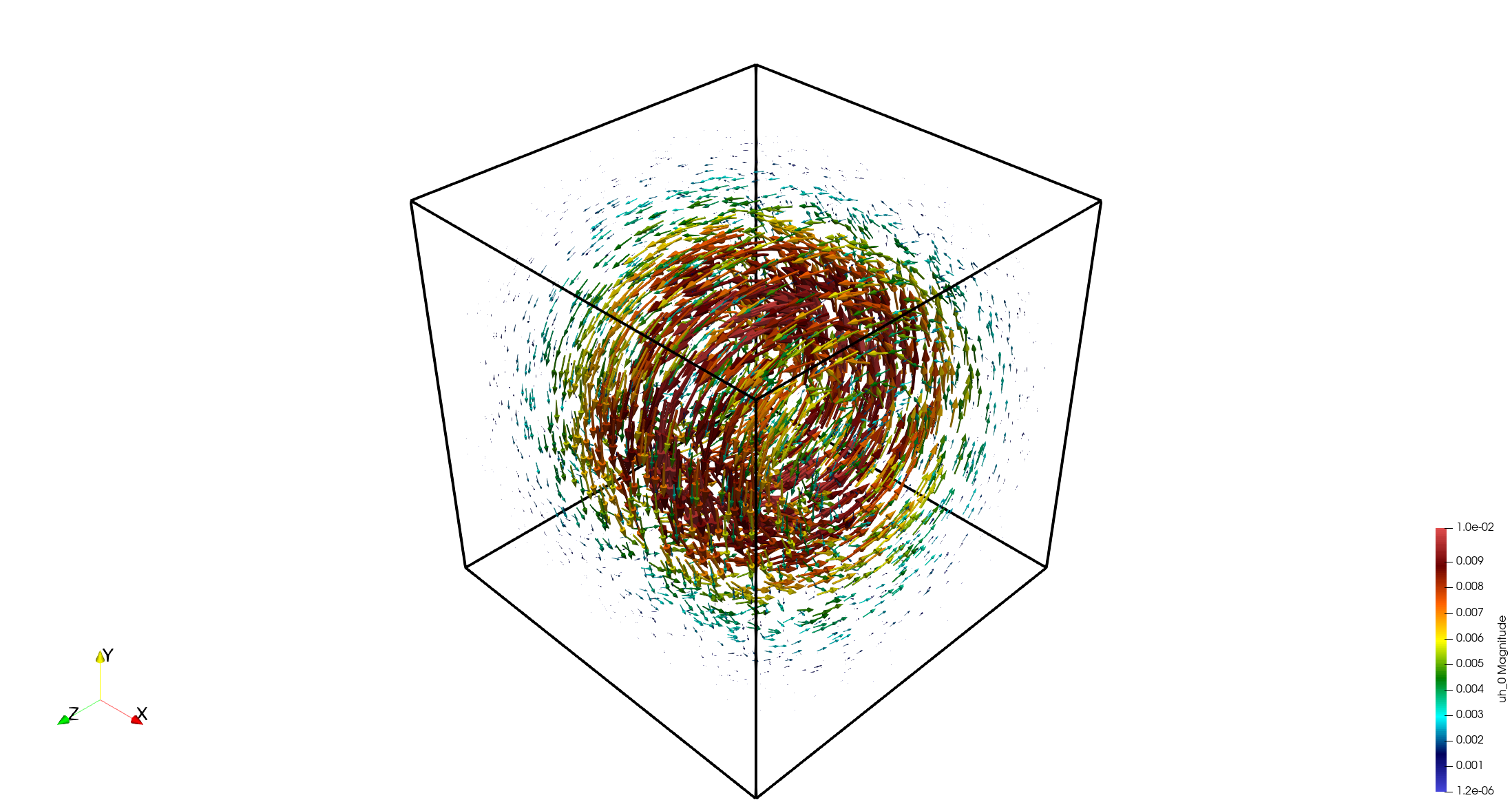}
\end{minipage}	
\begin{minipage}{0.49\linewidth}\centering
	{$\bu_{h,4}$}\\
	\includegraphics[scale=0.1,trim=20cm 0cm 20cm 2cm,clip]{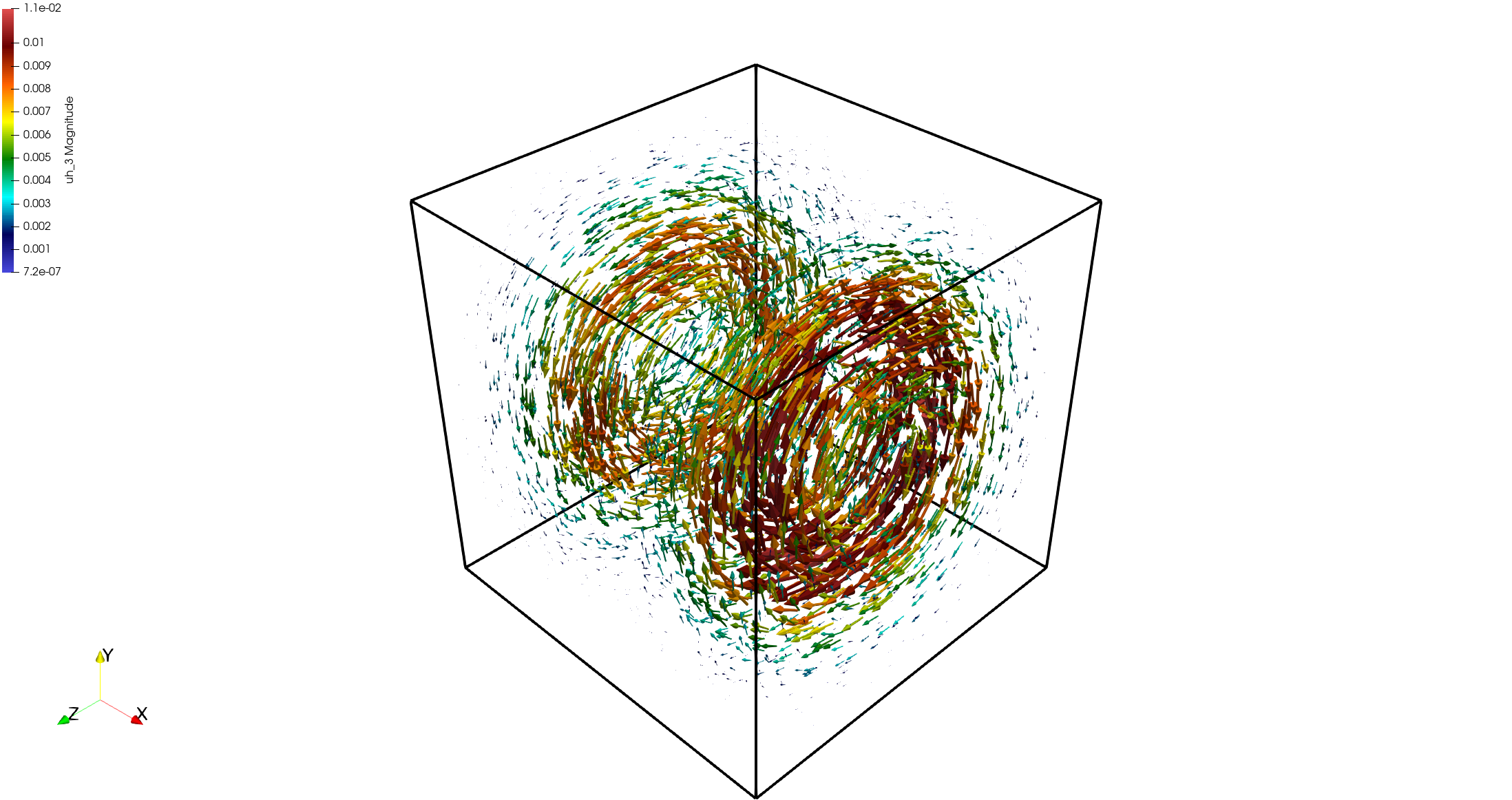}
\end{minipage}\\
\begin{minipage}{0.49\linewidth}\centering
	{$p_{h,1}$}\\
	\includegraphics[scale=0.1,trim=20cm 0cm 20cm 2cm,clip]{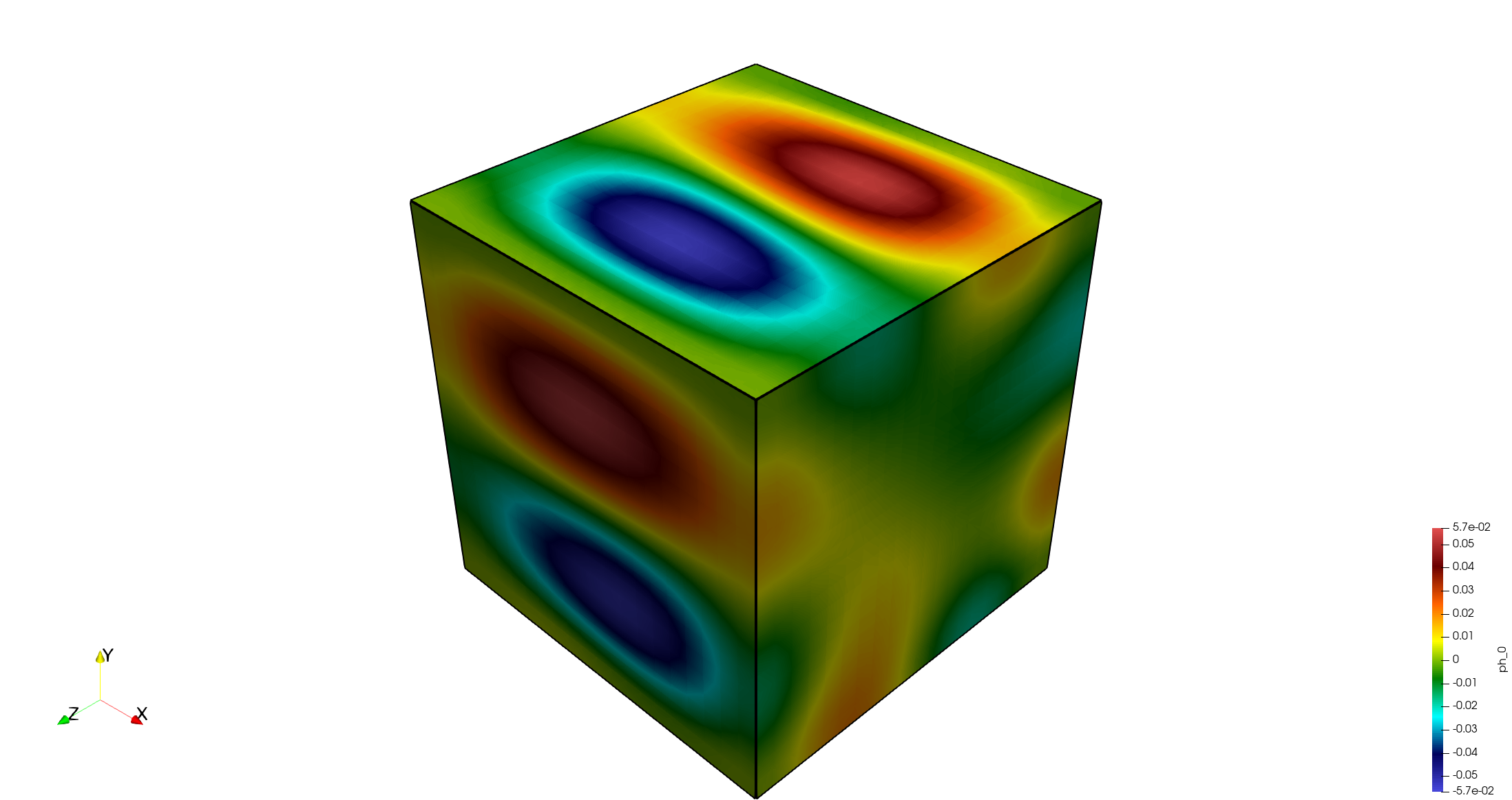}
\end{minipage}	
\begin{minipage}{0.49\linewidth}\centering
	{$p_{h,4}$}\\
	\includegraphics[scale=0.1,trim=20cm 0cm 20cm 2cm,clip]{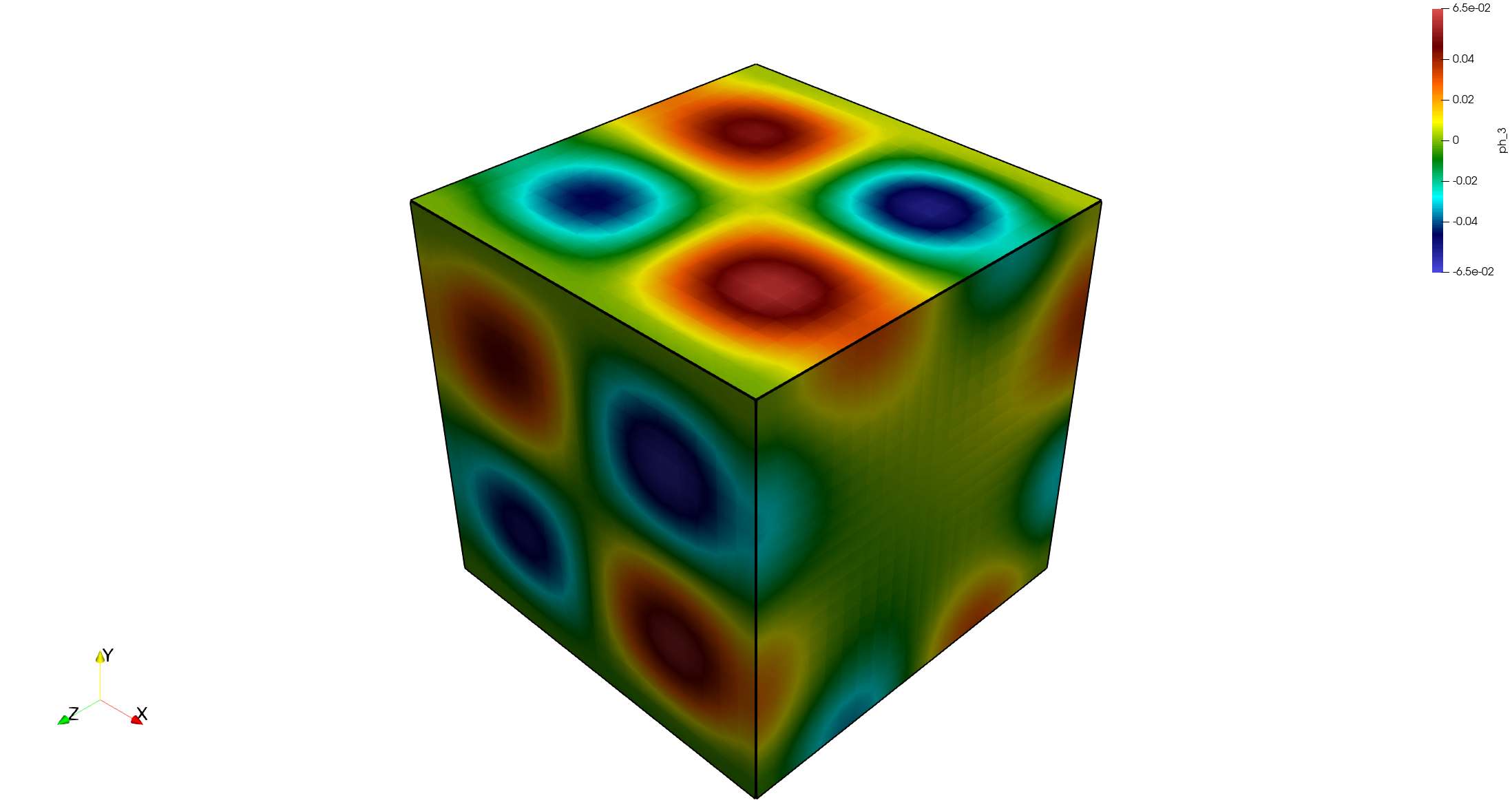}
\end{minipage}
\caption{Test \ref{subsec:cube-domain3D}. Velocity streamlines for the first and fourth eigenmode on the unit cube domain with convective velocity $\boldsymbol{\beta}=(1,0,0)^\texttt{t}$.}
\label{fig:uhs-cube3D}
\end{figure}

\subsubsection{A cylinder domain}\label{subsec:cylinder-domain}
In this experiment we test the scheme with a non-polygonal domain. We consider a cylinder domain defined by
$
\Omega:=(0,L)\times \Omega_c
$,
where $L=2$ and $\Omega_c:=\{(y,z)\in\mathbb{R}^2\,:\, y^2+z^2\leq 1/4\}$. The convective velocity is taken in the $x-$direction as $\boldsymbol{\beta}=(1,0,0)^\texttt{t}$. 

Table \ref{table:cylinder3D-n16} contains the resulting computed eigenfunctions together with their respective convergence rate. Unlike the cube domain, the first two eigenvalues are complex (including their conjugates). A rate of $\mathcal{O}(h^2)\simeq \mathcal{O}(\texttt{dof}^{-2/3})$ is observed, which is the best to expect if we approximate a cylinder domain with tetrahedrons. We recall that a second order mesh can be used for higher order schemes to avoid the variational crime. This is followed by plots of the real and complex velocity field pressures in Figure \ref{fig:cylinder-uh1}.  

	

\begin{table}[hbt!]
	\setlength{\tabcolsep}{4.5pt}
	\centering 
	\caption{Example \ref{subsec:cylinder-domain}. Convergence behavior of the Lowest computed eigenvalues on the cylinder domain with $\mathbb{BDM}_1$ elements and homogeneous boundary conditions. We consider the field $\boldsymbol{\beta}=(1,0,0)^{\texttt{t}}$.}
	\label{table:cylinder3D-n16}
	{\small\begin{tabular}{rcccc}
			\hline\hline
			\texttt{dof}  &    $\lambda_{h,1}$  &   $\lambda_{h,2}$   &   $\lambda_{h,3}$  &   $\lambda_{h,4}$  \\
			\hline
			 93378 & $   61.8787  \pm 1.0879i$   & $      61.9368  \pm1.1256i$ & $   64.1401  $   & $   72.2509  $   \\
			268596 & $   60.5428  \pm1.1911i$   & $   60.5622  \pm1.1990i$& $   62.6688  $   & $   70.3809  $    \\
			606303 & $   60.0615  \pm 1.2198i$     & $   60.0689  \pm 1.2265i$ & $   62.1428  $   & $   69.7285  $   \\
			1133067 & $   59.8299  \pm 1.2367i$    & $   59.8330  \pm 1.2379i$ & $   61.8899  $   & $   69.4062  $  \\
			Order &                 2.30       &                 2.32  		& 2.32  & 2.36    \\
			Extrap &   $   59.4813  \pm 1.2636i$    & $   59.4823  \pm1.2572i$ &   $   61.5148  $   & $   68.9540  $  \\
			\hline
			\hline
	\end{tabular}}
	\smallskip
\end{table}

\begin{figure}[!hbt]\centering
	\begin{minipage}{0.49\linewidth}\centering
		{\footnotesize $Re(\bu_{h,1})$}\\
	\includegraphics[scale=0.15,trim= 23cm 5cm 20cm 8cm, clip]{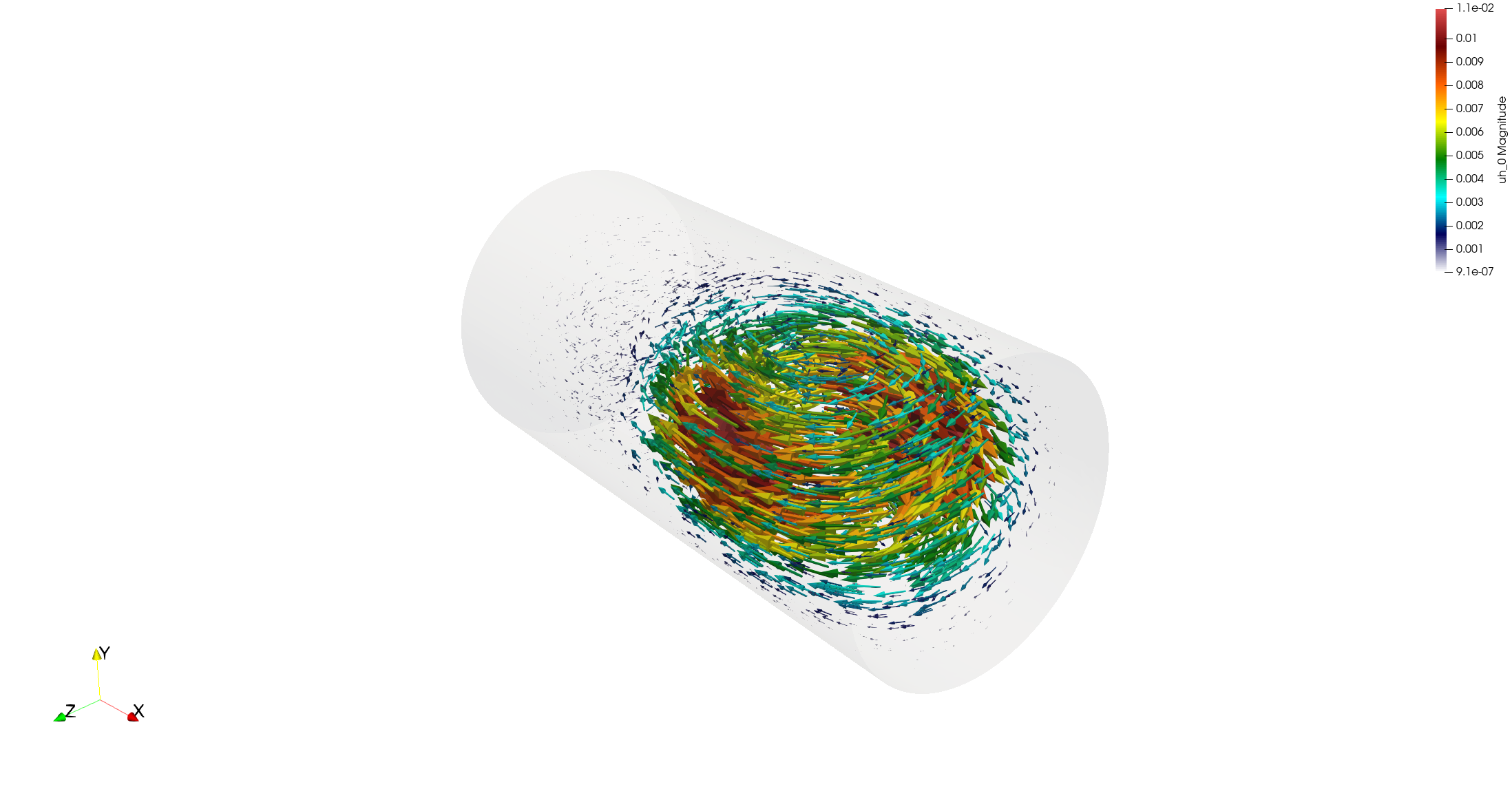}
	\end{minipage}
		\begin{minipage}{0.49\linewidth}\centering
			{\footnotesize $Im(\bu_{h,1})$}\\
		\includegraphics[scale=0.15,trim= 23cm 5cm 20cm 8cm, clip]{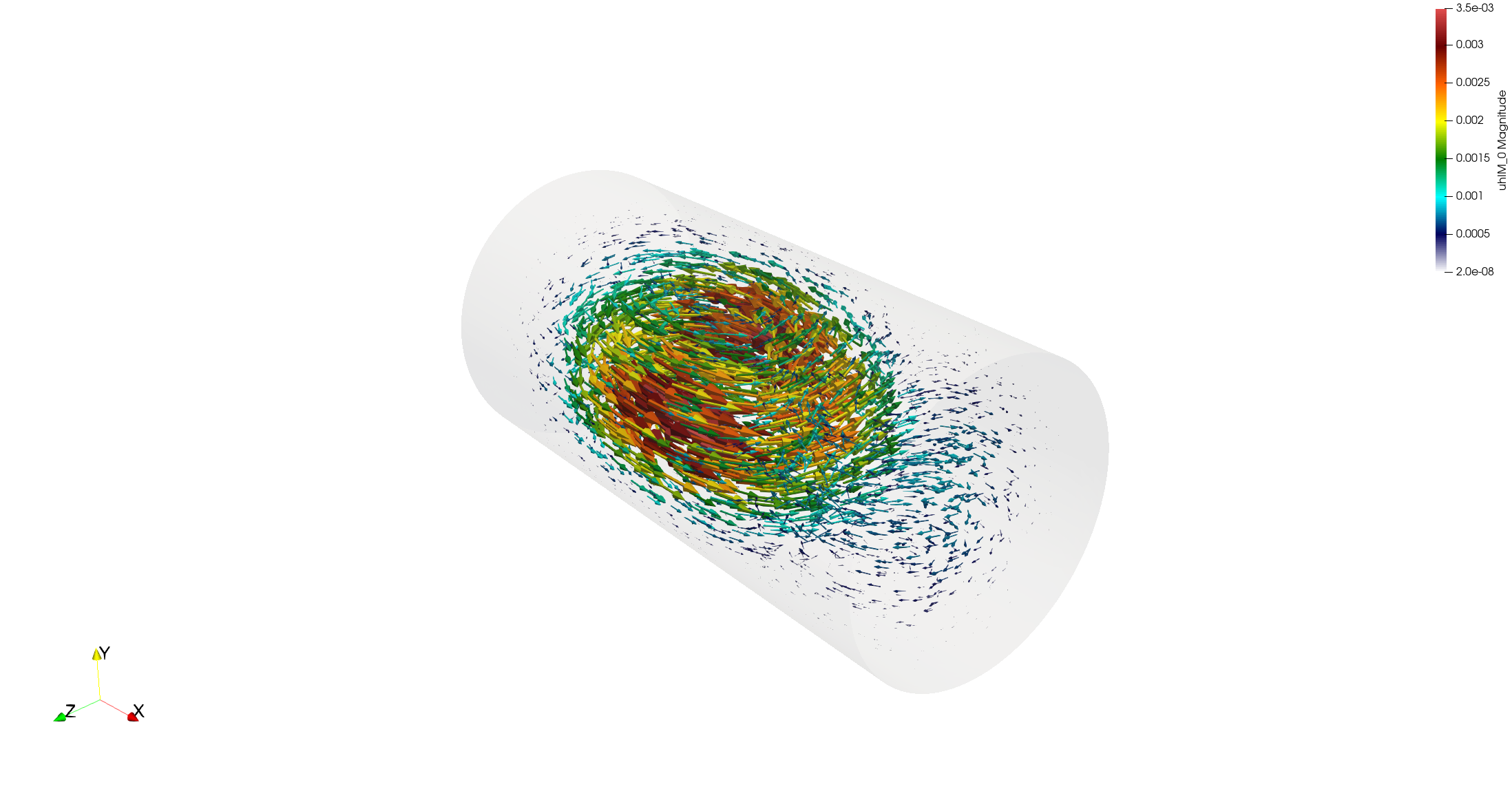}
	\end{minipage}\\
		\begin{minipage}{0.49\linewidth}\centering
			{\footnotesize $Re(p_{h,1})$}\\
		\includegraphics[scale=0.15,trim= 23cm 5cm 20cm 8cm, clip]{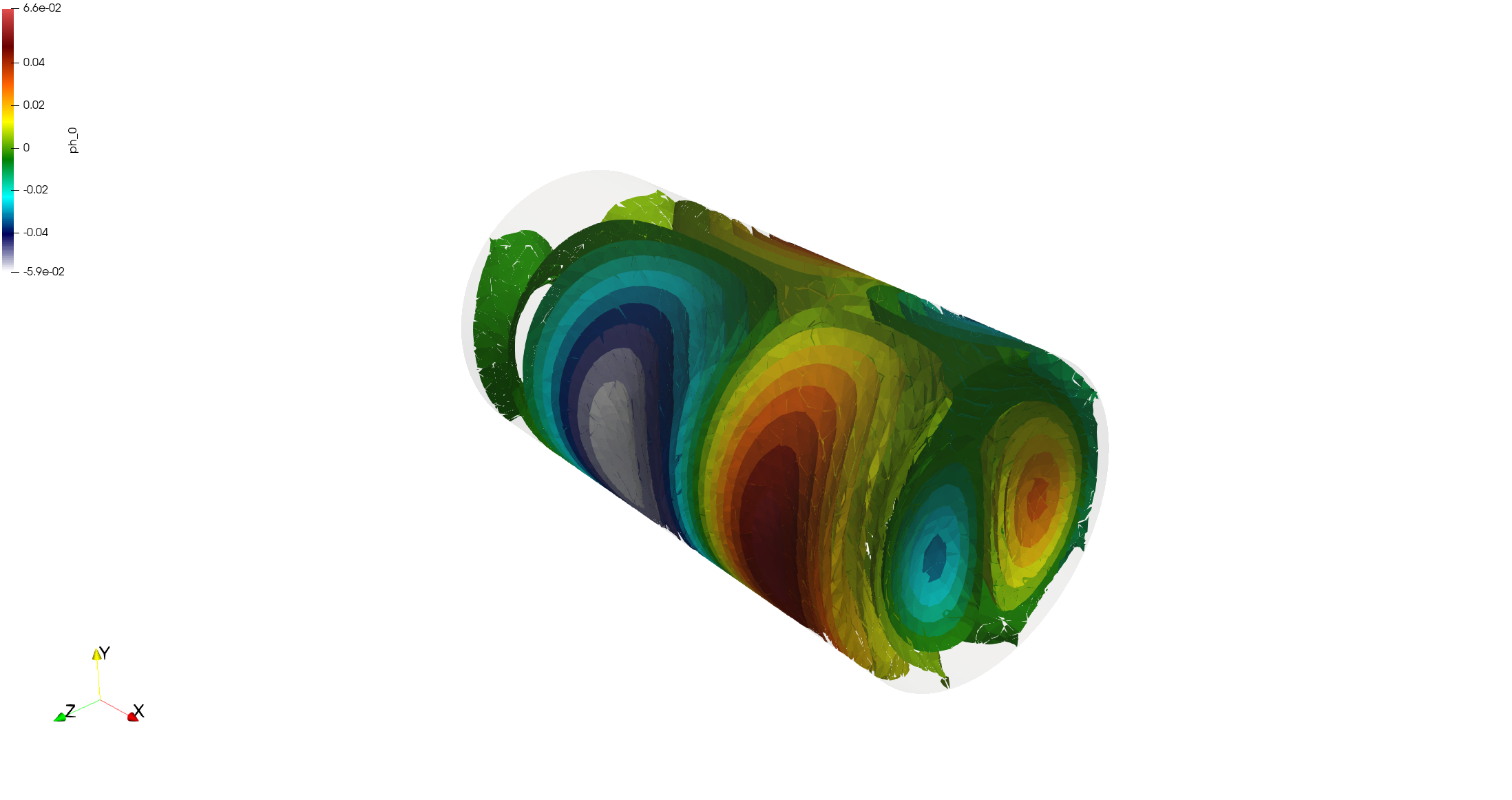}
	\end{minipage}
		\begin{minipage}{0.49\linewidth}\centering
			{\footnotesize $Im(p_{h,1})$}\\
		\includegraphics[scale=0.15,trim= 23cm 5cm 20cm 8cm, clip]{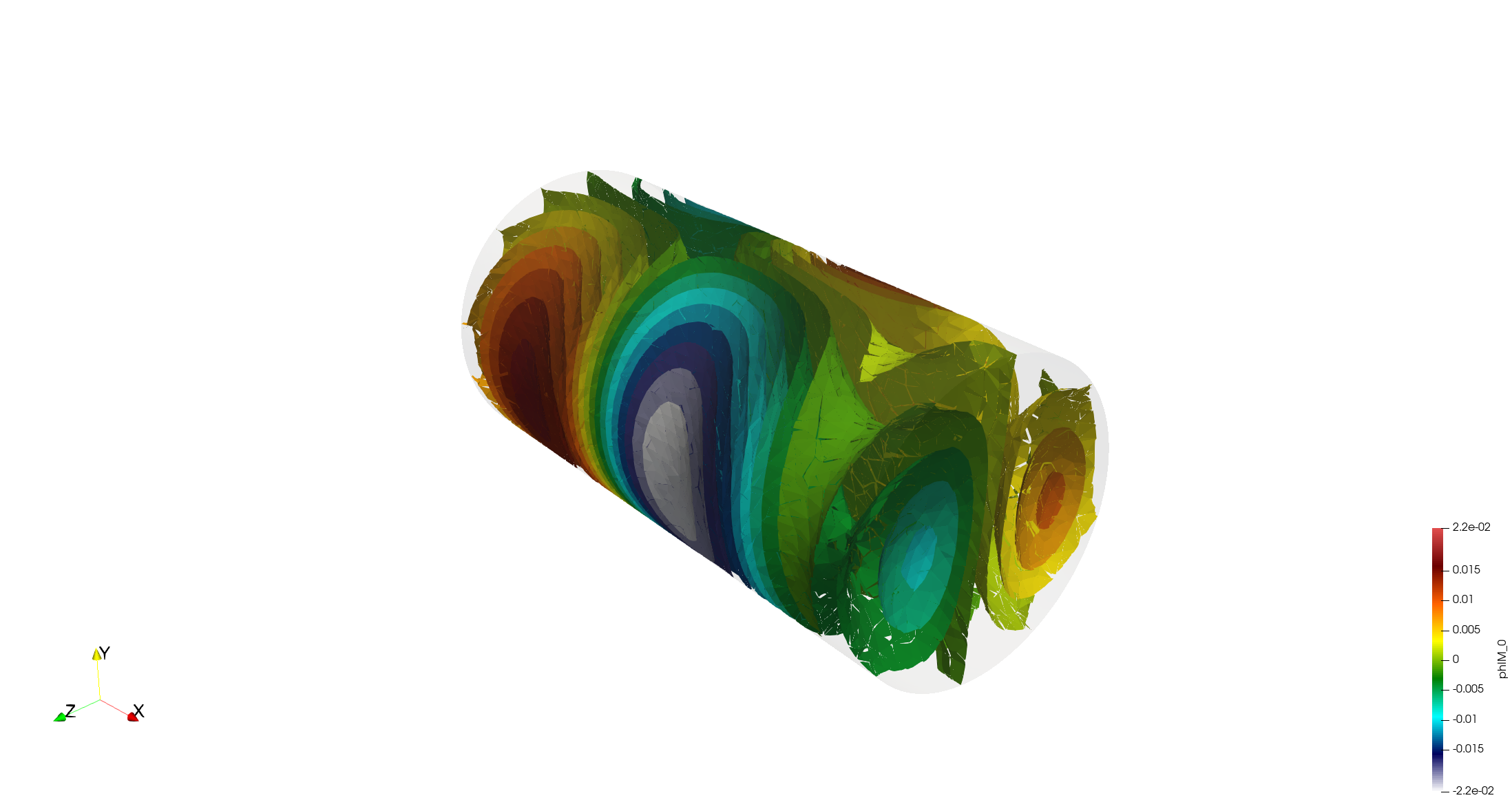}
	\end{minipage}\\
	\caption{Test \ref{subsec:cylinder-domain}. Velocity field and pressure contour plots for the first real and complex eigenmodes in the cylinder domain.}
	\label{fig:cylinder-uh1}
\end{figure}

\subsection{Robustness with respect to $\nu$}\label{subsec:robustness2D3D}
In this experiment we test the robustness of the scheme when we consider small values of $\nu$. According to \cite{John2016}, small values for $\nu$ are interesting for time depending problems. However, we aim to test the theoretical results in the fixed-point theory presented in Section \ref{sec:fixed-point-continuous} and \ref{subsec:the-discrete-eigenvalue-problem}. We consider the square domain from Section \ref{subsec:square-domain2D} and the cube domain from Section \ref{subsec:cube-domain3D}.  The convective velocity is given by $\boldsymbol{\beta}(x,y)=\nu(1,0)^\texttt{t}$ and  $\boldsymbol{\beta}(x,y)=\nu(1,0,0)^\texttt{t}$ for the square and the cube domain, respectively. Note that this yield to consider two case scenarios for this solenoidal vector: normalizing the field so that $\Vert\boldsymbol{\beta}\Vert_{\infty,\Omega}=1$, as stated at the beginning of the numerical section, and $ \Vert\boldsymbol{\beta}\Vert_{\infty,\Omega}=\nu$. Finally, the values for $\nu$ are taken to be $\nu=1/2^j$, for $j=0,...,15.$.  We present here the results with $\mathbb{BDM}_1$. Similar results were achieved by considering $\mathbb{RT}$.

In Figure \ref{fig:robustnes_nonorm} we present the results for the normalized solenoidal field. Although we observe a constant decrease in the computed eigenvalues,  we note that the scheme becomes unstable after the first two values for $\nu$, namely, after $\nu=1/2$.  We observe that the small data requirement can be broken if the fraction $\Vert \boldsymbol{\beta}\Vert_{\infty,\Omega}/\nu$ grows without control. By normalizing the vector, we have that the estimator depends on $1/\nu$, which leads us to the observed instability for small values of $\nu$, where the solver even yields negative eigenvalues.

On the other hand, in Figure \ref{fig:robustnes_norm} we have the opposite case. Note that due to the constant scaling $\Vert \boldsymbol{\beta}\Vert_{\infty,\Omega}/\nu=1$, the error rates remain constant and optimal, independent of the viscosity.  It is worth noting that both scenarios for $\boldsymbol{\beta}$ produce completely different spectrum.

\begin{figure}[!hbt]\centering
	\begin{minipage}{0.49\linewidth}
		\includegraphics[scale=0.35, trim=0cm 0cm 1.8cm 1.2cm, clip]{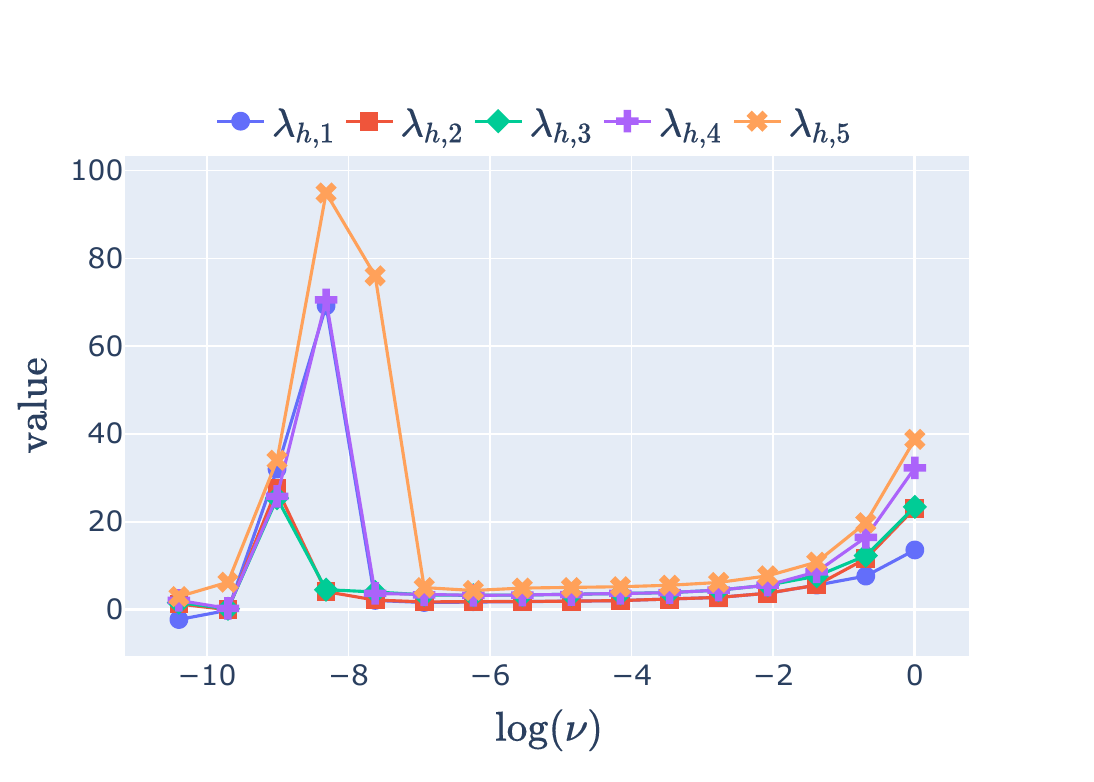}
	\end{minipage}
	\begin{minipage}{0.49\linewidth}
		\includegraphics[scale=0.35, trim=0cm 0cm 1.8cm 1.2cm, clip]{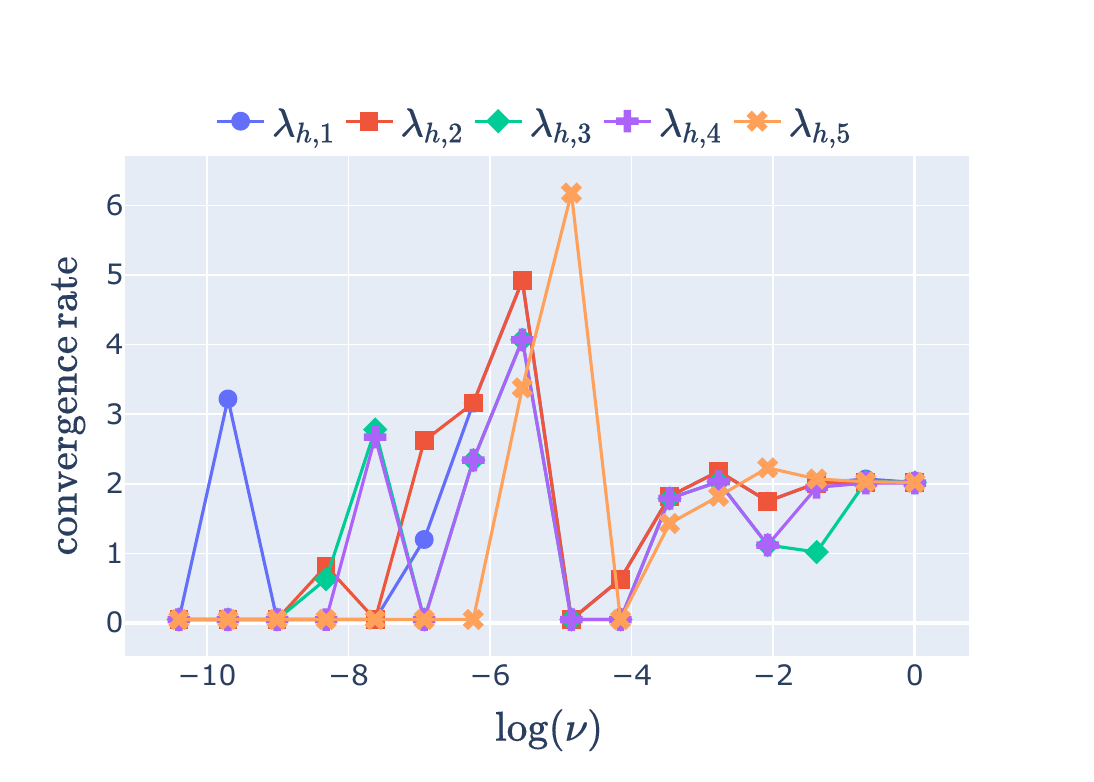}
	\end{minipage}\\
	\begin{minipage}{0.49\linewidth}
		\includegraphics[scale=0.35, trim=0cm 0cm 1.8cm 1.2cm, clip]{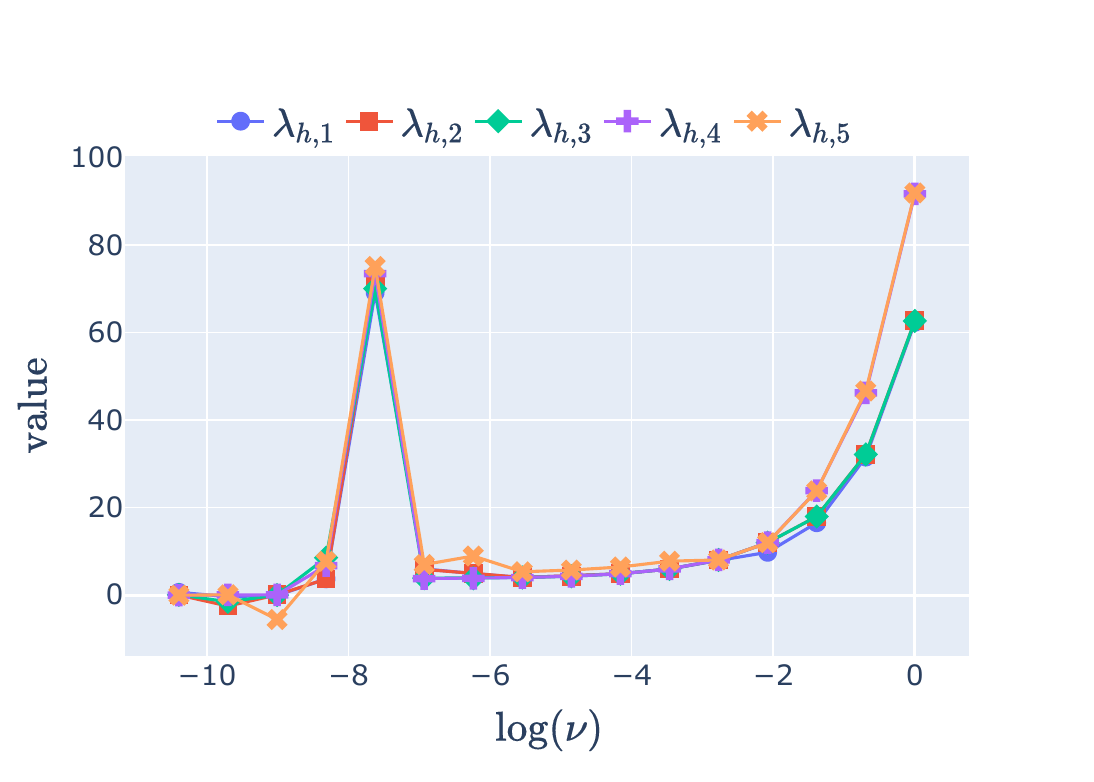}
	\end{minipage}
	\begin{minipage}{0.49\linewidth}
		\includegraphics[scale=0.35, trim=0cm 0cm 1.8cm 1.2cm, clip]{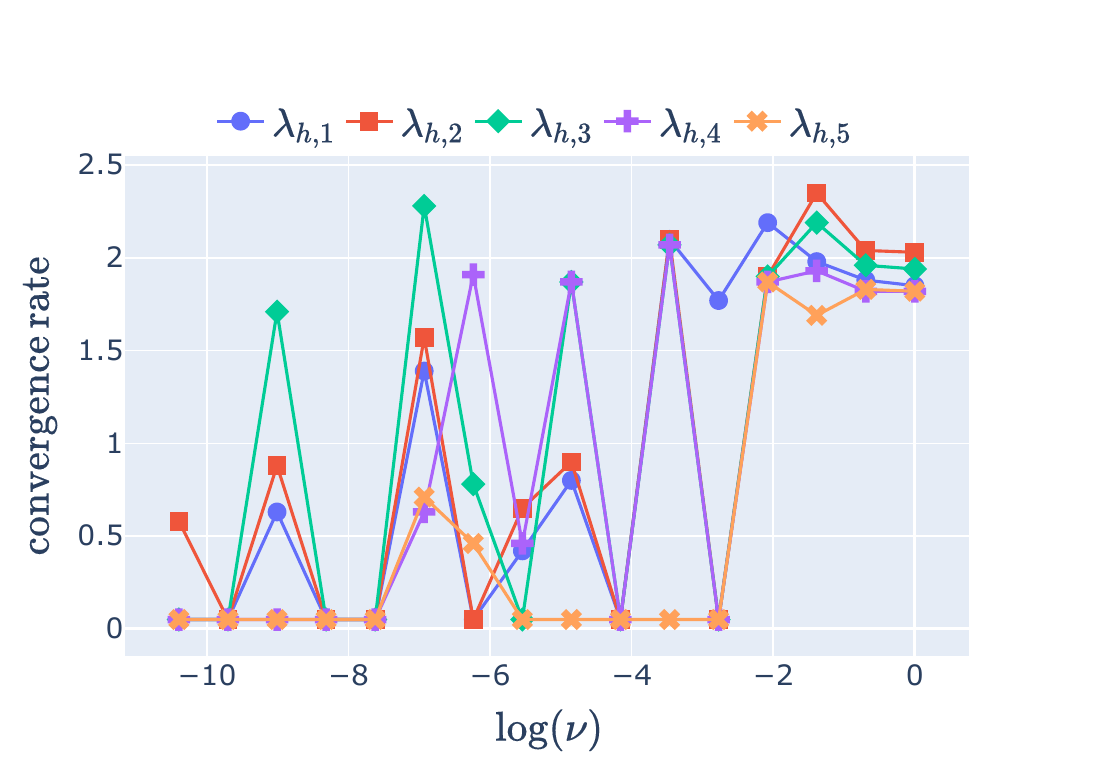}
	\end{minipage}\\
	\caption{Test \ref{subsec:robustness2D3D}. Behavior of the computed eigenvalues and the convergence rate with respect to the viscosity in the square domain (top row) and the unit cube domain (bottom row). Here, we consider $\Vert\boldsymbol{\beta}\Vert_{\infty,\Omega}=1$.  }
	\label{fig:robustnes_nonorm}
\end{figure}

\begin{figure}[!hbt]\centering
	\begin{minipage}{0.49\linewidth}
		\includegraphics[scale=0.35, trim=0cm 0cm 1.8cm 1.2cm, clip]{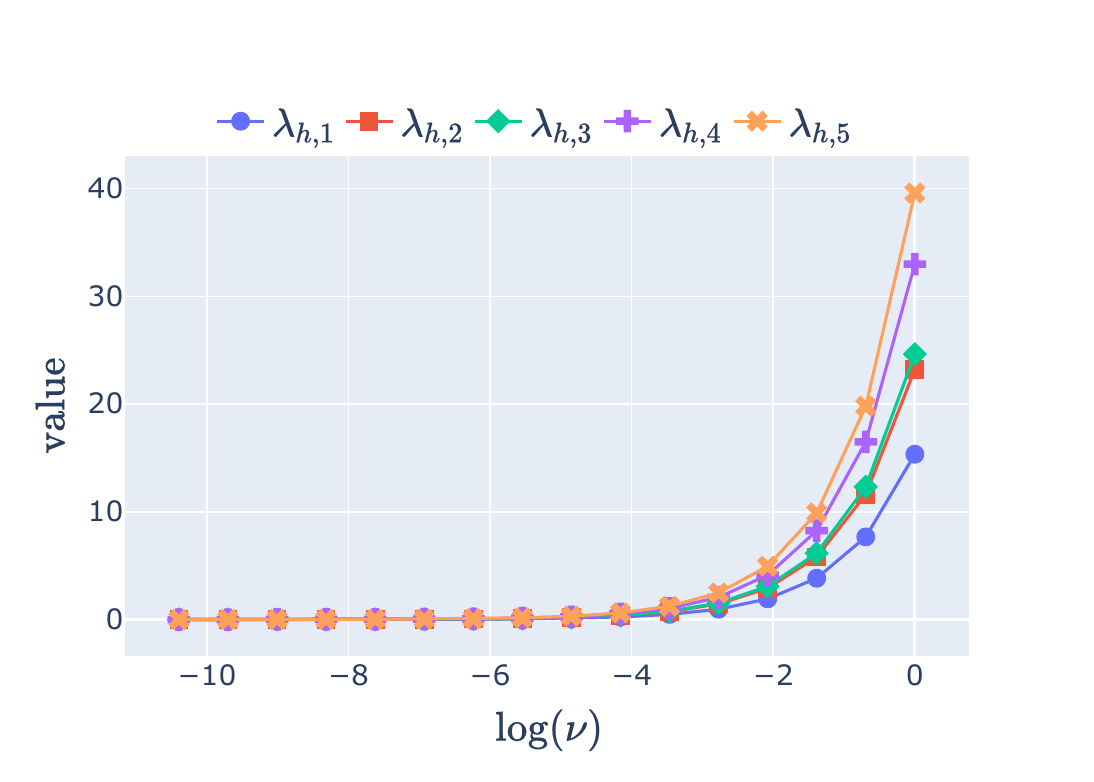}
	\end{minipage}
	\begin{minipage}{0.49\linewidth}
		\includegraphics[scale=0.35, trim=0cm 0cm 1.8cm 1.2cm, clip]{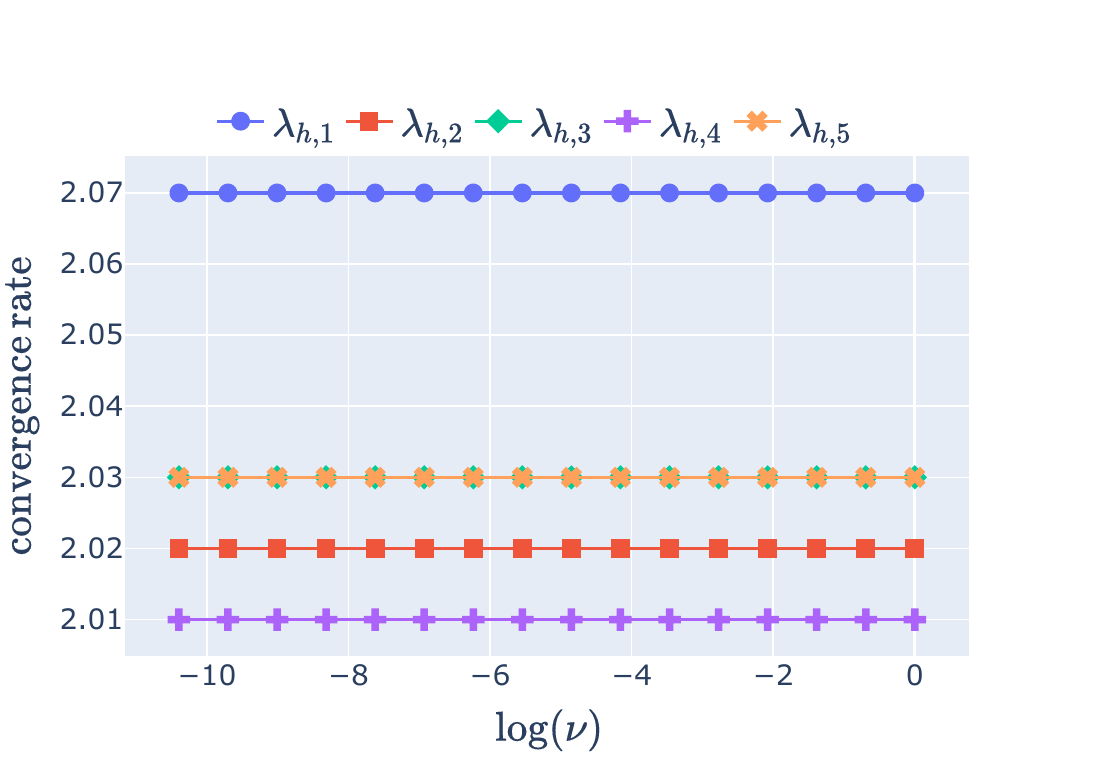}
	\end{minipage}\\
	\begin{minipage}{0.49\linewidth}
		\includegraphics[scale=0.35, trim=0cm 0cm 1.8cm 1.2cm, clip]{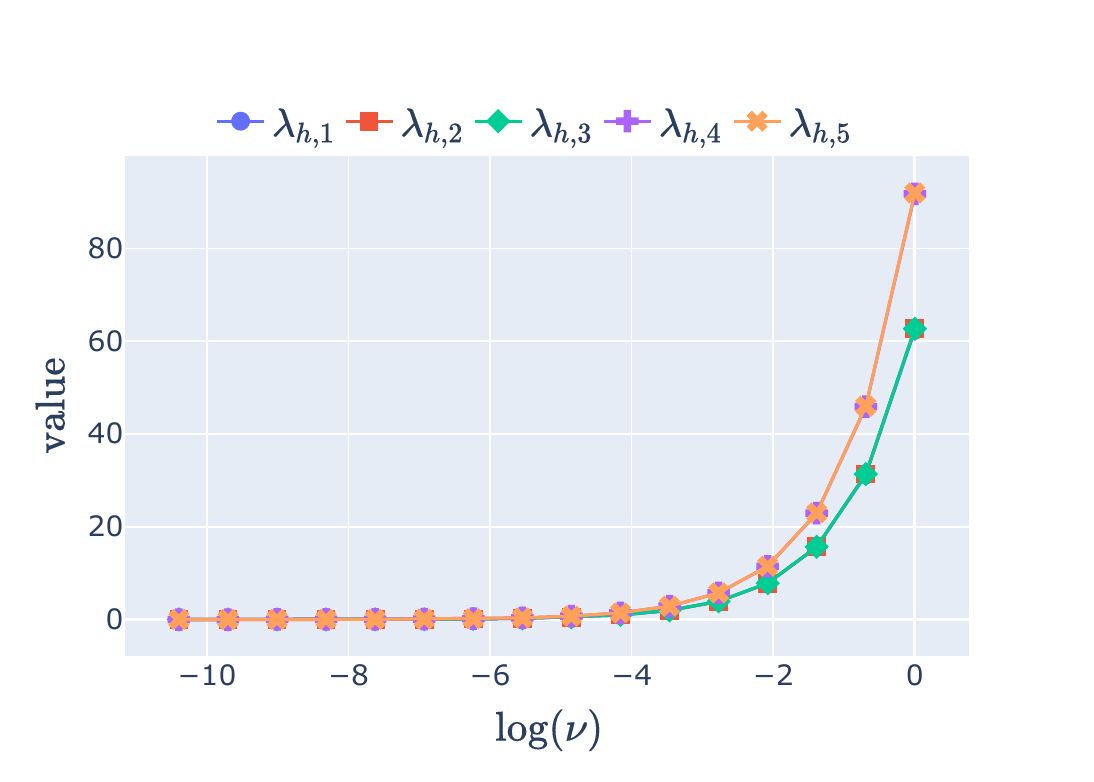}
	\end{minipage}
	\begin{minipage}{0.49\linewidth}
		\includegraphics[scale=0.35, trim=0cm 0cm 1.8cm 1.2cm, clip]{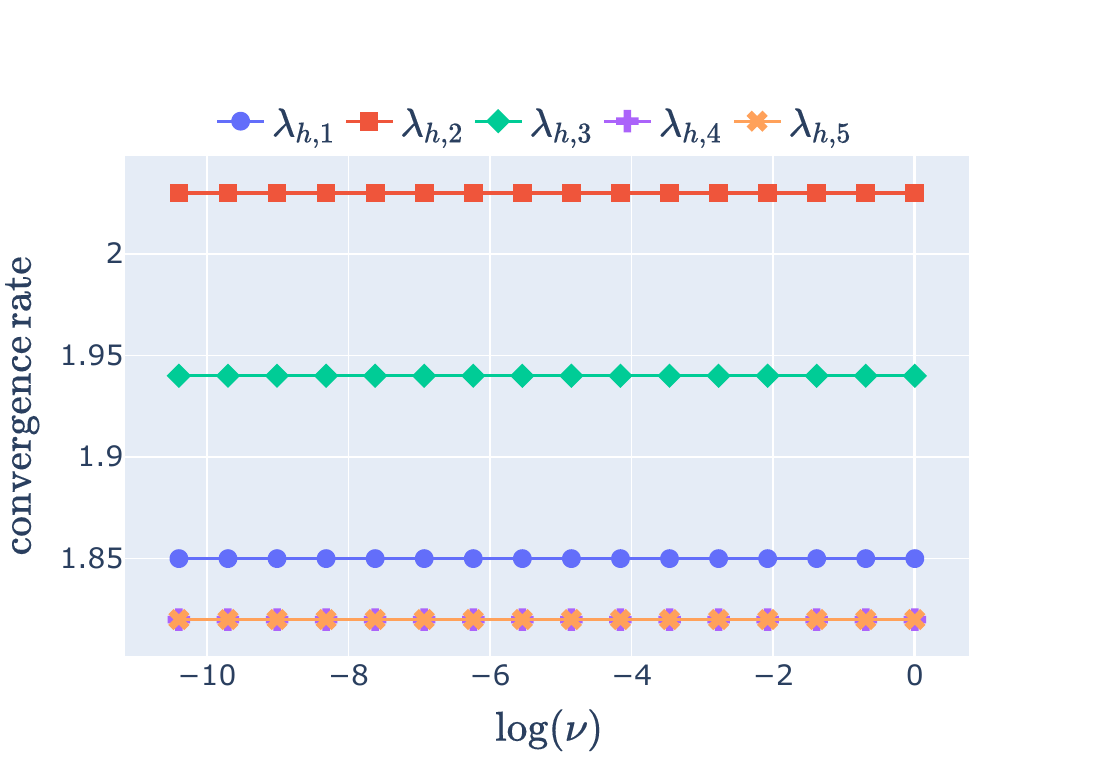}
	\end{minipage}\\
	\caption{Test \ref{subsec:robustness2D3D}. Behavior of the computed eigenvalues and the convergence rate with respect to the viscosity in the square domain (top row) and the unit cube domain (bottom row). Here, we consider $ \Vert\boldsymbol{\beta}\Vert_{\infty,\Omega}=\nu$. }
	\label{fig:robustnes_norm}
\end{figure}

\subsection{A 3D benchmark}\label{subsec:3D-benchmark}
We conclude the numerical section with a benchmark designed to analyze the computational cost of the proposed method. Specifically, we focus on the cube domain from Section \ref{subsec:cube-domain3D} with the convective velocity set as $\boldsymbol{\beta} = (1, 0, 0)^\texttt{t}$. The performance metrics we measure include assembly time, solve time, non-zero percentage, and floating point operations (flops). Assembly time refers to the time taken to assemble the PETSc matrix representing the left-hand side of the linear system. Solve time is measured as the time SLEPc takes to compute the four lowest eigenvalues. Non-zero percentage is calculated based on the matrix size and its sparsity pattern. Flops represent an estimate of the total double precision operations (DPOPS) executed by the MUMPS linear solver.

We compare the computations with respect to \cite{LEPE2024116959} using Taylor-Hood (TH) and mini element (MINI). For the present study, RT (resp. BDM) denotes the method using Raviart-Thomas (resp. Brezzi-Douglas-Marini) elements.

The results are reported in Figure \ref{fig:benchmark-3D}. We observe that the fastest assembly and solve time are measured when using RT elements, while the most expensive, up to a certain number of degrees of freedom, is the scheme that uses the TH pair. We note that solve time is almost on par with the MINI element. The non-zero percentage is smaller for the proposed scheme due to an overall higher $\texttt{dof}$. Finally, we note that BDM and TH are the schemes that take the most DPOPS per mesh size, while RT and MINI elements are the most efficient ones.
\begin{figure}[!hbt]
	\centering
	\begin{minipage}{0.49\linewidth}
		\includegraphics[scale=0.35, trim=0cm 0cm 1.8cm 1.2cm, clip]{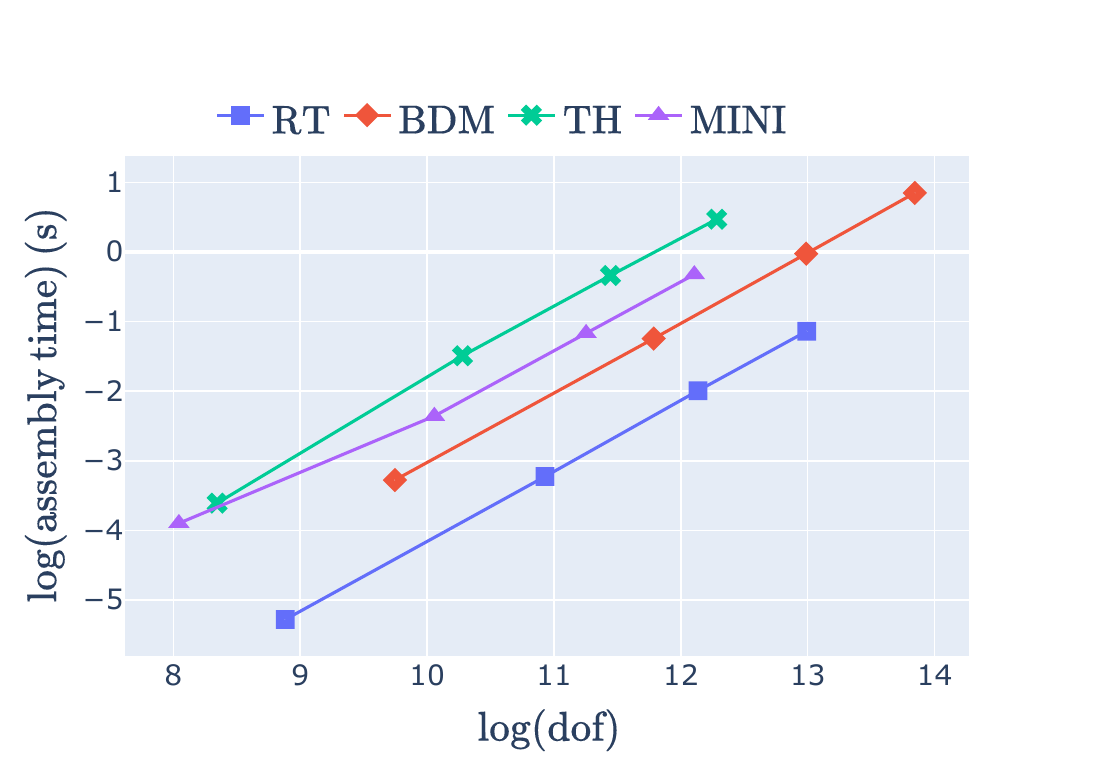}
	\end{minipage}
	\begin{minipage}{0.49\linewidth}
		\includegraphics[scale=0.35, trim=0cm 0cm 1.8cm 1.2cm, clip]{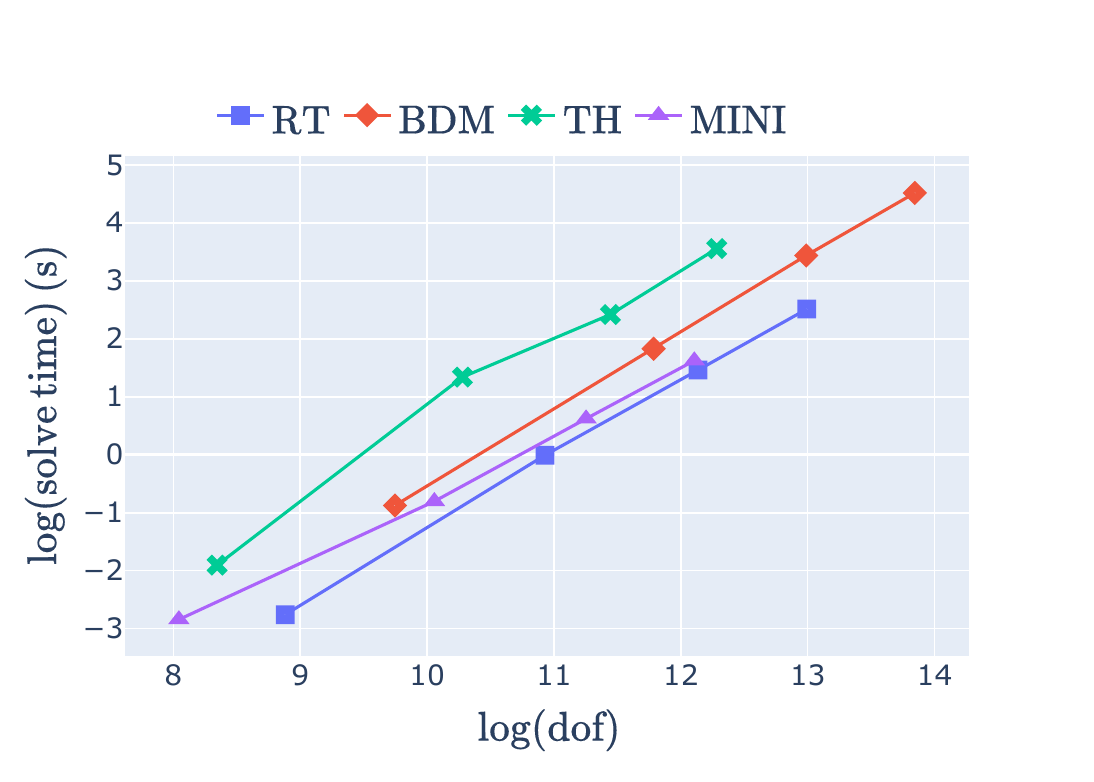}
	\end{minipage}\\
	\begin{minipage}{0.49\linewidth}
		\includegraphics[scale=0.35, trim=0cm 0cm 1.8cm 1.2cm, clip]{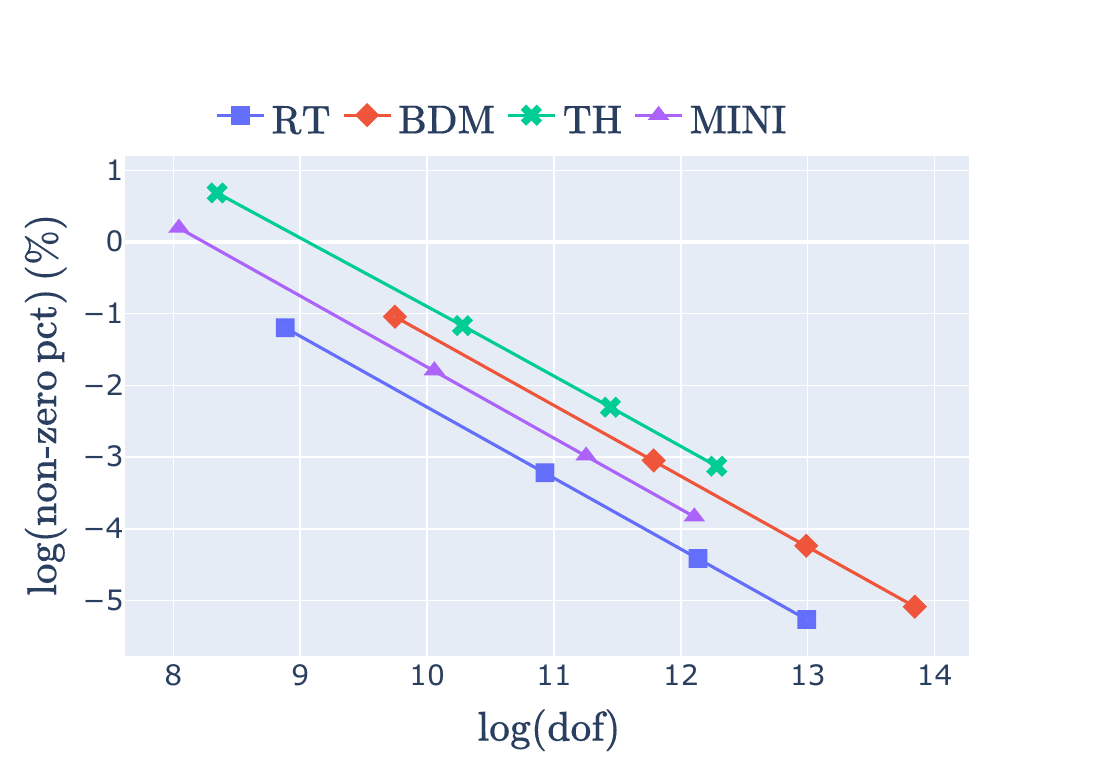}
	\end{minipage}
	\begin{minipage}{0.49\linewidth}
		\includegraphics[scale=0.35, trim=0cm 0cm 1.8cm 1.2cm, clip]{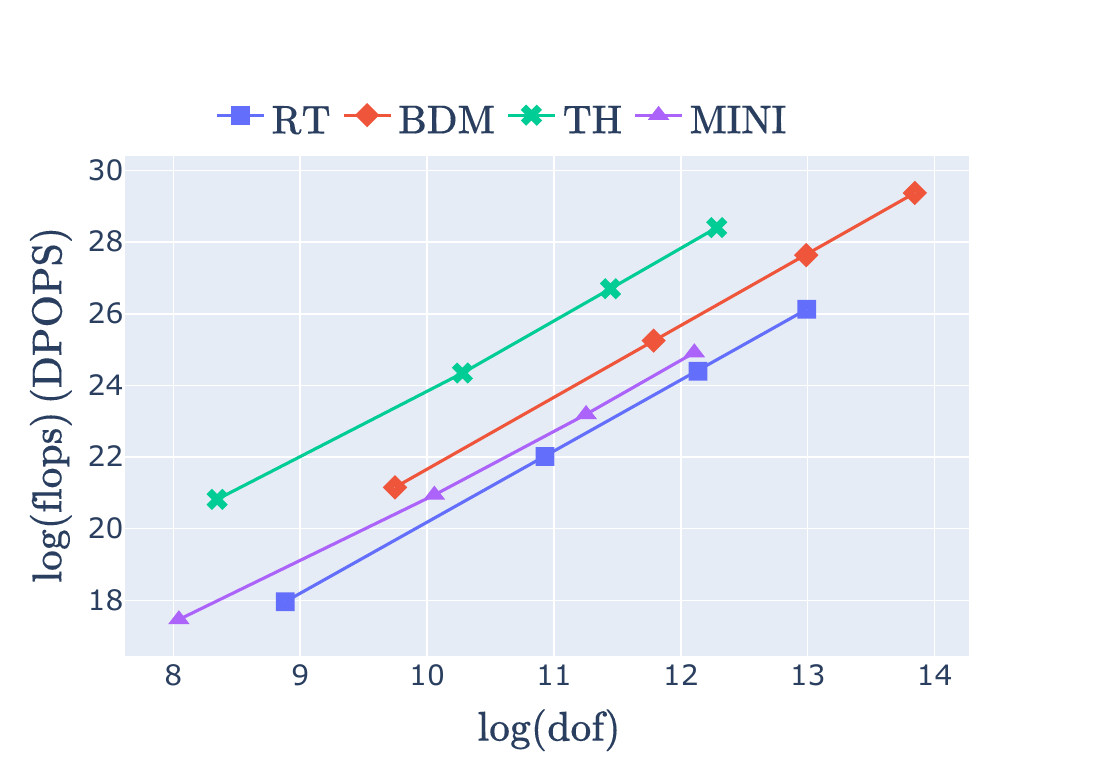}
	\end{minipage}\\
	\caption{Test \ref{subsec:3D-benchmark}. Comparison of computational cost related measures in the eigenvalue problem on the unit
		cube domain. We consider the field $\boldsymbol{\beta}=(1,0,0)^{\texttt{t}}$. }
	\label{fig:benchmark-3D}
\end{figure}

\bibliographystyle{siamplain}
\bibliography{oseen-eigenvalue}

\begin{thebibliography}{10}

\bibitem{MR1115235}
{\sc I.~Babu\v{s}ka and J.~Osborn}, {\em Handbook of numerical analysis. {V}ol.
  {II}},  (1991), pp.~x+928.
\newblock Finite element methods. Part 1.

\bibitem{barrata2023dolfinx}
{\sc I.~A. Barrata, J.~P. Dean, J.~S. Dokken, M.~HABERA, J.~HALE,
  C.~Richardson, M.~E. Rognes, M.~W. Scroggs, N.~Sime, and G.~N. Wells}, {\em
  {DOLFINx}: The next generation fenics problem solving environment},  (2023),
  \url{https://doi.org/10.5281/zenodo.10447666}.

\bibitem{MR2652780}
{\sc D.~Boffi}, {\em Finite element approximation of eigenvalue problems}, Acta
  Numer., 19 (2010), pp.~1--120,
  \url{https://doi.org/10.1017/S0962492910000012}.

\bibitem{MR3097958}
{\sc D.~Boffi, F.~Brezzi, and M.~Fortin}, {\em Mixed finite element methods and
  applications}, vol.~44 of Springer Series in Computational Mathematics,
  Springer, Heidelberg, 2013, \url{https://doi.org/10.1007/978-3-642-36519-5}.

\bibitem{MR799685}
{\sc F.~Brezzi, J.~Douglas, Jr., and L.~D. Marini}, {\em Two families of mixed
  finite elements for second order elliptic problems}, Numer. Math., 47 (1985),
  pp.~217--235, \url{https://doi.org/10.1007/BF01389710}.

\bibitem{MR1115205}
{\sc F.~Brezzi and M.~Fortin}, {\em Mixed and hybrid finite element methods},
  vol.~15 of Springer Series in Computational Mathematics, Springer-Verlag, New
  York, 1991, \url{https://doi.org/10.1007/978-1-4612-3172-1}.

\bibitem{MR3535625}
{\sc P.~Bringmann, C.~Carstensen, and C.~Merdon}, {\em Guaranteed velocity
  error control for the pseudostress approximation of the {S}tokes equations},
  Numer. Methods Partial Differential Equations, 32 (2016), pp.~1411--1432,
  \url{https://doi.org/10.1002/num.22056}.

\bibitem{cai2010}
{\sc Z.~Cai, C.~Tong, P.~S. Vassilevski, and C.~Wang}, {\em Mixed finite
  element methods for incompressible flow: stationary {S}tokes equations},
  Numer. Methods Partial Differential Equations, 26 (2010), pp.~957--978,
  \url{https://doi.org/10.1002/num.20467}.

\bibitem{MR4307023}
{\sc J.~Cama\~{n}o, C.~Garc\'{\i}a, and R.~Oyarz\'{u}a}, {\em Analysis of a
  momentum conservative mixed-{FEM} for the stationary {N}avier-{S}tokes
  problem}, Numer. Methods Partial Differential Equations, 37 (2021),
  pp.~2895--2923, \url{https://doi.org/10.1002/num.22789}.

\bibitem{MR4434148}
{\sc X.~Chen and Y.~Li}, {\em Superconvergent pseudostress-velocity finite
  element methods for the {O}seen equations}, J. Sci. Comput., 92 (2022),
  pp.~Paper No. 17, 27, \url{https://doi.org/10.1007/s10915-022-01856-1}.

\bibitem{MR4593742}
{\sc C.~I. Correa, G.~N. Gatica, and R.~Ruiz-Baier}, {\em New mixed finite
  element methods for the coupled {S}tokes and {P}oisson-{N}ernst-{P}lanck
  equations in {B}anach spaces}, ESAIM Math. Model. Numer. Anal., 57 (2023),
  pp.~1511--1551, \url{https://doi.org/10.1051/m2an/2023024}.

\bibitem{MR2050138}
{\sc A.~Ern and J.-L. Guermond}, {\em Theory and practice of finite elements},
  vol.~159 of Applied Mathematical Sciences, Springer-Verlag, New York, 2004,
  \url{https://doi.org/10.1007/978-1-4757-4355-5}.

\bibitem{MR3453481}
{\sc G.~N. Gatica, L.~F. Gatica, and F.~A. Sequeira}, {\em A priori and a
  posteriori error analyses of a pseudostress-based mixed formulation for
  linear elasticity}, Comput. Math. Appl., 71 (2016), pp.~585--614,
  \url{https://doi.org/10.1016/j.camwa.2015.12.009}.

\bibitem{MR4627698}
{\sc G.~N. Gatica, C.~Inzunza, and F.~A. Sequeira}, {\em New {B}anach
  spaces-based fully-mixed finite element methods for pseudostress-assisted
  diffusion problems}, Appl. Numer. Math., 193 (2023), pp.~148--178,
  \url{https://doi.org/10.1016/j.apnum.2023.07.017}.

\bibitem{geuzaine2009gmsh}
{\sc C.~Geuzaine and J.-F. Remacle}, {\em Gmsh: A 3-{D} finite element mesh
  generator with built-in pre-and post-processing facilities}, International
  journal for numerical methods in engineering, 79 (2009), pp.~1309--1331.

\bibitem{MR4789346}
{\sc Z.~Gharibi and M.~Dehghan}, {\em Analysis of {W}eak {G}alerkin {M}ixed
  {F}inite {E}lement {M}ethod {B}ased on the {V}elocity--{P}seudostress
  {F}ormulation for {N}avier--{S}tokes {E}quation on {P}olygonal {M}eshes}, J.
  Sci. Comput., 101 (2024), p.~Paper No. 12,
  \url{https://doi.org/10.1007/s10915-024-02651-w}.

\bibitem{hernandez2005slepc}
{\sc V.~Hernandez, J.~E. Roman, and V.~Vidal}, {\em {SLEPc}: A scalable and
  flexible toolkit for the solution of eigenvalue problems}, ACM Transactions
  on Mathematical Software (TOMS), 31 (2005), pp.~351--362.

\bibitem{MR2009375}
{\sc R.~Hiptmair}, {\em Finite elements in computational electromagnetism},
  Acta Numer., 11 (2002), pp.~237--339,
  \url{https://doi.org/10.1017/S0962492902000041}.

\bibitem{MR4570534}
{\sc D.~Inzunza, F.~Lepe, and G.~Rivera}, {\em Displacement-pseudostress
  formulation for the linear elasticity spectral problem}, Numer. Methods
  Partial Differential Equations, 39 (2023), pp.~1996--2017,
  \url{https://doi.org/10.1002/num.22955}.

\bibitem{John2016}
{\sc V.~John}, {\em Finite element methods for incompressible flow problems},
  vol.~51 of Springer Series in Computational Mathematics, Springer, Cham,
  2016, \url{https://doi.org/10.1007/978-3-319-45750-5}.

\bibitem{MR0203473}
{\sc T.~Kato}, {\em Perturbation theory for linear operators}, Springer-Verlag,
  Berlin, 1995.
\newblock Reprint of the 1980 edition.

\bibitem{MR4480275}
{\sc D.~Kim, E.-J. Park, and B.~Seo}, {\em Error analysis for the pseudostress
  formulation of unsteady {S}tokes problem}, Numer. Algorithms, 91 (2022),
  pp.~959--996, \url{https://doi.org/10.1007/s11075-022-01288-w}.

\bibitem{MR4077220}
{\sc F.~Lepe and D.~Mora}, {\em Symmetric and nonsymmetric discontinuous
  {G}alerkin methods for a pseudostress formulation of the {S}tokes spectral
  problem}, SIAM J. Sci. Comput., 42 (2020), pp.~A698--A722,
  \url{https://doi.org/10.1137/19M1259535}.

\bibitem{MR4430561}
{\sc F.~Lepe, G.~Rivera, and J.~Vellojin}, {\em Mixed methods for the
  velocity-pressure-pseudostress formulation of the {S}tokes eigenvalue
  problem}, SIAM J. Sci. Comput., 44 (2022), pp.~A1358--A1380,
  \url{https://doi.org/10.1137/21M1402959}.

\bibitem{LRVSISC}
{\sc F.~Lepe, G.~Rivera, and J.~Vellojin}, {\em Mixed methods for the
  velocity-pressure-pseudostress formulation of the {S}tokes eigenvalue
  problem}, SIAM Journal on Scientific Computing, 44 (2022), pp.~A1358--A1380,
  \url{https://doi.org/10.1137/21M1402959}.

\bibitem{MR4471016}
{\sc F.~Lepe, G.~Rivera, and J.~Vellojin}, {\em A posteriori analysis for a
  mixed {FEM} discretization of the linear elasticity spectral problem}, J.
  Sci. Comput., 93 (2022), pp.~Paper No. 10, 25,
  \url{https://doi.org/10.1007/s10915-022-01972-y}.

\bibitem{LEPE2024116959}
{\sc F.~Lepe, G.~Rivera, and J.~Vellojin}, {\em Finite element analysis of the
  {O}seen eigenvalue problem}, Computer Methods in Applied Mechanics and
  Engineering, 425 (2024), p.~116959,
  \url{https://doi.org/https://doi.org/10.1016/j.cma.2024.116959}.

\bibitem{MR4666864}
{\sc F.~Lepe and J.~Vellojin}, {\em A posteriori analysis for a mixed
  formulation of the {S}tokes spectral problem}, Calcolo, 60 (2023), pp.~Paper
  No. 52, 28.

\bibitem{MR3335223}
{\sc S.~Meddahi, D.~Mora, and R.~Rodr\'iguez}, {\em A finite element analysis
  of a pseudostress formulation for the {S}tokes eigenvalue problem}, IMA J.
  Numer. Anal., 35 (2015), pp.~749--766,
  \url{https://doi.org/10.1093/imanum/dru006}.

\bibitem{MR1115239}
{\sc J.~E. Roberts and J.-M. Thomas}, {\em Mixed and hybrid methods}, in
  Handbook of numerical analysis, {V}ol. {II}, Handb. Numer. Anal., II,
  North-Holland, Amsterdam, 1991, pp.~523--639.

\bibitem{scroggs2022basix}
{\sc M.~W. Scroggs, I.~A. Baratta, C.~N. Richardson, and G.~N. Wells}, {\em
  Basix: a runtime finite element basis evaluation library}, Journal of Open
  Source Software, 7 (2022), p.~3982.

\end{thebibliography}
\end{document}